\documentclass[10p]{amsart}
\usepackage[margin=1in]{geometry}
\usepackage{version,amsmath, amssymb, mathtools,color}
\usepackage{color}
\usepackage{amsthm, amsopn, amsfonts, xypic}
\usepackage[colorlinks=true]{hyperref}
\usepackage{mathrsfs}
\usepackage{slashed}
\hypersetup{urlcolor=blue, citecolor=blue}

\usepackage[color=yellow]{todonotes}

\let\arXiv\arxiv
\def\doi#1{ {\href{http://dx.doi.org/#1}
   {{\mdseries\ttfamily DOI}}}}

\newcommand{\beq}{\begin{equation}}\newcommand{\eq}{\end{equation}}

\renewcommand{\tt}{ \tilde t}

\newcommand{\pa}{{\partial}}
\newcommand{\tphi}{{\tilde{\phi}}}

\newcommand{\tchi}{{\tilde \chi}}

\newcommand{\lL}{{\underline L}}
\newcommand{\ts}{{\widetilde{s}}}
\newcommand{\te}{{\widetilde{e}}}
\newcommand{\tvarrho}{{\widetilde{\varrho}}}
\newcommand{\ttau}{{\widetilde{\tau}}}
\newcommand{\tgamma}{{\widetilde{\gamma}}}

\newcommand{\newsection}[1]
{\section{#1}\setcounter{theorem}{0} \setcounter{equation}{0}
\par\noindent}

\newtheorem{theorem}{Theorem}

\newtheorem{lemma}[theorem]{Lemma}
\newtheorem{corr}[theorem]{Corollary}
\newtheorem{prop}[theorem]{Proposition}

\newtheorem{remark}[theorem]{Remark}

\newcommand{\R}{{\mathbb R}}

\newcommand{\ang}{{\not\negmedspace\partial }}

\newcommand{\la}{\langle}
\newcommand{\ra}{\rangle}

\renewcommand{\div}{{\text{div}}}

\newcommand{\rs}{{r^*}}

\renewcommand{\S}{{\mathbb S}}
\newcommand{\M}{\mathcal M}

\newcommand{\tg}{{\tilde{g}}}

\newcommand{\tN}{{\tilde{N}}}

\renewcommand{\th}{{\tilde{h}}}
\newcommand{\uL}{\underline{L}}

\newcommand*{\bigone}[1]{\vcenter{\hbox{\scalebox{1.2}{\ensuremath#1}}}}
\newcommand*{\bigtwo}[1]{\vcenter{\hbox{\scalebox{1.4}{\ensuremath#1}}}}
\newcommand*{\bigthree}[1]{\vcenter{\hbox{\scalebox{1.6}{\ensuremath#1}}}}

\DeclareMathOperator{\tsum}{{\textstyle{\sum}}}

\newcommand{\weight}{{\Bigl(1-\frac{2M}{r}\Bigr)}}

\begin{document}

\title
{
 A local energy estimate for wave equations on metrics asymptotically close to Kerr
}

\author{Hans Lindblad}
\address{Department of Mathematics, Johns Hopkins University, Baltimore,
  MD  21218}
\email{lindblad@math.jhu.edu}
\author{Mihai Tohaneanu}
\address{Department of Mathematics, University of Kentucky, Lexington,
  KY  40506}
\email{mihai.tohaneanu@uky.edu}

\begin{abstract}
   In this article we prove a local energy estimate for the linear wave equation on metrics with slow decay to a Kerr metric with small angular momentum. As an application, we study the quasilinear wave equation $\Box_{g(u, t, x)} u = 0$
  where the metric $g(u, t, x)$ is close (and asymptotically equal)
  to a Kerr metric with small angular momentum $g(0,t,x)$. Under suitable assumptions on the metric
  coefficients, and assuming that the initial data for $u$ is small enough,
  we prove global existence and decay of the solution $u$.
 \end{abstract}

\mathtoolsset{showonlyrefs=true}

\maketitle
\emph{\emph{}}

\tableofcontents

\newsection{Introduction}

In this paper we consider energy estimates for the wave operator $\Box_g$ on the background of a metric $g$ that is close to
the Kerr black hole metric $g_K$ with small angular momentum $a$ expressed in the modified Boyer-Lindquist coordinates that are smooth over the event horizon, see Section \ref{sec:thecoordinates}. We show that if the metric decay slowly in time towards Kerr then we have local energy decay estimates. We also use our estimate to prove global existence of solutions to quasilinear wave equations close to Kerr, as well as pointwise decay for solutions to the linear problem.

 The paper is structured as follows. In Section 1 we discuss the estimate and present some heuristics of why it should hold. Section 2 contains a detailed proof of the estimate in the case of the Kerr metric. In Section 3 we prove some estimates for pseudodifferential operators with limited regularity. Section 4 contains our main linear estimate for perturbations of the Kerr metric. Section 5 deals with how to apply the estimate to obtain global well-posedness for the quasilinear problem following \cite{LT}. Section 6 applies the estimate to obtain  pointwise decay for solutions to the linear problem, following \cite{MTT}.

\subsection{Statement of the results and history}

\subsubsection{Assumptions on the metric}

We consider metrics that are perturbations of a Kerr metric $(g_K, M, a)$ with mass $M$ and angular momentum per unit mass $0<a\ll M$. In Boyer-Lindquist coordinates, the Kerr metric has a singularity at $r=0$ and two event horizons at $r_{\pm} = M\pm \sqrt{M^2-a^2}$. By a suitable change of coordinates, the metric can be extended past $r=r_{+}$ (see Section \ref{sec:thecoordinates} for details). We fix some $r_{-} < r_e < r_{+}$ and consider the manifold $\M=\{r>r_e\}$, which in particular includes the domain of outer communication.

An important role in our analysis will be played by the trapped set of null geodesics in the exterior region $r>r_+$. It is well-known that they are spatially contained in a small neighborhood of $r=3M$ of size $O(a)$. The most delicate part of our analysis will take place in this region.

We now define a metric $g$ on $\M$ that is a perturbation of $g_K$. We assume that
\begin{equation}\label{rdecayintro}
| h^{\alpha\beta}|+r|\partial h^{\alpha\beta}| \,\lesssim \epsilon r^{-\delta}, \qquad\text{where}\quad h=g-g_K,
\end{equation}
for some $\delta>0$, where $r=|x|$ is the radial coordinate.
This condition can be modified in the wave zone $|t-r|<t/2$ to take into account the weak null structure of the wave operator, see Section \ref{sec:metricconditions}, but let us leave this aside for now since here we focus on the analysis close to the photon sphere $|r-3M|<M/4$. The analysis in this region is quite delicate due to lack of decay due to trapped geodesics. Because of that we need to assume that the metric is asymptotically approaching
the Kerr black hole metric $g_K$ as time tends to infinity
\begin{align}
  |\partial h^{\alpha\beta}|+| h^{\alpha\beta} | &\leq \kappa_1(t)\lesssim \epsilon  \la t\ra^{-1/2} ,\qquad\text{when}\quad |r-3M|<M/4,\label{cpt1intro}\\
 |\partial_{t} h^{\alpha\beta}| &\leq \kappa_0(t)\lesssim \epsilon \la t\ra^{-1},\qquad\text{when}\quad |r-3M|<M/4. \label{cpt2intro}
\end{align}
In contrast to our previous result on Schwarzschild, see Section 4 of \cite{LT}, we also need smallness (but not decay) on higher order derivatives of $h$. This is due to the use of pseudodifferential operators, which generate errors involving such derivatives. We will assume that the H\"older norms satisfy
\begin{equation}\label{holderintro}
\| h^{\alpha\beta}(t,\cdot) \|_{C^{3, \delta}(I_{ps})} \lesssim \epsilon,\qquad\text{where}\quad I_{ps}=\{x;\, |r-3M|<M/4\}.
\end{equation}

\subsubsection{The Local energy decay estimate}

Our main result is the following:

\begin{theorem}\label{LEintro} Let $u$ solve the inhomogeneous linear wave equation
$ \Box_g u = F$ in $\M$,
where $g$ is a Lorentzian metric satisfying the conditions \eqref{rdecayintro},
\eqref{cpt1intro}, \eqref{cpt2intro} and \eqref{holderintro} or alternatively the conditions in section 4.
 Then for any $0\leq\tt_0 < \tt_1$
\begin{equation}\label{LEKintro}
\|\partial  u\|_{L^{\infty}[\tt_0, \tt_1] L^2}^2 + \|u\|_{LE_K^1[\tt_0, \tt_1]}^2
 \lesssim \|{\kappa}\pa u\|_{L^2_{ps}}^2+\|\partial u(\tt_0)\|_{L^2}^2 +B(F,u)_{[\tt_0, \tt_1]},
\end{equation}
where $\kappa^2=\kappa_0+\kappa_1^2$ and the implicit constant is independent of $\tt_0$, $\tt_1$, $\epsilon$.
Here $B(F,u)$ is a bilinear norm defined in section 4, see \eqref{bilform}, the $LE_K^1$ norm is defined in \eqref{leK},
and $L^2_{ps}=L^2[\tt_0, \tt_1] L^2(I_{ps})$.
\end{theorem}

The theorem is an improvement of the result for the Kerr metric by the second author and Tataru \cite{TT}.

The bilinear form $B(F,u)_{[\tt_0, \tt_1]}$ vanishes for a solution of the homogeneous equation $F=0$ and in general can be estimated in different ways in terms of the other norms in the estimate above multiplied with norms of $F$ depending on the application. In particular we have
\begin{equation}\label{bilformintro}
B(F,u)_{[\tt_1, \tt_2]}
\lesssim
 \|F\|_{L^1[\tt_1, \tt_2] L^2}\|\partial  u\|_{L^{\infty}[\tt_1, \tt_2] L^2}.
 \end{equation}
The local energy norm  $ \|u\|_{LE_K^1[\tt_0, \tt_1]}$ is expressed in terms of pseudo differential operator but in particular bounds the following local norms:
\begin{equation}\label{eq:lowerboundLEKnormintro}
\|(1-\chi)\pa u\|_{LE[\tt_0, \tt_1]}
+\|\pa_r  u\|_{LE[\tt_0, \tt_1]}+\|r^{-1} u\|_{LE[\tt_0, \tt_1]}
\lesssim \| u\|_{LE_K^1[\tt_0, \tt_1]},
\end{equation}
where
\begin{equation}
  \| u\|_{LE[\tt_0, \tt_1]} ={\sup}_{R\geq 0} \, \| \la r\ra^{-\frac12} u\|_{L^2 ([\tt_0, \tt_1] \times A_R)},\qquad A_R=\{x\in\Bbb R^3; \, R/2\leq |x|\leq 2R+1\},
\label{ledefintro}\end{equation}
and $\chi$ is a cutoff that selects a small neighborhood of the trapped set.

\begin{remark}
We are thus able to estimate both the energy and the spacetime norm of the solution globally. Moreover, whereas we spend most of the analysis in understanding high frequency behavior near trapping, our estimate includes low frequencies, since the $L^2L^2$ norm of $u$ is included in the local energy norm.
\end{remark}

\begin{remark}
Since one can not control the $L^2$ space time norm of all derivatives close to $I_{ps}$ in the right hand side of \eqref{LEKintro} the best we can do is to control it with the energy
$\|\partial  u\|_{L^{\infty}[\tt_0, \tt] L^2}$ and it follows from Gr\"onwall's lemma that
\begin{equation}\label{LEKintro2}
\|\partial  u\|_{L^{\infty}[\tt_0, \tt_1] L^2}^2 + \|u\|_{LE_K^1[\tt_0, \tt_1]}^2
 \lesssim e^{\,C{\bigone\int}_{\!\!\tt_0}^{\tt_1}\!\! \kappa(\tt)^2 d\tt}\Big(\|\partial u(\tt_0)\|_{L^2}^2 +\|F\|_{L^1[\tt_0, \tt_1] L^2}\Big).
\end{equation}

\end{remark}

\begin{remark} The result for exactly Kerr can be used  to prove that the norms in the left hand side of \eqref{LEKintro2} are bounded under much stronger decay assumptions for perturbations $h\sim \epsilon\, t^{-1-\epsilon}$.
The novelty of our estimate is that we require less decay and that
this is relevant for the applications. In general for nonlinear equations that only satisfy  the weak null condition one can not prove better than $u\sim \epsilon \, t^{-1+\epsilon}$ interior decay even using the most advanced energy methods and even without a black hole \cite{LR10,Lin1, K18}.  Additionally using the nonradial fundamental solution one can for certain equations satisfying the weak null condition prove asymptotics $u\sim \epsilon\, t^{-1}$, see \cite{L17,LS17,CKL19}, but these methods are unlikely to work in the presence of a black hole.
\end{remark}

\begin{remark} We remark that it is natural to assume additional decay for the
time derivative \eqref{cpt2intro} since this is likely to hold in the nonlinear applications. In fact one gets decay estimates from energy estimates for vector fields applied to the equation and the scaling vector field $S=t\pa_t +r\pa_r $
gives a good decay estimate for $\pa_t $ for bounded $r$ since the local energy
gives a good estimate for  $\partial_r$ and for $S$ applied to the equation.
\end{remark}

\begin{remark} As mentioned above one gets decay estimates from energy estimates for vector fields applied to the equation. The decay estimates that scales the same as the energy and local energy estimate are $ u\sim \epsilon \, t^{-1/2}$, so one can not hope to get anything better from just boundedness of the energy and local energy applied to scale invariant vector fields.
However, with $\kappa_1(t)=\epsilon \, t^{-1/2}$ in \eqref{cpt1intro} the estimate \eqref{LEKintro2} grows like $t^{C\epsilon}$ and we would hence not get back bounded energies.
This has to be the case for the highest order energies for the most vector fields applied to the solution. However, for lower order energies the term with $\kappa$ that is causing the growth in \eqref{LEKintro} can be control by the $L^2$ part of the local energy norm, the last term in the left of \eqref{eq:lowerboundLEKnormintro}, from higher order energies, since this
gives control also close to the photon sphere. This means that one can get bounds for lower order energies and hence get back the sharp decay $\epsilon\, t^{-1/2}$ decay. We have carried out this procedure for quasilinear equations close to black hole background in \cite{LT} and here in Section \ref{sec:quasilinear}.
\end{remark}

\begin{remark} In order to prove better than $u\sim \epsilon t^{-1/2} $ decay
in the presence of a black hole one has to use energies with positive $r$ weights, see \cite{DR09,Mos}. For equations satisfying the weak null condition we hope to prove interior decay $u\sim \epsilon \,t^{-1+\epsilon}$. As remarked above one can not expect to prove better.
\end{remark}

\subsubsection{The history}

In the Minkowski case one can prove the stronger estimate
\begin{equation}\label{LEM}
\|\partial  u\|_{L^{\infty}[\tt_0, \tt_1] L^2}^2 + \|u\|_{LE_m^1[\tt_0, \tt_1]}^2
 \lesssim \|\partial u(\tt_0)\|_{L^2}^2 +B(F,u)_{[\tt_0, \tt_1]},
\end{equation}
where
\begin{equation}
\|u\|_{LE_m^1[\tt_0, \tt_1]}=\|\pa u\|_{LE[\tt_0, \tt_1]}
+\|r^{-1} u\|_{LE[\tt_0, \tt_1]}.
\end{equation}

The first estimate of this kind was obtained by Morawetz for the Klein-Gordon equation \cite{M}. There are many similar results obtained in the case of
small perturbations of the Minkowski space-time; see, for example, \cite{KSS}, \cite{KPV}, \cite{SmSo}, \cite{St}, \cite{Strauss},
\cite{Al}, \cite{MS} and \cite{MT}. Even for large perturbations, in the absence of trapping, \eqref{LEM} still sometimes holds, see for instance \cite{BH}, \cite{MST}.
In the presence of trapping, \eqref{LEM} is known to fail, see \cite{Ral}, \cite{Sb}. On the other hand, if the trapping is weak enough, similar types of estimates with loss of regularity have been established, see for instance \cite{Chr}, \cite{NZ}, \cite{WZ}, \cite{BCMP}.

In the Schwarzschild case one can prove that
\begin{equation}\label{LES}
\|\partial  u\|_{L^{\infty}[\tt_0, \tt_1] L^2}^2 + \|u\|_{LE_S^1[\tt_0, \tt_1]}^2
 \lesssim \|\partial u(\tt_0)\|_{L^2}^2 +B(F,u)_{[\tt_0, \tt_1]},
\end{equation}
where
\begin{equation}
\|u\|_{LE_S^1[\tt_0, \tt_1]}=\|(r-3M)r^{-1}\pa u\|_{LE[\tt_0, \tt_1]}
+\|\pa_r  u\|_{LE[\tt_0, \tt_1]}+\|r^{-1} u\|_{LE[\tt_0, \tt_1]}.
\end{equation}
In the previous estimate there is no loss of derivatives, but the weights degenerate at the trapped set.

Local energy estimates were first proved in \cite{LS} for radially
symmetric Schr\"odinger equations on Schwarzschild backgrounds.  In
\cite{BS1, BS2, BSerrata}, those estimates are proved for the wave equation and general initial data.  The same authors, in
\cite{BS3,BS4}, have provided studies that give certain improved
estimates near the photon sphere $r=3M$.  Moreover, we note that
variants of these bounds have played an important role in the works
\cite{BSter} and \cite{DR1}, \cite{DR2} which prove analogues of the
Morawetz conformal estimates on Schwarzschild backgrounds. See also \cite{Schlue} and \cite{LM} for a similar result on higher dimensional Schwarzschild black holes.

The estimate \eqref{LES} with the precise weights at infinity was proved in \cite{MMTT}. The same estimate for perturbations decaying at an integrable rate near the trapped set was established in \cite{MTT}, while for perturbations satisfying weaker conditions akin to \eqref{cpt1intro} and \eqref{cpt2intro} it was proved in \cite{LT}.

The estimate for Kerr with small angular momentum was proven in \cite{TT} (see also \cite{AB} and \cite{DaRoNotes} for related works), for large angular momentum $|a|<M$ in \cite{DRSR}, and for extremal Kerr $|a|=M$ in \cite{Ar}.

There is also a rich literature of high frequency estimates and resonance distribution near weakly trapped sets in a variety of settings, see for instance \cite{GS}, \cite{CVP}, \cite{Chr}, \cite{NZ}, \cite{WZ}. For the Schwarzschild (and de Sitter-Schwarzschild) metric we refer the interested reader to \cite{BH2}, \cite{SZ}, and for the Kerr metric see \cite{DZ}, \cite{Dy1}, \cite{Dy2}. See also the recent paper of Hintz \cite{Hin}, who proves a microlocal estimate near the trapped set for a variety of spacetimes, including asymptotically converging perturbations of Kerr.

\subsubsection{Applications}

In the last two sections we apply our main estimate to two different problems.

The main application of our estimate comes in showing global existence of solutions to quasilinear wave equations $\Box_{g(u, t, x)} u = 0$ with small initial data,
  where the metric $g(u, t, x)$ is close (and asymptotically equal) to a Kerr metric with small angular momentum $g(0,t,x)$. This can be seen as a toy model for Einstein's Equations close to Kerr metrics, and extends previous results of the first author \cite{Lin1} in the Minkowski case and of both authors \cite{LT} in the Schwarzschild case; see also for \cite{Yang1}, \cite{Yang2} for similar results for time dependent metrics close to Minkowski, and \cite{HV} for the asymptotically Kerr- de Sitter case. Once the local energy for perturbations of the Kerr metric, Theorem~\ref{LE}, are established, global existence can be established by using the same methods as in \cite{LT}. We refer the reader to Section 5 for a sketch of the argument, and on how to deal with the minor differences in the proof.

For our second application, we establish sharp pointwise decay estimates for solutions to the linear wave equation on perturbations of Kerr metrics with small angular momentum. In the case of the Schwarzschild metric, the solution to the wave equation was conjectured to decay at the rate of $\tt^{-3}$ on a compact region by Price \cite{Pri}, and this rate of decay was shown to hold for a variety of spacetimes, including Schwarzschild and Kerr with small $a$, see \cite{DSS}, \cite{Tat}, \cite{MTT}. It is by now well understood that once local energy decay is established in a compact region on an asymptotically flat region, one can obtain pointwise decay rates that are related to how fast the metric coefficients decay to the Minkowski metric; see, for example, \cite{Tat}, \cite{MTT}, \cite{OS}, \cite{Mos}, \cite{AAG1}, \cite{AAG2}, \cite{Mor}, \cite{Hin2}.
In this paper, we improve on the result of the second author and collaborators \cite{MTT}, which proves $\tt^{-3}$ decay on compact regions assuming certain conditions on the metric, see Section~\ref{sec:ptwse}.

\subsection{The heuristics} Below we explain why the result should be true, without going into all the technical details.
\subsubsection{The Schwarzschild case}
The estimate \eqref{LES} with the precise weights at infinity was proved in \cite{MMTT} .
The idea of the proof is that there exist a smooth vector field
\[
 \tilde X=C\pa_{\tt} + b(r)\bigthree(1-\frac{3M}r\bigthree) \partial_r + c(r)\partial_{\tt},
\]
and a function $q$ so that
\begin{equation}
\nabla^\alpha P_\alpha[g,\tilde X,q,m] =  \Box_g u \big(\tilde Xu +
 q u\big)+ Q[g,\tilde X,q,m].
\label{divintro}\end{equation}
Here $Q$ and $P$ are quadratic forms in $u$ and $\pa u$ with coefficients depending on $g$ so that
$Q[g_S,\tilde X,\tilde q,m]$
is positive definite, and
\begin{equation}
Q[g_S,\tilde X,\tilde q,m] \gtrsim r^{-2} |\partial_r u|^2 + \big(1-\frac{3M}r
\big)^2 (r^{-2} |\partial_t u|^2 + r^{-1}|\ang u|^2) + r^{-4}
u^2.
\label{posS}\end{equation}
 Let $T$ be the future pointing unit normal on the hyper surface $\tilde{t}=\tilde{t}_1 $ or
the normal pointing towards the black hole at the hyper surface $r=r_e$. We have
\begin{equation}\label{eq:Pboundary}
- P_\alpha[g,X,q,m] T^\alpha \gtrsim |\pa u|^2.
\end{equation}
Integrating \eqref{divintro} over the domain $[\tilde{t}_0,\tilde{t}_1]\times \{r>r_e\}$
gives an interior integral with \eqref{posS} and  boundary integrals with
\eqref{eq:Pboundary}:
\beq
  \int_{\tt_0}^{\tt_1}\!\!\!\int{\Box}_g u\cdot{(\tilde{X}u+q u)}\, \sqrt{|g|}dx dt
=\!\int_{\tt_0}^{\tt_1}\!\!\!\int Q[g,\tilde X,\tilde q,m]
\, \sqrt{|g|} dx dt+\text{BDR}\Big|_{t=\tt_0}^{t=\tt_1}+\text{BDR}\Big|_{r=r_e}.
\eq
\subsubsection{Perturbations of Schwarschild}\label{sec:perturbschwarzschild}
 Let us now consider perturbations $g=g_S+h$. Schematically near the trapped set one has, for any smooth $X$ and $q$ modulo boundary terms
\beq
\int_{\tt_0}^{\tt_1}\!\!\!\int  \Box_h u\, (Xu+qu) dxdt \lesssim\int_{\tt_0}^{\tt_1}\!\!\!\int (|h|+|\pa h|)(|\pa u|^2 + |u|^2) dxdt + BDR,
\eq
so, if for instance,
$
|h| + |\pa h| \lesssim \epsilon \tt^{-1-\delta},
$
the error terms can be absorbed by Gr\"onwall's inequality.

A more careful computation (see \cite{LT}) shows that for the specific $\tilde X$ and $q$ above, we have
\begin{equation}
\int_{\tt_0}^{\tt_1}\!\!\!\int  \Box_h u\, (\tilde Xu+qu) dxdt \lesssim\int_{\tt_0}^{\tt_1}\!\!\!\int  |\pa_{\tt} h| \, |\pa u|^2 + (|h|+|\pa h|)\,\left(|r-3M|  |\pa u|^2+ |\pa_r u| |\pa u|\,
 + |u|^2\right) dx dt +BDR.
\end{equation}
Thus if one only assumes
$
|\pa_{\tt} h| \lesssim \epsilon \tt^{-1-\delta}$
and $|h| + |\pa h| \lesssim \epsilon \tt^{-1/2-\delta}$,
the error terms can be absorbed by using Cauchy Schwarz, Gr\"onwall and the local energy norm.

\subsubsection{Trapped geodesics and a pseudo differential operator vanishing on the trapped set}
In order to define pseudodifferential operators we use  cartesian coordinates $x^j$ and  dual variables in the cotangent bundle $\xi_j$. On the other hand, it is easier to understand trapping in spherical coordinates $(t, r, \phi, \theta)$. Let $\tau$, $\xi_r$, $\Phi$, $\Theta$ be
the dual variables in the cotangent bundle corresponding to $t$, $r$, $\phi$ and $\theta$; they are related to $x_j$ and $\xi_j$ by
\beq\label{eq:cotantrans}
\xi_r = \frac{\pa x^k}{\pa r}\xi_k, \quad \Phi = \frac{\pa x^k}{\pa \phi}\xi_k, \quad \Theta = \frac{\pa x^k}{\pa \theta}\xi_k,\quad \tau=\xi_0.
\eq
Let
\[
p_K( r, \theta,\tau, \xi_r, \Phi,\Theta)=g^{\alpha\beta}_K\xi_\alpha\xi_\beta
=g_K^{tt}\tau^2+2g_K^{t\phi}\tau\Phi+g_K^{\phi\phi}\Phi^2
+g_K^{rr}\xi_r^2 +g_K^{\theta\theta}\Theta^2,
\]
 be the principal symbol of $\Box_{g_K} $.
 On any null geodesic one has
\begin{equation}\label{Hamintro}
 p_K ( r, \theta,\tau, \xi_r, \Phi,\Theta) =0.
\end{equation}
The Hamilton flow equations also give us, in particular, that
\begin{equation}\label{rdotintro}
 \dot{r}=-\frac{\partial p_K}{\partial \xi_r}=-2\rho^{-2} {\Delta} \xi_r,
 \qquad\text{where}\quad {\Delta}=r^2-2Mr +a^2,
\end{equation}
\begin{equation}\label{xidotaintro}
  \dot{\xi_r}=\frac{\partial p_K}{\partial r}=
  - 2\widehat{R}_a(r,\Phi/\tau)\rho^{-2} \tau^2 +  \partial_r(\rho^{-2}) p_K + 2(r-M)\rho^{-2}\xi_r^2 ,
\end{equation}
where
\[
\widehat{R}_a(r,\Phi/\tau) =\widetilde{R}_a(r,\Phi/\tau)\big(r-r_a(\Phi/\tau)\big),\qquad
\widetilde{R}_0=(1-\tfrac{2M}{r})^{-2},\quad r_0=3M.
\]
As it turn out there are trapped null geodesics
satisfying
\begin{equation}
r-r_a(\Phi/\tau)=0, \qquad \xi_r=0,\qquad \Phi=\Phi_0,\qquad \tau=\tau_0.
\end{equation}
 Heuristically we will replace the vector field from Schwarzschild with a pseudo differential operator of order $1$ with symbol
\begin{equation}
\widetilde{s}(x,\xi)=i\big(r-r_a(\Phi/\tau)\big)\xi_r.
\end{equation}

\subsubsection{Positive commutator estimates}
Let $P_g=D_{\alpha} g^{\alpha\beta} D_{\beta}$ be a symmetric operator. Suppose $S$ is a skew-adjoint and $E$ is a self-adjoint operator; in what follows, $S$ will be taken to be a pseudo differential operator of order $1$ generalizing the vector field $X$ used in Schwarzschild and $E$ will be a pseudo differential operator of order $0$ generalizing the function $q$ in Schwarzschild.
Integrating by parts and using that $S$ is skew adjoint and $E$ self adjoint  we obtain
modulo boundary terms
\beq
\label{eq:Qhatdef0intro}
 \Re \int_{\tt_0}^{\tt_1}\!\!\!\int P_g u\cdot \overline{(Su+E u)}\, dx dt
=\!\int_{\tt_0}^{\tt_1}\!\!\!\int \widehat{T[P_g]} u\cdot\overline{u}\, \ dx dt+\text{BDR}\Big|_{t=\tt_0}^{t=\tt_1},
\eq
where
\beq
\widehat{T[P_g]}\!=\frac12\big([P_{g},S] + P_{g} E + E\,P_{g}\big).
\eq
We now apply this to
\[
P_g = \widehat{\Box}_g=\sqrt{|g|}\,\Box_g={\pa}_\alpha (\widehat{g}^{\alpha\beta} {\pa}_\beta),
\]
where $\widehat{g}^{\alpha\beta}=\sqrt{|g|}\, g^{\alpha\beta}$, is
also symmetric with respect to $dx dt$. In this case, we also use the abbreviated notation $\widehat{T}_g = \widehat{T[P_g]}$.

$ \widehat{T}_g$ is an operator of order $2$ and after further integration by parts
will give a quadratic form. The principal symbol of the commutator $[\widehat{\Box}_g,S]$ is given by the Poisson bracket
\begin{equation}\label{eq:poisson}
\{\widehat{p},s\}=\frac{\partial \widehat{p}}{\partial{\xi_\alpha}} \, \frac{\partial s}{\partial{x^\alpha}}-\frac{\partial \widehat{p}}{\partial{x^\alpha}}\, \frac{\partial s}{\partial{\xi_\alpha}},\qquad\text{where}\quad
\widehat{p}=\sqrt{|g|}\,p .
\end{equation}

Let us now look at the special case of the Kerr metric. The Poisson bracket is invariant under changes of coordinates and takes a particularly simple form in spherical coordinates due to the form of $p_K$ and $\widetilde{s}$:
\begin{equation}
\{\widehat{p}_K,\widetilde{s}\}=\frac{\partial \widehat{p}_K}{\partial \xi_r} \frac{\partial \widetilde{s}}{\partial r}
-\frac{\partial \widehat{p}_K}{\partial r} \frac{\partial \widetilde{s}}{\partial \xi_r}.\label{eq:poissonkerr}
\end{equation}
 We get
\begin{equation}
\{\rho^2 p_K,\widetilde{s}\}=2 \widetilde{R}_a(r,\Phi/\tau)\big(r-r_a(\Phi/\tau)\big)^2 \tau^2+2\bigtwo(\triangle  -(r-M) \big(r-r_a(r,\Phi/\tau)\big)\bigtwo)\xi_r^2 ,
\end{equation}
and hence
\begin{equation}
\{\rho^2 p_K,\widetilde{s}\}\geq \widetilde{\mu}_0^2 + \widetilde{\mu}_r^2,\qquad \text{where} \quad
 \widetilde{\mu}_0=\big(r-r_a(\Phi/\tau)\big)\tau/4,\quad  \widetilde{\mu}_r= 2M\xi_r  ,
\end{equation}
when $|r-3M|<M/4$ and $a$ is sufficiently small.

Hence control of the commutator
would, after integrating one of the factors by parts, give us control of the space time integrals
$\| \widetilde{\mu}_0(x,D) u\|_{L^2_{ps}}$ and $\| \widetilde{\mu}_r(x,D) u\|_{L^2_{ps}}$, modulo lower order terms.

Similarly with the principal symbol of $E$
\begin{equation}
\widetilde{e}= c\big(r-r_a(r,\Phi/\tau)\big)^2,\qquad c>0 .
\end{equation}
The term $E\,\Box_{g_K} +\Box_{g_K}E$ has principal symbol
\begin{equation}
2 c\big(r-r_a(r,\Phi/\tau)\big)^2 p_K,
\end{equation}
which given that we have control of $ \widetilde{\mu}_0$ gives us control of the additional derivatives $ \widetilde{\mu}_{i}(x,\xi)=\big(r-r_a(r,\Phi/\tau)\big)\xi_i$.

Following \cite{TT}, to localize the estimates in time we will replace $\widetilde{s}$ by an operator that is a differential operator in time and pseudo differential in space only. In this case the proof is quite a bit more involved, and the above argument is only heuristic.

\subsubsection{Kerr with small angular momentum $a$}\label{sec:smalla}
 When using Poisson bracket to calculate the leading order term in the commutator of pseudo differential operators as described above we produce lower order errors. However the pseudo differential operators can be constructed so that they are equal to the differential operators used in the Schwarzschild case plus a pseudo differential part of size $O(a)$.
Therefore also the lower order errors are bounded by a constant times $a$  multiplied by the space time $L^2$ norm of the function itself.
Moreover since we control the space time $L^2$ norm of the function itself in the Schwarzschild case we will be able to absorb these errors and do so in the Kerr case for sufficiently small $a$.

There are additional errors coming from cutoffs that we have to multiply the pseudo differential corrections with to localize them close to the trapped set. These are also of size $a$ but multiplied by the derivative of the function. They can be controlled because they are supported away from the trapped set, where the $LE_K^1$ norm does not degenerate.

\subsubsection{Perturbations of Kerr}
We now consider symmetric operators $P_{\widehat h}=D_{\alpha} h^{\alpha\beta} D_{\beta}$ with $h$ satisfying \eqref{rdecayintro}-\eqref{holderintro}.  Integration by parts in \eqref{eq:Qhatdef0intro} gives
modulo different boundary terms
\beq
 \Re \int_{\tt_0}^{\tt_1}\!\!\!\int P_{\widehat h} u\cdot \overline{(Su+E u)}\, dx dt
=\!\int_{\tt_0}^{\tt_1}\!\!\!\int \widehat{Q}[P_{\widehat h},S,E]\, \ dx dt+\text{BDR}\Big|_{t=\tt_0}^{t=\tt_1},
\label{eq:Qhatdef0intro2}
\eq
where
\beq\label{bfintro}
\widehat{Q}[P_{\widehat h},S,E]\!= \Re \Bigl(h^{\alpha\beta}[D_{\beta}, S]u \cdot \overline{D_{\alpha} u} + \frac12 [ h^{\alpha\beta}\!, S] D_{\beta} u\cdot \overline{D_{\alpha} u} +  h^{\alpha\beta} E D_{\beta} u\cdot \overline{D_{\alpha} u} + \frac12  h^{\alpha\beta} [D_{\beta},E] u\cdot \overline{D_{\alpha}u} \Bigr).
\eq
We can now consider $\widetilde{S}$ and $\widetilde{E}$ with symbols $\widetilde{s}$ and $\widetilde{e}$ defined above.
The principal symbol of $[D_\beta, \widetilde{S}]$ is given by the Poisson bracket
$\bigtwo\{ \xi_\beta, \big(r\!-\!r_a(\Phi/\tau)\big)\xi_r\bigtwo\}=\omega_\beta \xi_r$.  Therefore the first term can be estimated by
$|h| \, |\pa_r u|\, |\pa u|$ plus lower order. Similarly the principal symbol of
$\big[h^{\alpha\beta}\!,\widetilde{S}\big]$ given by
$\bigtwo\{h^{\alpha\beta}\! , \big(r\!-\!r_a(\Phi/\tau)\big)\xi_r\bigtwo\}=\pa_{x^\gamma} h^{\alpha\beta}\pa_{\xi^\gamma} \big(\big(r\!-\!r_a(\Phi/\tau)\big)\xi_r\big)$,
has a factor of either a factor of $\xi_r$ or $r-r_a(\Phi/\tau)$ multiplied by $\pa h$. The third term also has at least one factor of $r-r_a(\Phi/\tau)$ and the last is lower order.

Again the above is only heuristics since the actual proof will use operators that are only differential operators in time and pseudo differential in space.
Moreover one can not just use the standard pseudo differential calculus for the terms involving $h$ since we require finite regularity of $h$.

\subsection{The proof} In section 2 we start by giving a more streamlined overview of the
proof from \cite{TT} in the Kerr case described below in sections \ref{sec:schwarz}-\ref{sec:operators}. Then below we continue to describe our proof
for perturbations in sections \ref{sec:vansihing}-\ref{sec:timederivatives}

\subsubsection{The Schwarzschild multiplier written as a pseudo differential operators}
\label{sec:schwarz} The microlocal analysis close the trapped set alone does not give us an estimate for the $L^2$ norm of the function itself but rather is assuming this to control lower terms. This analysis therefore has to be combined with the
global estimates already done for Schwarzschild. Since we are assuming that the angular momentum $a$ is small, Kerr can be considered as a small perturbation of
Schwarzschild away from the trapped set where the norms are not degenerate.   However close to the trapped set it is not a small perturbation and has to be dealt with using microlocal analysis.

We must therefore first write the Schwarzschild multipliers as pseudo differential operators close to the trapped set in order to write the operators in the Kerr case as a perturbations of size $a$ supported close to the trapped set of the Schwarzschild case. With respect to the Weyl quantization, any vector field $X=X^j \pa_j$ can be written as
\[
X= i\left((X^j \xi_j)^w - (\pa_j X^j)/2\right)
\]

In particular, we see that
\[
\frac1{i}b(r)\bigthree(1-\frac{3M}r\bigthree) \pa_r = \left(b(r)\bigthree(1-\frac{3M}r\bigthree) \xi_r \right)^w - \frac{1}{2r^2}\pa_r \left(r^2 b(r)\bigthree(1-\frac{3M}r\bigthree)\right)
\]
and thus near the trapped set we can express the multiplier in the Weyl quantization as
\[
b(r)\bigthree(1-\frac{3M}r\bigthree) \pa_r + q(r) = i\left(b(r)\bigthree(1-\frac{3M}r\bigthree) \xi_r \right)^w + q_S^w, \quad q_S = q - \frac{1}{2r^2}\pa_r \left(r^2 b(r)\bigthree(1-\frac{3M}r\bigthree)\right)
\]

\subsubsection{Construction of the space operators and the local energy norm}\label{sec:operators}
 One could now try to define a spacetime multiplier by quantizing a (multiple of) $\varrho_a=r-r_a(\Phi/\tau)$. Whereas this works well at the symbol level, it has the disadvantage of having a complicated dependence on $\tau$, which makes it inconvenient for energy estimates on constant time slices.

 From now on, we will use $\, \widetilde{} \,$ on top of a symbol to denote homogenous symbols.

We factor
 \begin{equation}\label{tauidef}
p_K = g_K^{tt}(\tau -\ttau_1) (\tau -\ttau_2),
\end{equation}
where $\ttau_i=\ttau_i( r, \phi, \theta, \xi_r, \Phi,\Theta)$,are real distinct smooth $1$-homogeneous symbols with respect to the space Fourier variables.

On the cone $\tau = \ttau_i$ the symbol $r-r_a(\Phi/\tau)$ equals
\begin{equation}\label{cidef}
\tvarrho_i ( r, \phi, \xi_r, \Phi,\Theta) = r- r_i( r, \phi, \theta, \xi_r, \Phi,\Theta),\qquad\text{where} \quad r_i=r_a(\Phi/\tau_i) , \qquad i = 1,2.
\end{equation}

Note that when $a=0$ then ${b(r)}r^{-1}\widetilde{s} = b(r)\bigthree(1-\frac{3M}r\bigthree) \xi_r$, i.e. the symbol of the vector field used in the Schwarzschild case. We can thus write ${b(r)}r^{-1}\widetilde{s}$ as a linear function of $\tau$ plus a smooth function of $\tau$ times $(\tau-\ttau_1)(\tau-\ttau_2)$:
\begin{equation}\label{XKdefintro}
b(r)\bigthree(1-\frac{r_a(\Phi/\tau)}r\bigthree) \xi_r = b(r)\bigthree(1-\frac{3M}r\bigthree) \xi_r  +  \ts_1(r,\xi_r, \theta,\Theta,\Phi)+
\ts_0(r,\xi, \theta,\Theta,\Phi) \tau+ h(\tau, r,\xi, \theta,\Theta,\Phi) p_K,
\end{equation}
with $\ts_1 \in aS^1_{hom}$, $\ts_0 \in aS^0_{hom}$, and $h$ a homogeneous symbol of spacetime.

We now define
$$
\frac1{i} \ts_K=b(r)\bigthree(1-\frac{3M}r\bigthree) \xi_r +\ts_1 +\ts_0 \tau .
$$
It is possible to pick $\te_K$ so that $\te_K - q_S = \te_0 + \te_{-1}\tau \in aS_{hom}^0 + a\tau S_{hom}^{-1}$, and
\[
\frac{1}{2i}\{\widehat {p}_K, \ts_K\} + \te_K \widehat p_K = \sum_{j=1}^8 \mu_j^2.
\]
Moreover, the symbols $\tvarrho_1(\tau-\ttau_2)$,
$\tvarrho_2(\tau-\ttau_1)$ and $\xi_r$ can be written as linear combinations of $\mu_j$ with coefficients in $S^0$.  See Lemma~\ref{TTmain} for more details.

In order to use smooth symbols, we define $\varrho_i$, $\tau_i$, $s_i$ and $e_i$ by multiplying $\tvarrho_i$, $\ttau_i$ and $\ts_i$ by a suitable cutoff in frequency that removes the singularity at $0$, and define
\[
\frac1{i} s_K=b(r)\bigthree(1-\frac{3M}r\bigthree) \xi_r +s_1 +s_0 \tau, \quad e_K = q_S + e_0 + e_{-1}\tau
\]

Given \eqref{eq:Qhatdef0intro} we can estimate $\int |M_j u|^2$ for suitably defined operators with symbol $\mu_j$, which in turn will control the local energy norm defined so that
\beq\label{eq:LElocalbounds}
\begin{split}
\| \varrho_2(D, x) \chi (D_t-\tau_1(D, x)\chi u\|_{L^2[\tt_0, \tt_1]L^2_{ps}}^2&+\| \varrho_1(D, x)\chi (D_t-\ttau_2(D,x)) \chi u\|_{L^2[\tt_0, \tt_1]L^2_{ps}}^2 \\ &
+\| D_r u\|_{L^2[\tt_0, \tt_1]L^2_{ps}}^2
\lesssim\|u\|_{LE_K^1[\tt_0, \tt_1]}^2.
\end{split}\eq
for suitable cutoffs $\chi$ supported near the trapped set. The $L^2$ norm can be absorbed in the positive $L^2$ term from the Schwarzschild case for $|a|\ll M$.

\subsubsection{Vanishing of symbols on the trapped set and estimating the perturbation}
\label{sec:vansihing}
The crucial observation that allows us to estimate the error term is the fact that the symbols $\ts_K$, $\pa \ts_K$ and $\te_K$ vanish on the trapped set. More precisely, we can explicitly compute
\begin{equation}\label{sKdecintro}
\frac{1}{i} \ts_K = \frac{b(r)}{r} \, \frac{\tvarrho_1(\tau-\ttau_2)-\tvarrho_2(\tau-\ttau_1)}{\ttau_1-\ttau_2} \,\, \xi_r,
\end{equation}
which implies that exist homogeneous symbols $s_{ij}^\ell$  of order $\ell$ so that
\begin{align}
 \label{DxsKdecintro}
   \pa_{x^i} \ts_K &= s_{1i}^{\,{\,}_{\!} 0}\, \tvarrho_1(\tau-\ttau_2) + s_{2i}^{\,{\,}_{\!}0}\, \tvarrho_2(\tau-\ttau_1) + (s_{3i}^{\,{\,}_{\!}0} + s_{0i}^{-1}\tau)\xi_r, \\
   \pa_{\xi_j} \ts_K &=  s_{1j}^{-1} \tvarrho_1(\tau-\ttau_2) + s_{2j}^{-1} \tvarrho_2(\tau-\ttau_1) + ( s_{3j}^{-1} +  s_{0j}^{-2}\tau)\xi_r,
\end{align}
Moreover, there are homogeneous symbols $e_i^\ell,e_{ij}^\ell$  of order $\ell$ so that
\begin{equation}\label{qKdecintro}
\te_K = e_1^{-1} \tvarrho_1(\tau-\ttau_2) + e_2^{-1}\varrho_2(\tau-\ttau_1) + (e_3^{-1} + e_0^{-2}\tau) \xi_r.
\end{equation}

This is the content of Lemma~\ref{qK}, and it was not obvious, since the symbols $\ts_K$ and $\te_K$ were defined indirectly in \cite{TT}. In particular the symbol $\te_K$ seems difficult to compute explicitly, and we use an indirect argument to show that it vanishes.

This observation allows us to conclude that the principal term coming from the perturbation $h$ in the bilinear form $\widehat{Q}$ in \eqref{bfintro} satisfies
\beq\label{bfintro2}
\widehat{Q}[h,S,E]\!= \Re \bigtwo(\widehat{Q}^\alpha [h,S,E] u \cdot \overline{D_{\alpha}u} \bigtwo),
\end{equation}
where
\begin{equation}\label{eq:Qhatoperator}
\widehat{Q}^\alpha[h,S,E]\!= h^{\alpha\beta}[D_{\beta}, S] + \frac12 [ h^{\alpha\beta}\!, S] D_{\beta} +  h^{\alpha\beta} E D_{\beta} + \frac12  h^{\alpha\beta} [D_{\beta},E] .
\eq
has principal symbol
\begin{multline}
\frac1i\left(h^{\alpha i}\pa_{x^i} s_K -\frac{1}{2}\pa_{x^\beta} h^{\alpha \beta}\, \pa_{\xi^\beta} s_K\right)+h^{\alpha \beta} e_K \xi_\beta 
=\sum\limits_{|\gamma|\leq 1} \pa^\gamma h^{\alpha\beta}\bigthree( d_{1\gamma\beta}^{\,0}\, \varrho_1(\tau-\tau_2) + d_{2\gamma\beta}^{\,0} \,\varrho_2(\tau-\tau_1) + \bigtwo( d_{3\gamma\beta}^{\,0} +  d_{0\gamma\beta}^{-1}\,\tau\bigtwo)\xi_r\bigthree),
\end{multline}
where $d_{k\gamma\beta}^{\,\ell}$ and $d_{ki}^{\,\ell}$ are operators of order $\ell$.
In other words each of the terms has a factor of the operators that we can estimate in \eqref{eq:LElocalbounds} and \eqref{eq:localenergyDxminusoneDt} below.
In turn, this allows us to estimate the principal part of the error by the right hand side of \eqref{LEKintro}, see Proposition~\ref{commhest}.

\subsubsection{Pseudo differential operators with low regularity coefficients}
\label{sec:lowregularity}
In order to estimate the lower order terms coming from $h$, we need $L^2 \to L^2$ and $H^{-1} \to H^{-1}$ bounds for operators with symbols with limited regularity. Since such results are not as standard in the literature as the corresponding results for classical symbols, we provide all the needed estimates in Section 3. We also tried to optimize the number of derivatives on $h$ that one must control. We believe that this is sharp, with the possible exception of the $\delta$ in \eqref{holderintro}.

We first notice that the first term in \eqref{eq:Qhatoperator}, i.e.
$ h^{\alpha\beta}[D_{\beta}, S] $,  is particularly easy to estimate
since we don't have to worry about the regularity of $[D_{\beta}, S]$, and its principal part is estimated above. We write
\begin{equation}
\langle h^{\alpha\beta}[D_{\beta}, S]u,D_\alpha u\rangle
= \langle [D_{\beta}, S]u, h^{\alpha\beta} D_\alpha u\rangle,
\end{equation}
and bound the first term by local energy and the second by $\|h^{\alpha\beta}\|_{L^\infty}\|D u\|_{L^2}$.

We would like to try to do the same thing for the second term in \eqref{eq:Qhatoperator}, i.e.
$[ h^{\alpha\beta}\!, S] D_{\beta}  $.
The explicit expression for the Weyl quantization is
\begin{equation}\label{eq:lowregularitypseudodiff}
\int [h^{\alpha\beta}\! , S ]{D}_j u\cdot \overline{{D}_\beta u} \,dx
=\iiint  [h^{\alpha\beta}(x)-h^{\alpha\beta}(y)] [s_K(\tfrac{x+y}{2}, \xi,D_t) {D}_j u(y) e^{i(x-y)\xi}
  \overline{{D}_\beta u(x)} \, d\xi dy \, dx.
\end{equation}
This can already be seen to be of the same form, if one moves $h^{\alpha\beta}(x)$ to ${D}_\beta u(x)$ and $h^{\alpha\beta}(y)$ to ${D}_j u(y)$. However, this wouldn't take into account the cancellation which makes the commutator lower order.
We now write
\begin{equation}
h^{\alpha\beta}(x)-h^{\alpha\beta}(y) = \frac12(x-y)^k \big(\pa_{x^k} h^{\alpha\beta}(x) + \pa_{y^k} h^{\alpha\beta}(y)\big) + r^{\alpha\beta}(x,y).
\end{equation}
where these things are just multiplications and the remainder $r^{j\beta}$ cancels to higher order $|x-y|^{3+\delta}$.
The factor $i(x^k-y^k) e^{i(x-y)\xi}=\pa_{\xi_k} e^{i(x-y)\xi}$ can be integrated by parts to fall on  $s_K(\tfrac{x+y}{2}, \xi,D_t) $ and decrease the order by one and then we are left with multiplication operators by $\pa_{x^k} h^{\alpha\beta}(x) $ that can multiply ${D}_\beta u(x)$  and $\pa_{y^k} h^{\alpha\beta}(y)$ that can multiply ${D}_j u(y) $. A similar argument is used to control the remainder when the difference in \eqref{eq:lowregularitypseudodiff} is replaced by $r^{\alpha\beta}$. Since it cancels to order
$|x-y|^{3+\delta}$ we can as above integrate by parts in $\xi$ three times to reduce the order of the operator so the integral
\eqref{eq:lowregularitypseudodiff} is a bilinear form on $L^2$,
see Lemma \ref{lem:theerrorlemma} and Proposition~\ref{commhest}. It is here that we need extra smallness (but not decay!) assumptions on higher order derivatives of $h$.

\subsubsection{How to deal with time derivatives}
\label{sec:timederivatives}
There are two technical issues involving the presence of time derivatives that we need to address. Both of these issues stem from the fact that $S$ and $E$ contain a term including $\pa_t$, and thus do not appear in the Schwarzschild case.

The first issue is that for general metrics $g$ and operators $S$ and $E$ the operator $\widehat{T_g}$ could have a term that has three time derivatives. This issue already arises in the case of the Kerr metric, and it was dealt with by imposing on $S$ and $E$ an extra condition, see \eqref{nodt3}, under which there is no such term.

The same problem arises for perturbations, but we can deal with it in a simpler way.  Our approach will be to multiply $S$ and $E$ by a suitable function $f_0$ to achieve that $h^{00}=0$. If we multiply
$\widehat{\Box}_g={\pa}_\alpha (\widehat{g}^{\alpha\beta} {\pa}_\beta)$, where $\widehat{g}^{\alpha\beta}=\sqrt{|g|}\, g^{\alpha\beta}$, with $f_0$
we get
\begin{equation}\label{eq:makingh00equalto0}
f_0\widehat{\Box}_g u
 ={\pa}_\alpha \big(f_0\widehat{g}^{\alpha\beta} {\pa}_\beta u\big)
-\widehat{g}^{\alpha\beta}\pa_\alpha f_0\,\,   {\pa}_\beta u  .
\end{equation}
We achieve that for the perturbation $h^{00}\!=0$ by choosing $f_0$ so that
$f_0\widehat{g}^{00}\!=\widehat{g}_K^{00}$ close to the trapped set, see \eqref{f0def}.
This destroys the symmetry of the operator but the additional term introduced in the energy identity \eqref{eq:Qhatdef0intro} is under control since it is the additional term in \eqref{eq:makingh00equalto0}, where $\pa f_0=O(\pa h)$,  multiplied by the operator $(S+E)u$ that has the cancellation on the trapped set.

The second issue is that the errors will have terms containing $D_t \la D_x\ra^{-1}$ which are not directly included in the local energy norms. We thus introduce a new norm,
\[
\|D_t u\|_{LE^0[\tt_0, \tt_1]} = \|\chi D_r \la D_x\ra^{-1} \chi D_t u\|_{L^2[\tt_0, \tt_1]L^2} + \|\chi \la D_x\ra^{-1} \chi D_t u\|_{L^2[\tt_0, \tt_1]L^2},
\]
and use the equation to control this by the local energy norms and the inhomogeneous term. In particular we prove that (see \eqref{LEtest})
\begin{equation}\label{eq:localenergyDxminusoneDt}
\|D_t u\|_{LE^0[\tt_0, \tt_1]}^2 \lesssim \|u\|_{LE_K^1[\tt_0, \tt_1]}^2 +  E[u](\tt_0)+E[u](\tt_1) + \|\kappa \pa u\, \|_{L^2_{ps}[\tt_0, \tt_1]}^2+B(F,u)_{[\tt_0, \tt_1]}.
\end{equation}

\bigskip

\newsection{Local energy estimates on Kerr backgrounds}
\subsection{The setup and statement}
\subsubsection{The coordinates}\label{sec:thecoordinates}
 The Kerr geometry in Boyer-Lindquist coordinates is given by
\[
ds^2 = g^K_{tt}dt^2 + g_{t\phi}dtd\phi + g^K_{rr}dr^2 + g^K_{\phi\phi}d\phi^2,
 + g^K_{\theta\theta}d\theta^2
\]
 where $t \in \R$, $r > 0$, $(\phi,\theta)$ are the spherical coordinates
on $\S^2$ and
\[
 g^K_{tt}=-\frac{\Delta-a^2\sin^2\theta}{\rho^2}, \qquad
 g^K_{t\phi}=-2a\frac{2Mr\sin^2\theta}{\rho^2}, \qquad
 g^K_{rr}=\frac{\rho^2}{\Delta},
 \]
\[ g^K_{\phi\phi}=\frac{(r^2+a^2)^2-a^2\Delta
\sin^2\theta}{\rho^2}\sin^2\theta, \qquad g^K_{\theta\theta}={\rho^2},
\]
with
\[
\Delta=r^2-2Mr+a^2, \qquad \rho^2=r^2+a^2\cos^2\theta.
\]

 Here $M$ represents the mass of the black hole, and $aM$ its angular momentum.

 A straightforward computation gives us the inverse of the metric:
\[ g_K^{tt}=-\frac{(r^2+a^2)^2-a^2\Delta\sin^2\theta}{\rho^2\Delta},
\qquad g_K^{t\phi}=-a\frac{2Mr}{\rho^2\Delta}, \qquad
g_K^{rr}=\frac{\Delta}{\rho^2},
\]
\[ g_K^{\phi\phi}=\frac{\Delta-a^2\sin^2\theta}{\rho^2\Delta\sin^2\theta}
, \qquad g_K^{\theta\theta}=\frac{1}{\rho^2}.
\]

The case $a = 0$ corresponds to the Schwarzschild space-time.  We shall
subsequently assume that $a$ is small $0 < a \ll M$, so that the Kerr
metric is a small perturbation of the Schwarzschild metric. Note also that the coefficients depend only $r$ and $\theta$ but are independent of
$\phi$ and $t$. We denote the Kerr metric by $g_K$, and the Schwarzschild metric by $g_S$. For any Lorentzian metric $g$, let
$\Box_{g} $ denote the d'Alembertian associated to it.

In the above coordinates the Kerr metric has singularities at $r = 0$,
on the equator $\theta = \pi/2$, and at the roots of $\Delta$, namely
$r_{\pm}=M\pm\sqrt{M^2-a^2}$. To remove the singularities at $r = r_{\pm}$ we
introduce functions $r_K^*=r_K^*(r)$, $v_{+}=t+r_K^*$ and $\phi_{+}=\phi_{+}(\phi,r)$ so that (see
\cite{HE})
\[
 dr_K^*=(r^2+a^2)\Delta^{-1}dr,
\qquad
 dv_{+}=dt+dr_K^*,
\qquad
 d\phi_{+}=d\phi+a\Delta^{-1}dr.
\]

Note that when $a=0$ the $r_K^*$ coordinate becomes the Schwarzschild Regge-Wheeler coordinate
\[
\rs=r+2M\log(r-2M)
\]

The Kerr metric can be written in the new coordinates $(v_+, r, \phi_+, \theta)$
\[
\begin{split}
ds^2= &\
-(1-\frac{2Mr}{\rho^2})dv_{+}^2+2drdv_{+}-4a\rho^{-2}Mr\sin^2\theta
dv_{+}d\phi_{+} -2a\sin^2\theta dr d\phi_{+} +\rho^2 d\theta^2 \\
& \ +\rho^{-2}[(r^2+a^2)^2-\Delta a^2\sin^2\theta]\sin^2\theta \,
d\phi_{+}^2
\end{split}
\]
which is smooth and nondegenerate across the event horizon up to but not including
$r = 0$. We introduce the function
\[
\tt = v_{+} - \mu(r),
\]
where $\mu$ is a smooth function of $r$. In the $(\tt,r,\phi_{+},
\theta)$ coordinates the metric has the form
\[
\begin{split}
ds^2= &\ (1-\frac{2Mr}{\rho^2}) d\tt^2
+2\left(1-(1-\frac{2Mr}{\rho^2})\mu'(r)\right) d\tt dr \\
 &\ -4a\rho^{-2}Mr\sin^2\theta d\tt d\phi_{+} + \Bigl(2 \mu'(r) -
 (1-\frac{2Mr}{\rho^2}) (\mu'(r))^2\Bigr)  dr^2 \\
 &\ -2a\theta (1+2\rho^{-2}Mr\mu' (r))\sin^2dr d\phi_{+} +\rho^2
 d\theta^2 \\
 &\ +\rho^{-2}[(r^2+a^2)^2-\Delta a^2\sin^2\theta]\sin^2\theta
d\phi_{+}^2.
\end{split}
\]

On the function $\mu$ we impose the following two conditions:

(i) $\mu (r) \geq  \rs$ for $r > 2M$, with equality for $r >
{5M}/2$.

(ii)  The surfaces $\tt = const$ are space-like, i.e.
\[
\mu'(r) > 0, \qquad 2 - (1-\frac{2Mr}{\rho^2}) \mu'(r) > 0.
\]
As long as $a$ is small, we can use the same
function $\mu$ as in the case of the Schwarzschild space-time in \cite{MMTT}.

 For convenience we also introduce
\[
\tphi = \zeta(r)\phi_{+}+(1-\zeta(r))\phi,
\]
where $\zeta$ is a cutoff function supported near the event horizon
and work in the $(\tt,r,\tphi, \theta)$ coordinates which are
identical to $(t,r,\phi,\theta)$ outside of a small neighborhood of
the event horizon,  and in particular near the trapped set, where most of our analysis takes place.

Much of the analysis we will need to do, in particular using the Fourier transforms close to the photon sphere, are easiest to do in rectangular coordinates $x\in\Bbb R^3$.
We will therefore let the rectangular coordinates $(x^1,x^2,x^3)$ stand for the coordinates which in spherical coordinates are given by $(r,\phi,\theta)$, and we let $x^0=\tilde{t}$.

\subsubsection{The domain and boundary energies}
Given $0 < r_e <r_{+}$, we consider the wave equation
\begin{equation}
 \Box_{g_K} u = f,
 \label{boxsinhom}\end{equation}
in the cylindrical region
\begin{equation}
 \M_{[\tt_0, \tt_1]} =  \{ \tt_0 \leq \tt \leq \tt_1, \ r \geq r_e \}.
\label{mr}\end{equation}

The lateral boundary of $\M_{[\tt_0, \tt_1]}$,
\begin{equation}
 \Sigma_R^+ =   \M_{[\tt_0, \tt_1]} \cap \{ r = r_e\},
\label{mr+}\end{equation}
is space-like, and can be thought of as the exit surface
for all waves which cross the event horizon.

We define the  outgoing energy on $\Sigma_R^+$ as
\begin{equation}\label{energy2}
 E[u](\Sigma_R^+) = \int_{\Sigma_R^+}
 \left(  |\partial_r u|^2 +   |\partial_t u|^2    +
|\ang u|^2 \right) r_e^2   dt d\omega,
\end{equation}
and the energy on an arbitrary slice $T=\tt$  as
\begin{equation}\label{energy3}
 E[u](T) = \int_{ \M_{[\tt_0, \tt_1]} \cap \{\tt = T\}}
 \left(
|\partial_r u|^2 +   |\partial_t u|^2    + |\ang u|^2
\right) r^2  dr  d\omega.
\end{equation}

\subsubsection{Trapped geodesics}
We will use the results from \cite{TT}, which we now recall. In order to define pseudodifferential operators, we will use the usual cartesian coordinates $x^j$ and the dual variables in the cotangent bundle $\xi_j$. On the other hand, it is easier to understand trapping in spherical coordinates $(t, r, \phi, \theta)$. Let $\tau$, $\xi_r$, $\Phi$ and $\Theta$ be
the dual variables in the cotangent bundle corresponding to $t$, $r$, $\phi$ and $\theta$; they are related to $x_j$ and $\xi_j$ by
\[
\xi_r = \frac{\pa x^k}{\pa r}\xi_k, \quad \Phi = \frac{\pa x^k}{\pa \phi}\xi_k, \quad \Theta = \frac{\pa x^k}{\pa \theta}\xi_k,\quad \tau=\xi_0.
\]

Let
\[
p_K( r, \phi,\tau, \xi_r, \Phi,\Theta)=g^{\alpha\beta}_K\xi_\alpha\xi_\beta
=g_K^{tt}\tau^2+2g_K^{t\phi}\tau\Phi+g_K^{\phi\phi}\Phi^2
+g_K^{rr}\xi_r^2 +g_K^{\theta\theta}\Theta^2,
\]
 be the principal symbol of $\Box_{g_K} $.

 On any null geodesic one has
\begin{equation}\label{Ham}
 p_K ( r, \phi,\tau, \xi_r, \Phi,\Theta) =0.
\end{equation}

The Hamilton flow equations also give us, in particular, that
\begin{equation}\label{rdot}
 \dot{r}=-\frac{\partial p_K}{\partial \xi_r}=-\frac{2\Delta}{\rho^2} \xi_r,
\end{equation}
\begin{equation}\label{xidota}
  \rho^2 \dot{\xi_r}=\rho^2 \frac{\partial p_K}{\partial r}=
  - 2R_a(r,\tau,\Phi) \Delta^{-2} + \rho^2 \partial_r(\rho^{-2}) p_K + 2(r-M)\xi^2,
\end{equation}
where
\[
R_a(r,\tau,\Phi) =
(r^2+a^2)(r^3-3Mr^2+a^2r+a^2M)\tau^2 - 2aM(r^2-a^2)\tau\Phi
 - a^2(r-M)\Phi^2.
\]

 As a consequence, all trapped null geodesics in the exterior $r>r_+$ must lie in the region $|r-3M| \leq 2a$ and also satisfy (see \cite{TT} for more details):
 \begin{equation}\label{xi}
 \xi_r= 0,
 \end{equation}
 \begin{equation}\label{tauphi}
  R_a(r,\tau,\Phi)= 0.
 \end{equation}

 By \eqref{Ham} we can bound $\Phi$ in terms of $\tau$,
\begin{equation}\label{Phibd}
|\Phi| \leq 4 M |\tau|.
\end{equation}
For $\Phi$ in this range and small $a$ the polynomial
$\tau^{-2} R_a(r,\tau,\Phi)$ can be viewed as a small perturbation of
\[
\tau^{-2} R_0(r,\Phi,\tau) = r^4(r-3M),
\]
which has a simple root at $r = 3M$. Hence for small $a$ the
polynomial $R_a$ has a simple root close to $3M$, which we denote by
$r_a(\Phi/\tau)$:
\[
\tau^{-2} R_a(r,\Phi/\tau) = \widehat{R}_a(r,\Phi/\tau) \big(r-r_a(\Phi/\tau)\big),\qquad
\widehat{R}_a(r,\Phi/\tau)\geq c>0.
\]
Thus if we denote
 \[
 \M_{ps}{[\tt_0,\tt_1]} = [\tt_0,\tt_1] \times I_{ps}, \quad\text{where}\quad  I_{ps} :=  \big\{ \, r\,;\,  |r-3M|\leq M/4\, \big\},
 \]
all the trapped  null geodesics in the exterior region lie in $\M_{ps}{[\tt_0,\tt_1]}$. This is the region where most of our analysis will take place.

\subsubsection{The local energy decay norm}
 One could now try to define a spacetime multiplier by quantizing a (multiple of) $\varrho_a=r-r_a(\Phi/\tau)$. Whereas this works well at the symbol level, it has the disadvantage of having a complicated dependence on $\tau$, which makes it inconvenient for energy estimates on constant time slices. Instead, we factor
 \begin{equation}\label{eq:tauidef}
p_K = g_K^{tt}(\tau -\ttau_1) (\tau -\ttau_2),
\end{equation}
where $\ttau_i=\ttau_i( r, \phi, \xi_r, \Phi,\Theta)$,are real distinct smooth $1$-homogeneous symbols with respect to the space Fourier variables.

On the cone $\tau = \ttau_i$ the symbol $r-r_a(\Phi/\tau)$ equals
\begin{equation}\label{eq:cidef}
\tvarrho_i ( r, \phi, \xi_r, \Phi,\Theta) = r- r_i( r, \phi, \xi_r, \Phi,\Theta),\qquad\text{where} \quad r_i=r_a(\Phi/\tau_i) =3M -a F\Big(\frac{a}{M}, \frac{\Phi}{M \tau_i}\Big), \qquad i = 1,2.
\end{equation}

If $r$ is close to $3M$ and $|a| \ll M$ then on the characteristic set
of $p_K$ we have \eqref{Phibd}, therefore the symbols $\tvarrho_i$ are
well defined, smooth and homogeneous.

 We use the symbols $\varrho_i$ to define associated microlocally weighted function spaces
$L^2_{\varrho_i}$ in $I$. In order to remove the singularity at zero frequencies, we define
\[
\varrho_i = \chi_{\geq 1} (r- r_a(\Phi/\tau_i)), \qquad i = 1,2,
\]
where $\chi_{\geq 1}$ is a smooth symbol which equals $1$ for large frequencies and $0$ for small ones.

For functions $u$ supported in $I_{ps}$ we set
\[
\| u\|_{L^2_{\varrho_i}}^2 = \| \varrho_i(D, x) u\|_{L^2}^2 + \|u\|_{H^{-1}}^2,
\]
and the dual norm
\[
\| g\|_{\varrho_i L^2}^2 = \inf_{\varrho_i(x, D) g_1 + g_2 = g} (\|g_1\|_{L^2}^2 +
 \|g_2\|_{H^1}^2), \quad \|g\|_{\varrho L^2} := \|g\|_{\varrho_1 L^2 + \varrho_2 L^2}.
\]
Here $P(D,x) u(x)=\iint p(\xi,y) e^{(x-y)\cdot\xi} \,u(y)\, dy\, d\xi/(2\pi)^3$.

We consider a partition
of $ \R^{3}$ into the dyadic sets $A_R= \{\la r \ra \approx R\}$ for
$R \geq 1$, with the obvious change for $R=1$, and define the local energy norm $LE$
\begin{equation}
 \| u\|_{LE} = {\sup}_R \, \| \la r\ra^{-\frac12} u\|_{L^2 (\R \times A_R)}  ,\qquad
 \| u\|_{LE[\tt_0, \tt_1]} ={\sup}_R \, \| \la r\ra^{-\frac12} u\|_{L^2 ([\tt_0, \tt_1] \times A_R)},
\label{ledef}\end{equation}
its $H^1$ counterpart
\begin{equation}
  \| u\|_{LE^1} = \| \nabla u\|_{LE} + \| \la r\ra^{-1} u\|_{LE},\qquad
 \| u\|_{LE^1[\tt_0, \tt_1]} = \| \nabla u\|_{LE[\tt_0, \tt_1]} + \| \la r\ra^{-1} u\|_{LE[\tt_0, \tt_1]},
\end{equation}
as well as the dual norm
\begin{equation}
 \| f\|_{LE^*} = {\sum}_R  \| \la r\ra^{\frac12} f\|_{L^2 (\R \times A_R)},\qquad
 \| f\|_{LE^*[\tt_0, \tt_1]} = {\sum}_R  \| \la r\ra^{\frac12} f\|_{L^2 ([\tt_0, \tt_1] \times A_R)}.
\label{lesdef}\end{equation}

Now we can define local energy norms associated to the Kerr
space-time.  Let $\chi(r)$ be a smooth cutoff function which is
supported in $I_{ps}$ and which equals $1$
in  $ I^\prime_{ps}= \big\{ \, r\,;\,  |r-3M|\leq M/8\, \big\}\subset I_{ps}$. Let $\chi_{0}^2=1-\chi^2$ and let $\widetilde{\chi}_0$ be a smooth cutoff so that $\widetilde{\chi}_{0} = 1$ on the support of $\chi_0$, and $\widetilde{\chi}_0=0$ in
 $I^{\prime\prime}_{ps}= \big\{ \, r\,;\,  |r-3M|\leq M/16\, \big\}$. We define
\begin{equation}
\begin{split}
\|u\|_{LE_K^1[\tt_0, \tt_1]} = & \|\chi(D_t - \tau_2(D,x))\chi u\|_{L^2[\tt_0, \tt_1]L^2_{\varrho_1}}
+ \|\chi (D_t - \tau_1(D,x))\chi u\|_{L^2[\tt_0, \tt_1]L^2_{\varrho_2}} + \|\widetilde{\chi}_{0} \partial_t  u\|_{LE[\tt_0, \tt_1]}
\\ &\ +  \|\widetilde{\chi}_{0} \ang  u\|_{LE[\tt_0, \tt_1]}
+ \| \partial_r u\|_{LE[\tt_0, \tt_1]} + \| r^{-1} u \|_{LE[\tt_0, \tt_1]}.
\end{split}
\label{leK}\end{equation}
and similarly for $LE_K^1$. We remark that this norm is equivalent to the $LE^1$ norm
outside of $I_{ps}$.

For the nonhomogeneous term in the equation we define a dual structure,
\[
\|f\|_{LE_K^*} = \| (1-\chi) f\|_{LE^*} + \|\chi f\|_{L_t^2 \varrho L^2}.
\]

Similarly we define the $LE_K^1[\tt_0, \tt_1]$ and $LE_K^*[\tt_0, \tt_1]$ to be the analogous norms when the integration in time is on the interval $[\tt_0, \tt_1]$.

The following result was proved in \cite{TT}:
\begin{theorem}\label{Kerr}
 Let $u$ be so that $\Box_{g_K} u = f$.
 Then we have
\begin{equation}\label{main.est.Kerr}
E[u](\Sigma_R^+) + {\sup}_{\tt_0 \leq t\leq \tt_1} E[u](\tt) + \|u\|_{LE_K^1[\tt_0, \tt_1]}^2
\lesssim E[u](\tt_0) + \|f\|_{L^1[\tt_0, \tt_1] L^2+LE^*_K[\tt_0, \tt_1]}^2.
\end{equation}
\end{theorem}
\begin{remark} Note that here
 \begin{equation}\label{eq:Fnorm}
\|f\|_{L^1[\tt_1, \tt_2] L^2+LE^*_K[\tt_1, \tt_2]}
\lesssim
 \inf_{f_1 + f_2 =f} \bigtwo(\|f_{\!1}\|_{L^1[\tt_0, \tt_1] L^2} + \|\pa^{\leq 1}\chi  f_2\|_{L^2[\tt_0, \tt_1]L^2} + \|(1-\chi)f_2\|_{LE^*[\tt_0, \tt_1] }\bigtwo).
\end{equation}
\end{remark}
As we will generalize Theorem \ref{Kerr} to perturbations of the Kerr metric, we recall and expand on the key steps in its proof from \cite{MMTT}, \cite{TT}, and \cite{LT}.

\subsection{The Schwarzschild case}

We now discuss the key results from \cite{MMTT} and \cite{LT} that we will need in the next section.

\subsubsection{The energy momentum tensor and deformation tensor}
Let
\[
P_\alpha[g,X]=Q_{\alpha\beta}[g]X^\beta,\qquad\text{where}\qquad Q_{\alpha\beta}[g]=\partial_\alpha u \,\partial_\beta u -
\frac{1}{2}g_{\alpha\beta}\partial^\gamma u\, \partial_\gamma u .
\]
be the energy-momentum tensor and $X$ any vector field.
Then
\[
\nabla^\alpha P_\alpha[g,X] = \Box_g u \cdot Xu + Q[g, X],
\qquad\text{where}\qquad
Q[g, X] =\frac{1}{2} Q_{\alpha\beta}[g]\pi_X^{\alpha \beta},
\]
and $\pi_X^{\alpha \beta} $ is the deformation tensor of $X$. In terms of the Lie derivative of the inverse of the metric we have
\[
\pi^{\alpha\beta}_X=\nabla^\alpha X^\beta + \nabla^\beta X^\alpha=-{\mathcal L}_X g^{\alpha\beta}=-X(g^{\alpha\beta})
+g^{\alpha\gamma} \partial_\gamma X^\beta+g^{\beta\gamma}\partial_\gamma X^\alpha.
\]
This can also be expressed in terms of the Poisson bracket of the symbols $p_g=g^{\alpha\beta}\xi_\alpha\xi_\beta $ and $s_X=X^\alpha\xi_\alpha$:
\begin{equation}\label{eq:poissonone}
\pi^{\alpha\beta}_X\xi_\alpha\xi_\beta=\{p_g,s_X\},\qquad\text{where}\qquad  \{p,s\}=\frac{\partial p}{\partial{\xi_\alpha}} \, \frac{\partial s}{\partial{x^\alpha}}-\frac{\partial p}{\partial{x^\alpha}}\, \frac{\partial s}{\partial{\xi_\alpha}}.
\end{equation}
We also note that with $|g|\!=\!|_{\!}\det{g}|$
\beq\label{eq:pidiv}
\pi_X^{\alpha \beta}Q_{\alpha\beta}[g]=\pi_X^{\alpha \beta}\partial_\alpha u \,\partial_\beta u -
\div X\,\partial^\gamma u\, \partial_\gamma u,\qquad\text{where}\qquad
\div X\!=\!\nabla_{\!\alpha} X^\alpha\!=|g|^{-1/2}\pa_\alpha \big( |g|^{1/2} X^{\alpha}\big).
\eq

For a vector field
$X$, a scalar function $q$ and a 1-form $m$ we further define
\[
P_\alpha[g,X,q,m] = P_\alpha[g,X] + q u \,\partial_\alpha u  + \frac{1}{2}\bigtwo(m_{\alpha}-\partial_\alpha q\bigtwo)u^2.
\]
The divergence formula gives
\begin{equation}
\nabla^\alpha P_\alpha[g,X,q,m] =  \Box_g u \big(Xu +
 q u\big)+ Q[g,X,q,m],
\label{div}\end{equation}
where
\beq\label{Qdef}
 Q[g,X,q,m] =
Q[g, X] + q\,
\partial^\alpha u\, \partial_\alpha u + m_\alpha u\,
\partial^\alpha u +\frac{1}{2} \big(\nabla^\alpha m_\alpha -
\nabla^\alpha \partial_\alpha q\big) \, u^2.
\eq

\subsubsection{The vector field in the Schwarzschild case}
 In the Schwarzschild case, let us recall the results of \cite{MMTT} and Section 2 of \cite{LT}. Pick any $0<\delta\ll 1$. There is a smooth vector field (in the nondegenerate coordinate system)
\[
 X = b(r)\bigthree(1-\frac{3M}r\bigthree) \partial_r + c(r)\partial_{\tt} + f(r)\partial_r,
 \]
with $c$ supported near $r=2M$, $b>0$ bounded, and $f$ supported on $r\geq R_1$ for $R_1$ large, a smooth function
\[
q(r) = \frac{1}{2r^2} \weight \partial_r \Bigl(\frac{r^2(r-3M)}{r-2M} b(r)\Bigr) - \delta_1 \frac{(r-3M)^2}{r^4} + \frac{f(r)}r, \qquad 0< \delta_1 \ll 1 ,
\]
so that $\Box_{g_S} q <0$,
and a smooth 1-form $m$ supported near $r=2M$ so that
\begin{equation}\label{SchwQ}
Q[g_S,X,q,m] \gtrsim r^{-1-\delta} |\partial_r u|^2 + \big(1-\frac{3M}r
\big)^2 r^{-1-\delta} (|\partial_t u|^2 + |\ang u|^2) + r^{-3-\delta}
u^2.
\end{equation}
Moreover, for a large constant $C$,
\begin{equation}
   \int_{\M_{[\tt_0,\tt_1]}}\!\!\!\!\! Q[g_S, C\partial_t + X, q, m] dV_{\mathbf S} =
   - \int_{\M_{[\tt_0,\tt_1]}} \!\!\!\!\! (\Box_{g_S} u) (C\partial_t + X + q) u \, dV_{\mathbf S}
   - \left. BDR^{\mathbf S}[u]\right|_{\tt=\tt_0}^{\tt = \tt_1} - \left. BDR^{\mathbf S}[u] \right|_{r=r_e}.
   \end{equation}
where the boundary terms are positive and satisfy
 \begin{equation}\label{bdrs}
      BDR^{\mathbf S}[u]\big|_{\tt = \tt_i} \approx \| \nabla u(\tt_i)\|_{L^2}, \quad\text{and}\quad
     BDR^{\mathbf S}[u]\big|_{r=r_e} \approx \| u\|_{H^1(\Sigma^+_{[\tt_0,\tt_1]})}^2.
 \end{equation}

 The principal symbol of $\Box_{g_S}$ can be written
\[
p_S = - \Big(1-\frac{2M}r \Big)^{-1} (\tau^2-\tau_S^2),\qquad\text{where}\qquad
 \tau_S^2=\Big(1-\frac{2M}r \Big)\Big(\Big(1-\frac{2M}r
\Big) \xi_r^2 + \frac{1}{r^2}\Big(\frac{1}{\sin^2\theta} \Phi^2+\Theta^2\Big)\Big).
\]

 For later use, we record that $Q[g_S, X, q, m]$ near $r=3M$ is given by
\[
Q[g_S, X, q, m] = q^{S, \alpha\beta}\partial_{\alpha}u\, \partial_{\beta} u + q^{S, 0} u^2,
\]
where
\begin{equation}
q^{S, 0} = -\Box_{g_S} q /2\geq c_0> 0,
\end{equation}
 and by \eqref{eq:poissonone} and \eqref{eq:pidiv} the coefficients are given by
\[
q^S := q^{S, \alpha\beta}\xi_{\alpha}\xi_{\beta} = \frac12 \bigl\{ p_S, b(r)\bigthree(1-\frac{3M}r\bigthree) \xi_r \bigr\} + \bigl(q - q_X \bigr) p_S, \qquad\text{where}\qquad  q_X := \frac12 \div(X).
\]
We also compute
\[
q_S:= q - q_X = -\delta_1 \frac{(r-3M)^2}{r^4} - \frac{M(r-3M)}{r^2(r-2M)} b(r).
\]

 The exact formulas for $b$ and $m$ are not important, but we emphasize the following two facts: the coefficient of $\partial_r$ in $X_S$ vanishes on the trapped set $r=3M$, and so does the expression $q _S$. This observation was already used in our previous paper \cite{LT}, and we will prove that a similar phenomenon happens in Kerr.

 A straightforward computation gives (see (4.32) in \cite{TT})
 \beq\label{sumsqS0}
r^2 q^S
=\alpha_S^2(r)  \tau^2 + \beta_S^2(r)  \xi^2 + \tilde q_S(r)r^2 p_S.
\eq
In fact,
 \[
r^2 q^S=  \frac12 \bigl\{ r^2 p_S, b(r)\bigthree(1-\frac{3M}r\bigthree) \xi_r \bigr\} + \bigl(q - q_X+ r^{-2} b(r)(r-3M) \bigr) r^2 p_S.
\]
Here
 \begin{equation}\label{eq:sumsquareschwarz}
 \frac12 \bigl\{ r^2 p_S, b(r)\bigthree(1-\frac{3M}r\bigthree) \xi_r \bigr\}
 =\alpha_S^2(r)  \tau^2 + \beta_S^2(r)  \xi^2,
 \end{equation}
 where, near  $r = 3M$,
\[
\alpha_S^2(r) =  \frac{ r b(r) (r-3M)^2}{(r-2M)^2},
\qquad\text{and}\qquad\beta_S^2(r) = \frac{3M}{r^2} b(r^2 -2Mr) + \Big(1-\frac{3M}r\Big)(b'(r^2 -2Mr) -b(r-M)),
\]
respectively
\[
\tilde q_S (r) =  q_S + r^{-2} b(r)(r-3M) = -\delta_1 \frac{(r-3M)^2}{r^4} + \frac{(r-3M)^2}{r^2(r-2M)} b(r).
\]

Let $\nu(r)$ be defined as
\begin{equation}\label{nu}
(1-\nu)\alpha_S^2 = \delta_1 \frac{(r-3M)^2}{r^4}, \qquad 0<\nu(r)<1.
\end{equation}
We can rewrite \eqref{sumsqS0} as the following sum of squares representation
\begin{equation}
r^2 q^S = (1-\nu)\alpha_S^2 \tau^2 + \nu \alpha_S^2 \tau_S^2 + \beta_S^2 \xi_r^2.
\end{equation}

The expression above has the disadvantage that $\tau_S$ is not the symbol of a differential operator. We can rectify this by writing the spherical Laplacian a sum of squares of differential symbols,
\[
\lambda^2 := \frac{1}{\sin^2\theta} \Phi^2+\Theta^2 = \lambda_1^2 + \lambda_2^2 + \lambda_3^2,
\]
where in Euclidean coordinates have
\begin{equation}\label{eq:thelambdas}
\{  \lambda_1, \lambda_2,\lambda_3\} = \{ x_i \xi_j -x_j \xi_i, \ i \neq
j\}.
\end{equation}

We can now write
\begin{equation}
  r^2 q^S =  (1-\nu(r)) \alpha_S^2(r)  \tau^2 + \beta_S^2(r)  \xi^2 +
 \frac{r-2M}{r^3} \nu(r) \alpha_S^2(r) \left(\lambda_1^2 +\lambda_2^2 +\lambda_3^2
  + (r^2 -2rM) \xi_r^2\right).
\label{sumsqS1}\end{equation}

In the next section we describe a similar decomposition in the case of the Kerr metric that is a small perturbation of the one above.

\subsection{The Kerr case}

 Moving to the Kerr case, we perturb $X$ and $q$ microlocally near the trapped set.

\subsubsection{The generalized energy momentum tensor for operators}

 We will consider a pseudodifferential operator of order $1$, which is differential in $t$,
 \[
C\partial_t + X + S + q + E,
 \]
where $S$ is a skew-adjoint and $E$ is a self-adjoint operator of the form
 \begin{equation}\label{Kerrpdo}
 S = S_1+ S_0 \partial_t, \qquad E= E_0 + E_{-1} \partial_t.
 \end{equation}
  Here $S_j, E_j \in a OPS^j$ are real-valued, time-independent operators with kernels supported near $r=3M$, and moreover $S_1$ and $E_{-1}$ are self-adjoint, while $S_0$ and $E_0$ are skew-adjoint with respect to $dxdt$.

In order to ensure that there is no term with three time derivatives, we will also require that
\begin{equation}\label{nodt3}
[g^{tt}, S_0] +  g^{tt} E_{-1}  + E_{-1} g^{tt} = 0.
\end{equation}
Then  with $dV_g=\sqrt{|g|} dx dt$ we have
\beq
 \int_{\tt_0}^{\tt_1}\!\!\!\int \!\Box_g u\cdot \overline{(Su\!+\!E u)}\, dV_g
=\!\int_{\tt_0}^{\tt_1}\!\!\!\int \!\widehat{\Box}_g u\cdot \overline{(Su\!+\!E u)}\, dx dt.
\eq
where $\widehat{\Box}_g=\sqrt{|g|}\,\Box_g={\pa}_\alpha (\sqrt{|g|}\, g^{\alpha\beta} {\pa}_\beta)$ is
also symmetric with respect to $dx dt$.
Integrating by parts and using that $S$ is skew adjoint and $E$ self adjoint we obtain
modulo boundary terms
\beq
 \Re \!\int_{\!\tt_0}^{\tt_1}\!\!\!\!\!\int \widehat{\Box}_g u\cdot \overline{(Su+\!E u)}\, dx dt
=\!\int_{\!\tt_0}^{\tt_1}\!\!\!\!\!\int \widehat{T}_g u\cdot\overline{u}\, \ dx dt+\text{BDR}\Big|_{t=\tt_0}^{t=\tt_1},\quad\text{where}\quad
\label{eq:Qhatdef0}
\widehat{T}_g\!=\frac12\big([\widehat{\Box}_{g},S]\! + \widehat{\Box}_{g} E + E\,\widehat{\Box}_{g}\big).
\eq
We remark that this definition depends on a choice of coordinates. We will assume that for this definition  the Kerr metric
is expressed in the rectangular coordinates corresponding to the spherical coordinates above, in which case
$\sqrt{|g_K|}=(\rho/r)^2=1+O(a)$.

\subsubsection{The operators in the Kerr case}
  Let $dV_{\mathbf K}=\rho^2 dr dt d\omega$ denote the Kerr induced
  measure, and $S$ and $E$ be as in \eqref{Kerrpdo}, \eqref{nodt3}. Integrating by parts we obtain
\[
\Re \int_{\M_{[\tt_0,\tt_1]}} \Box_{g_K} u\cdot \overline{(S+E) u} \, dV_K = \int_{\M_{[\tt_0,\tt_1]}} \widehat{T}_K u \cdot \overline{u} \,dx dt + \left. BDR_1^{\mathbf K}[u, S, E]\right|_{t=\tt_0}^{t = \tt_1},
\]
We can write
\[
\widehat{T}_K = \frac1{2}\bigtwo([\widehat{\Box}_{g_K},S] + \widehat{\Box}_{g_K} E + E\, \widehat{\Box}_{g_K}\bigtwo)= Q_2  + 2 Q_1 D_t + Q_0 D_t^2,
\]
where $Q_j \in OPS^j$ are selfadjoint pseudodifferential operators. Note that, due to \eqref{nodt3}, there is no $D_t^3$ term in $Q$.

Moreover, \eqref{Kerrpdo} and \eqref{nodt3} imply that the boundary terms satisfy
\begin{equation}\label{pdobdry}
\left| BDR_1^{\mathbf K}[u, S, E] \right| \lesssim a \bigl(E[u](\tt_0) + E[u](\tt_1)\bigr).
\end{equation}
Motivated by this, we define
\[
IQ[g_K, S, E] = \int_{\M_{[\tt_0,\tt_1]}} Q_{2} u \cdot \overline u + 2\Re Q_1 u \cdot \overline {D_t u} + Q_0 {D_t u} \cdot\overline {D_t u}\,  dx dt.
\]

Note that $Q_2$ is a second order space operator with compact support, so one can integrate by parts to
express $IQ[g_K, S, E]$ as a quadratic form bounded by first order derivatives of $u$.

We have
\[
IQ[g_K, S, E] = -\Re \int_{\M_{[\tt_0,\tt_1]}} (\Box_{g_K} u) \overline{(S+E) u}\, dV_K + \left. BDR_2^{\mathbf K}[u, S, E]\right|_{t=\tt_0}^{t = \tt_1},
\]
where $BDR_2^{\mathbf K}[u, S, E]$ also satisfies \eqref{pdobdry}.

We now define
\[
IQ[g_K, X, q, m, S, E] = \int_{\M_{[\tt_0,\tt_1]}} Q[g_K, X, q, m] dV_{\mathbf K} + IQ[g_K, S, E].
\]

  We thus have
 \begin{equation}\begin{split}
   IQ[g_K, X, q, m, S, E]  = &
   - \int_{\M_{[\tt_0,\tt_1]}} (\Box_{g_K} u) (C\partial_t + X + S + q + E) u \, dV_{\mathbf K} \\ &
   - \left. BDR^{\mathbf K}[u]\right|_{t=\tt_0}^{t = \tt_1} - \left. BDR^{\mathbf K}[u] \right|_{r=r_e}. \end{split}
   \label{intdivk}\end{equation}

 The boundary terms satisfy
 \[
\left. BDR^{\mathbf K}[u]\right|_{t = \tt_i} = \left. BDR^{\mathbf S}[u]\right|_{t = \tt_i} + \left. BDR_2^{\mathbf K}[u, S, E]\right|_{t = \tt_i},
 \]
 and thus, due to \eqref{bdrs} and \eqref{pdobdry} we have
 \begin{equation}\label{bdrposk}
     BDR^{\mathbf K}[u]\big|_{t = \tt_i} \approx \| \nabla u(\tt_i)\|_{L^2}^2\qquad\text{and}\qquad
      BDR^{\mathbf K}[u]\big|_{r=r_e} =  BDR^{\mathbf S}[u]\big|_{r=r_e} \approx \| u\|_{H^1(\Sigma^+_{[\tt_0,\tt_1]})}^2.
 \end{equation}

 The main result of \cite{TT} is that, for a suitable choice of $S$ and $E$, one has 
  \begin{equation}\label{Kerrbd}
  IQ[g_K, X, q, m, S, E] \geq C\| u\|_{LE^1_{{\mathbf K},w}[\tt_0,
   \tt_1]}^2 -  a \|D_t u\|_{L_t^2 H^{-1}_{comp}}^2 ,
 \end{equation}
 where the last term on the right represents the $H^{-1}$ norm of $D_t u$ in a compact region in $r$ (precisely, a neighborhood of $3M$), and
\begin{equation}\label{LEKdef}
   \|u\|_{LE_{\mathbf K,w}^1}^2 \!\!=  \|\chi(D_t - \tau_2(D,x))\chi u\|_{L^2_{\varrho_1}}^2\!
+ \|\chi (D_t - \tau_1(D,x))\chi u\|_{L^2_{\varrho_2}}^2 \!
 +
 \|r^{-1}\partial_r u\|_{L^2}^2 + \|r^{-2} u\|_{L^2}^2 +
   \|(1-\chi^2) r^{-1} \nabla u\|_{L^2}^2.
\end{equation}

The estimate \eqref{Kerrbd}, combined with the elliptic estimate (see p. 38 of \cite{TT})
\begin{equation}\label{Dtell}
\|D_t u\|_{H^{-1}_{comp}}^2 \lesssim  \|u\|_{L^2_{comp}}^2 + \| \Box_{g_K}
u\|_{LEK^*}^2 + E[u](\tt_0)+ E[u](\tt_1),
\end{equation}
immediately yield Theorem~\ref{Kerr}.

\subsubsection{The Weyl calculus}
Let $S^m$ be the class of space symbols $s(t,x,\xi)$ of order $m$ depending on  $t$
that satisfy $|\pa_t^k \pa_x^\alpha \pa_\xi^\beta s(t,x,\xi)|\!\leq \!\langle\xi\rangle^{m-|\beta|}$. The Weyl quantization is the pseudo differential operator of order $m$:
\beq
(s^w u)(t,x)=\frac{1}{(2\pi)^3}\int\int s(t,\tfrac{x+y}{2},\xi) e^{i(x-y)\cdot \xi}\, u(t,y)\, dy \,d\xi.
\eq

We will work with symbols that are a priori defined only in a small neighborhood of the trapped set, and which are applied to functions with support in $I_{ps}$. In order to make sense of the Weyl quantization in this case, we consider $\chi_0$ a cutoff that is identically $1$ on $I_{ps}$ and supported in a slightly larger neighborhood, and redefine
\beq
s^w u = (\chi_0 s)^w u
\eq

We also consider space time symbols that are polynomials in $\tau\!=\!D_t$ with coefficients $s_k(t,x,\xi)\!\in\! S^k$:
\beq\label{eq:polynomialspacetimesymbols}
s(t,x,\tau,\xi)=s_m(t,x,\xi) +s_{m-1}(t,x,\xi) \tau +\dots +s_{m-k}(t,x,\xi)\tau^{k}
\!\in S^{m,k}\!\!=S^m\!\!+\tau S^{m-1}\!\!+\dots+\tau^k S^{m-k}\!.
\eq
The Weyl quantization of a real symbol is self adjoint with respect to the complex inner product. In particular
the Weyl quantization of the symbol $X^j\xi_j$ is the operator $X^j D_j + (D_j X^j)/2 $, where $D_j=i^{-1}\pa_j$, and the Weyl quantization of the symbol $a^{jk}\xi_j\xi_k$, with $a^{jk}\!\!=a^{kj}\!$, is the operator
$a^{jk}\!D_{j} D_k\! +\!(D_{j} a^{jk})D_k\! +\!(D_{j} D_k a^{jk})/4$, see \cite{Taylor}.
 In order to make this true also for space time symbols $X^\alpha \xi_\alpha $ and
 $a^{\alpha\beta}\xi_\alpha \xi_\beta$, with $\xi_0=\tau$ we define
 \beq
s^w=s_m^w +s_{m-1}^w D_t + (D_t s_{m-1}^w)/2 +s_{m-2}^w D_t^2+(D_t s_{m-2}^w) D_t
+(D_t^2 s_{m-2}^w)/4 +\dots .
\eq
With this definition the Weyl quantization of $\!\sqrt{\!|g|}g^{\alpha\beta}\xi_\alpha\xi_\beta$ is related the operator $\widehat{\Box}_g u=\pa_\alpha (\!\sqrt{\!|g|}g^{\alpha\beta} \pa_\beta u)$ by
\beq
-\bigtwo(\sqrt{|g|}g^{\alpha\beta}\xi_\alpha\xi_\beta\bigtwo)^w=\widehat{\Box}_g+\bigtwo(\pa_\alpha\pa_\beta (\sqrt{|g|}g^{\alpha\beta})\bigtwo)/4 .
\eq

If $s$ is a symbol of order $m$ and $p$ a symbol of order $n$ then the commutator
$[\,p^w,s^w]=p^w\, s^w-s^w\,p^w$ is an operator of order $m+n-1$ which to principal order is given
is $\frac{1}{2i}\{p,s\}^w$, the Poisson bracket \eqref{eq:poissonone} of the symbols. Moreover
$[\,p^w,s^w]-\frac{1}{2i}\{p,s\}^w$ is an operator of order $m+n-3$. With the above definitions this is also true
for space time symbols of the form \eqref{eq:polynomialspacetimesymbols} and the space time Poisson bracket.

Finally, recall that the Poisson bracket is invariant under changes of coordinates in the sense
that
\beq
\{\overline{p},\overline{s}\}(y,\eta)=\{p,s\}\big(x(y),\xi(y,\eta) \big),\qquad\text{if}\quad \overline{p}(y,\eta)=p\big(x(y),\xi(y,\eta) \big),\quad\text{where}\quad \xi_\alpha(y,\eta)=\,\eta_a\pa y^a\!/\pa x^\alpha\! .
\eq

This will allow us to work in polar coordinates later on.

\subsubsection{The construction of the operators}
Let us recall how the operators $S$ and $E$ were chosen. Consider the symbol
\[
 s(r,\tau,\xi,\Phi)=i r^{-1}b(r)\big(r-r_a(\Phi/\tau)\big)\xi_r.
\]
which coincides with the principal symbol for $X$ when $a=0$ and vanishes on the trapped set. Using \eqref{xidota}, one can compute the Poisson bracket \eqref{eq:poissonone}
\[\begin{split}
\frac{1}{i}  \{ \rho^2 p_K,s\} = & 2r^{-1} b(r) R_a(r,\tau,\Phi)\Delta^{-2}(r-r_a(\Phi/\tau)) \\ & + \left[2\Delta\partial_{r}\left(r^{-1} b(r)\big(r-r_a(\Phi/\tau)\big)\right) - 2(r-M)
  r^{-1} b(r) \big(r-r_a(\Phi/\tau)\big)  \right] \xi_r^2.
\end{split}\]

Since $r_a(\tau,\Phi)$ is the unique zero of $R(r,\tau,\Phi)$ near
$r=3M$ and is close to $3M$, it follows that we can write
\begin{equation}
\frac{1}{2i} \{ \rho^2 p_K,s\} =
\alpha^2(r,\tau,\Phi) \big(r-r_a(\Phi/\tau)\big)^2 \tau^2 +
\beta^2(r,\tau,\Phi)\xi_r^2,
\label{pbs}\end{equation}
where $\alpha, \beta \in S^0_{hom} $ are positive symbols.

The drawback of this computation is that $s$ is not a polynomial in $\tau$, and thus we cannot integrate by parts in time. Instead, we write $s$ as a linear function of $\tau$ plus a smooth function of $\tau$ times $(\tau-\ttau_1)(\tau-\ttau_2)$:
\begin{equation}\label{XKdef}
\frac{1}i s = b(r)\bigthree(1-\frac{3M}r\bigthree)\xi_r  +  \ts_1(r,\xi_r, \theta,\Theta,\Phi)+
\ts_0(r,\xi, \theta,\Theta,\Phi) \tau+
a  h(\tau, r,\xi, \theta,\Theta,\Phi) p_K,
\end{equation}
with $\ts_1 \in aS^1_{hom}$, $\ts_0 \in aS^0_{hom}$, and $h$ a homogeneous symbol of spacetime.

We now define
and $\ts_K$ by
\begin{equation}
\frac1i \ts_K = b(r)\bigthree(1-\frac{3M}r\bigthree) \xi_r + \ts_1 + \ts_0\tau .
\end{equation}
Then $\frac1{2i}\{\rho^2 p_K, \ts_K\}$ is a polynomial of degree $3$ in $\tau$ which coincides with \eqref{pbs}
when $\tau=\ttau_1$ or $\tau=\ttau_2$, since $\{\rho^2 p_K, hp_K\}$ is a multiple of $p_K$.
It follows that there are $\tilde{\gamma}_0,\tilde{\gamma}_1,\tilde{\gamma}_2\in S^0_{hom}$, $\tilde{\gamma}_{-1}\in a S^{-1}_{hom}$ such that
\begin{equation}\label{eq:sumsquarekerr}
\frac{1}{2i}\{\rho^2 p_K, \ts_K\}=\tilde{\gamma}_2 (\tau-\ttau_1)^2 +\tilde{\gamma}_1(\tau-\ttau_2)^2 +\big( \tilde{\gamma}_0 +\tilde{\gamma}_{-1}\tau\big) (\tau-\ttau_1)(\tau-\ttau_2).
\end{equation}
Using that this has to agree with \eqref{pbs} when $\tau=\ttau_i$ we get that
\begin{equation}
\tilde{\gamma}_i=\frac{\alpha_i^2}{4} + \frac{\beta_i^2 \xi_r^2}{(\ttau_1-\ttau_2)^2},
\end{equation}
where
\[
\alpha_{i} =\frac{2|\ttau_{i}|}{\ttau_{1} -\ttau_{2}}
 \alpha(r,\ttau_{i},\Phi) \big(r - r_a(\Phi/\ttau_{i})\big) \in S^0_{hom},
\qquad \beta_{i} = \beta(r,\ttau_{i},\Phi).
\]
Moreover, since \eqref{eq:sumsquarekerr} has to agree with \eqref{eq:sumsquareschwarz} up to terms of size $a$ we can determine $\tilde{\gamma}_0$ and $\tilde{\gamma}_{-1}$ modulo terms of size $a$ by equating powers of $\tau^2$ respectively $\tau^3$ which gives:
\begin{equation}\label{eq:sumsquarekerr2}
\frac{1}{2}\{\rho^2 p_K, \ts_K\}=\tilde{\gamma}_2 (\tau-\ttau_1)^2 +\tilde{\gamma}_1(\tau-\ttau_2)^2 +\big(\alpha_S(r)^2-\tilde{\gamma}_2-\tilde{\gamma}_1\big) (\tau-\ttau_1)(\tau-\ttau_2)+\big( e_0 +e_{-1}\tau\big) (\tau-\ttau_1)(\tau-\ttau_2),
\end{equation}
where
$e=e_0+e_{-1} \tau\in a (S^0_{hom}+\tau S^{-1}_{hom})$.
We would like to show that
$ \{ \rho^2 p_K, \ts_K \}/2 + \tilde{q}_S p_K $ is a sum of squares  plus a small multiple
  of $p_K$ of size $a$.
This follows from rewriting
\begin{multline}
\alpha_2^2(\tau-\ttau_1)^2+\alpha_1^2(\tau-\ttau_2)^2
=\nu\big(\alpha_1 (\tau-\ttau_2)
  -\alpha_2(\tau-\ttau_1)\big)^2   + (1-\nu)\big(\alpha_1 (\tau-\ttau_2)
  +\alpha_2(\tau-\ttau_1)\big)^2\\- 2(1-2\nu)\alpha_1\alpha_2 (\tau-\ttau_1)(\tau-\ttau_2),
\end{multline}
and comparing with \eqref{sumsqS1}.
One can now prove the following, which is essentially Lemma 4.3 in \cite{TT}:

\begin{lemma}\label{TTmain}
With notation as above there is an $\te \in a\bigl(S^0_{hom} + \tau S^{-1}_{hom}\bigr)$ so that, if we define
$\te_K = q_S + \te$
then close to the trapped set we have
\begin{equation}
\rho^2\bigtwo( \frac1{2i} \{ p_K, \ts_K \} + p_K \te_K\bigtwo) = \sum_{j=1}^8 \mu_j^2,
\label{sumsqK}\end{equation}
where (recall \eqref{nu} for the definition of $\nu$ and \eqref{eq:thelambdas} for the definition of $\lambda_i$):
\begin{equation}
\begin{split}
&  \mu_1^2 = \frac{1-\nu}4 \big(\alpha_1 (\tau-\ttau_2) + \alpha_2(\tau-\ttau_1)\big)^2,  \qquad \mu_2^2 = \frac12 (\beta_1^2 + \beta_2^2 - Ca)\xi_r^2, \\ &
 \mu_{3,4,5}^2 =  \frac{\lambda_{1,2,3}^2}{\lambda^2+(r^2-2rM) \xi_r^2} \frac{r-2M}{r^3}\frac{\nu}4 (\alpha_1 (\tau-\ttau_2)
  -\alpha_2(\tau-\ttau_1))^2 ,\\ &  \mu_6^2 =
  \frac{(r^2-2rM) \xi_r^2}{\lambda^2+(r^2-2rM) \xi_r^2}
  \frac{\nu}4 (\alpha_1 (\tau-\ttau_2)
  -\alpha_2(\tau-\ttau_1))^2 ,\\ &
  \mu_7^2 = \frac{(Ca - \beta_2^2+\beta_1^2) (\tau-\ttau_2)^2} {2(\ttau_1-\ttau_2)^2}\xi_r^2 ,\qquad
    \mu_8^2 = \frac{ (Ca-\beta_1^2+\beta_2^2)
    (\tau-\ttau_1)^2}{2(\ttau_1-\ttau_2)^2}\xi_r^2,
\end{split}\label{mudef}\end{equation}
and $C$ is a constant so that all three terms in the last line are nonnegative.
\end{lemma}

 We remark that when $a=0$ we have that $\alpha_1 = \alpha_2 = \alpha_S$ and $\beta_1=\beta_2=\beta_S$, so \eqref{sumsqK} is identical to \eqref{sumsqS1}. Moreover, $\mu_j$ are symbols of differential operators for $a=0$.

 In order to use the classical pseudodifferential calculus, we need to remove the singularity at zero frequency. Let $\chi_{> 1}(x)$ be a smooth cutoff that equals $1$ when $x\geq 1$ and $0$ when $x \ll 1$ and define
 \[
 s_i = \chi_{> 1} (|\xi|)^2 \ts_i, \qquad e_i = \chi_{> 1} (|\xi|)^2 \te_i , \qquad \tau_i := \chi_{> 1} (|\xi|) \ttau_i,
 \]
 \[
 s_K = b(r)\bigthree(1-\frac{3M}r\bigthree) \xi_r + s_1 + s_0\tau, \quad e_K = q_S + e_0 + e_{-1}\tau.
 \]
 The errors we make are smoothing and small, $O(a)$, so they can be easily dealt with.

 The operator $S$ is now defined as
\[
S = iS_1 + S_0\partial_t, \qquad S_j = \chi  s_j^w \chi.
\]

The operator $E$ is defined similarly, with the extra twist that \eqref{nodt3} is satisfied. We first note that, by the Weyl calculus,
\[
[g_K^{tt}, S_0] +  g_K^{tt} e_{-1}^w  + e_{-1}^w g_K^{tt} \in OPS^{-3},
\]
since its principal part must be $0$ by Lemma~\ref{TTmain}.

We can now define
\[
E = E_0 + \frac{1}i E_{-1} \partial_t, \qquad E_0 = \chi e_0^w \chi, \qquad E_{-1} = e_{-1}^w - e^w_{aux},
\]
where the operator $e^w_{aux}$ is chosen so that
\[
g^{tt} e^w_{aux} +  e^w_{aux}g^{tt} =  [g_K^{tt}, S_0] +  g_K^{tt}  e_{-1}^w  + e_{-1}^w g_K^{tt}.
\]
This is possible since the coefficient $g_K^{tt}$ of $\tau^2$ in $p_K$ is
a scalar function which is nonzero near $r = 3M$. (In fact, if $K(x,y)$ is the kernel of the operator on the right hand side then the operator $e^w_{aux}$ with kernel $K(x,y)/(g^{tt}(x)+g^{tt}(y))$ solves the equation.) Also as defined
$e^w_{aux} \in OPS^{-3}$ (so the principal symbol calculation does not change), and has kernel supported near $r = 3M$.

For the convenience of the reader, we also give the proof of \eqref{Kerrbd}. We first note that,
away from the trapped set, we can estimate
\begin{equation}
\int_{\M{[\tt_0,\tt_1]}} (1-\chi^2) Q[g_K, X, q, m] dV_{\mathbf K}  \gtrsim \int_{\M{[\tt_0,\tt_1]}}
(1- \chi^2) \big( r^{-2} |\pa u|^2 + r^{-4}|u|^2\big)\,  d V_K.
\label{mekout}\end{equation}

For simplicity, for any operator $S = \sum_{j=0}^2 S_{m+2-j} D_t^{j}\in S^{m,2}$ we will define the integrated bilinear form
\[
IQ[S, u] = \int_{\M{[\tt_0,\tt_1]}} S_{m+2} u \cdot \overline {u} + 2\Re S_{m+1} u D_t u \cdot \overline {u} + S_m D_t u \cdot \overline {D_t u} \, dV_K
\]

We now write for the microlocal part
\[
IQ[g_K, S, E] = IQ^K_{princ}[g_K, S, E]
+  IQ^K_{aux}[g_K, S, E],
\]
where the main component is given by
\begin{equation}
IQ^K_{princ}[g_K, S, E]  = IQ[\chi \bigthree( \frac{1}{2i}
  \{p_K,s\} + p_K e \bigthree)^w \chi, u]
\label{qksform}
\end{equation}

The remainder can be controlled by a small constant times norms we control in the Schwarzschild case, and the energy away from the trapped set:
\begin{lemma}
\begin{equation}
|IQ^K_{aux}[g_K, S, E]| \lesssim a\bigtwo(\| r^{-2}u\|_{L^2[\tt_0,\tt_1] L^2}^2+
 \|D_t u\|_{L^2[\tt_0,\tt_1]H^{-1}_{comp}}^2+\|\chi_1 \pa u\|_{L^2[\tt_0,\tt_1] L^2}^2\bigtwo).
\label{iqaux}\end{equation}
\end{lemma}
where $\chi_1$ is a cutoff that is identically $1$ on the support of $\chi^{\prime}$
\begin{proof}
 We have that
\[
\chi \bigthree( \frac{1}{2i}
  \{p_K,s\} + p_K e \bigthree)^w \chi -
\bigthree( \frac{1}{i} \chi s \{p_K,\chi\}\bigthree)^w \in aOPS^{0,2}.
\]
Since
\[
\Bigl|\int_{\M{[\tt_0,\tt_1]}} OPS^{-2} D_t u \cdot\overline{D_t u} \, dV_K \Bigr|
\lesssim \|OPS^{-2} D_t u\|_{L^2[\tt_0, \tt_1] H^{1}_{comp}} \|D_t u\|_{L^2[\tt_0, \tt_1] H^{-1}_{comp}}
 \lesssim  \|D_t u\|_{L^2[\tt_0, \tt_1] H^{-1}_{comp}}^2
\]
 we obtain
\[
|IQ^K_{aux}[g_K, S, E]| \lesssim \bigl|IQ\bigl[\bigthree( \frac{1}{i} \chi s \{p_K,\chi\}\bigthree)^w, u\bigr]\bigr| + a\bigtwo(\| u\|_{L^2[\tt_0,\tt_1]L^2_{comp}}^2+ \|D_t u\|_{H^{-1}_{comp}}^2\bigtwo)
\]

Moreover after factorizing the symbol and integrating by parts in space we easily obtain
\begin{equation}
\bigl|IQ\bigl[\bigthree( \frac{1}{i} \chi s \{p_K,\chi\}\bigthree)^w, u\bigr]\bigr|
\lesssim a \left(\|\chi_1 \pa u\|_{L^2 L^2[\tt_0,\tt_1]}^2 + \|u\|_{L^2[\tt_0, \tt_1] L^2_{comp}}^2 + \|D_t u\|_{L^2[\tt_0, \tt_1] H^{-1}_{comp}}^2\right)
\end{equation}
which finishes the proof of the lemma.
\end{proof}

We now define
\[
M_k = \chi \mu_k(0)(x,D) + \nu_k^w \chi, \quad \nu_k := \chi_{> 1} (|\xi|) \left(\mu_k(a) -\mu_k(0)\right) \in a S^{1, 1}
\]

Note that, since $\chi$ appears in different places in the expression $LHS\eqref{symbMk}$, we need to also carefully define $M_k$ with the cutoff in different places, otherwise the lower order errors could be large. We also defined the operators so that they equal to their counterparts in Schwarzschild when $a=0$.

We first prove the following:
\begin{lemma} We have
\begin{multline}\label{symbMk}
\int_{\M{[\tt_0,\tt_1]}}\chi^2  Q[g_K, X, q, m] dV_{\mathbf K} + IQ^K_{princ}[g_K, S, E] \!\geq \!\!\int_{\!\M{[\tt_0,\tt_1\!]}}  \!{\tsum}_j |M_j u|^2\! + q^{K,0}\! \chi^2 u^2 dV_K \\ - O(a)\left(\|\chi_1 \pa u\|_{L^2 L^2[\tt_0,\tt_1]}^2+\| u\|_{L^2[\tt_0, \tt_1]L^2_{comp}}^2 + \|D_t u\|_{L^2[\tt_0, \tt_1]H^{-1}_{comp}}^2\right) .
\end{multline}
\end{lemma}
\begin{proof}
Note that for $X$ and $q$ we have near the trapped set
\[
Q[g_K, X, q, m] = q_K^{\alpha\beta}\pa_{\alpha} u\, \pa_{\beta} u + q^{K,0} u^2,
\]
where
\[
q^{K,\alpha \beta} \eta_\alpha \eta_\beta = \frac{1}{2i}\{ p_K,X\}+ q\,  p_K, \qquad
q^{K,0} = -\frac12 \Box_{g_K} q > 0 .
\]
We now write
\[
\chi^2 q_K^{\alpha\beta}\pa_{\alpha} u\, \pa_{\beta} u = \sum_j |\chi\mu_j (0)(x,D) u|^2 + \chi^2 (q_K^{\alpha\beta}-q_S^{\alpha\beta})\pa_{\alpha} u\, \pa_{\beta} u .
\]
Since $q_K^{\alpha\beta}-q_S^{\alpha\beta}=O(a)$ we obtain
\[\begin{split}
\!\!\int_{\!\M{[\tt_0,\tt_1\!]}} (q_K^{\alpha\beta}-q_S^{\alpha\beta})\pa_{\alpha} u\, \pa_{\beta} u \, dV_K - \Re \, IQ\left[ \chi \left(\frac1{2i}\{p_K-p_S, X\} + (p_K-p_S) q\right)^w \!\chi, u\right] \\ \lesssim a\left(\| u\|_{L^2[\tt_0, \tt_1]L^2_{comp}}^2 + \|D_t u\|_{L^2[\tt_0, \tt_1]H^{-1}_{comp}}^2\right) .
\end{split}\]
On the other hand, since
\[\mu_j(a)^2 - (\mu_j(0)+\nu_j)^2 \in a S^{1, 1}\]
we have by Lemma~\ref{TTmain}
\[
\frac1{2i}\{p_K, s\}+ p_K e = \sum_j (\mu_j(0)+\nu_j)^2  - \frac1{2i}\{p_K, X\}- p_K q = \sum_j \nu_j^2 + 2\nu_j\mu_j(0) - \frac1{2i}\{p_K-p_S, X\}- (p_K-p_S) q
\]
modulo a term in $a S^{1, 1}$.

We thus obtain
\begin{multline}
LHS\eqref{symbMk} \!\geq \!\!\int_{\!\M{[\tt_0,\tt_1\!]}} \!{\tsum}_j |\chi\mu_j(0)(x,D) u|^2\!  + q^{K,0}\! \chi^2 u^2 dV_K + IQ[\chi (\nu_j^2 + 2\nu_j\mu_j(0))^w \chi, u] \\
- O(a) \left(\| u\|_{L^2[\tt_0, \tt_1]L^2_{comp}}^2 + \|D_t u\|_{L^2[\tt_0, \tt_1]H^{-1}_{comp}}^2\right).
\end{multline}

By the Weyl calculus we have
\begin{multline}
IQ [\chi \left(\nu_j^2 + 2\nu_j\mu_j(0)\right)^w \chi, u] - IQ [\chi \left((\nu_j^w)^2 + 2\nu_j^w\mu_j^w(0)\right)\chi, u]  \\ \lesssim a\left(\| u\|_{L^2[\tt_0, \tt_1]L^2_{comp}}^2 + \|D_t u\|_{L^2[\tt_0, \tt_1]H^{-1}_{comp}}^2\right)
\end{multline}
Since $\mu_j(0)$ is a differential operator of first order, we have that $\mu_j(0)(x,D)-\mu_j^w(0)$ is skew-adjoint, which implies
\[
\Re \, IQ [\chi \nu_j^w\mu_j^w(0) \chi, u] = \Re \, IQ [\chi \nu_j^w\mu_j(0)(x,D) \chi, u]
\]
and thus
\[
IQ [\chi \left((\nu_j^w)^2 + 2\nu_j^w\mu_j^w(0)\right) \chi, u] = \int_{\!\M{[\tt_0,\tt_1\!]}} |\chi \nu_j^w \chi u|^2 + 2\Re \, \left(\mu_j(0)(x,D)\chi u \cdot \overline{\nu_j^w \!\chi u}\right) dV_K
\]
On the other hand we compute
\begin{multline}
\!\!\int_{\!\M{[\tt_0,\tt_1\!]}} |M_j u|^2 dV_K
= \!\!\int_{\!\M{[\tt_0,\tt_1\!]}} \chi^2 |\mu_j(0)(x,D)u|^2 + 2\Re\left(\chi \mu_j(0)(x,D)u \cdot \overline{\nu_j^w \!\chi u}\right) + |\chi \nu_j^w \chi u|^2  dV_K \\
= \!\!\int_{\!\M{[\tt_0,\tt_1\!]}} \chi^2 |\mu_j(0)(x,D)u|^2 + 2\Re\left(\mu_j(0)(x,D) \chi u \cdot \overline{\nu_j^w \!\chi u}\right) + |\chi \nu_j^w \chi u|^2 - 2\Re\left([\mu_j(0)(x,D), \chi] u \cdot \overline{\nu_j^w \!\chi u}\right) dV_K
\end{multline}
Since $\nu_j\in a(S^1+\tau S^0)$,
\[
\|\nu_j^w \chi \chi_1 u\|_{L^2}^2 \lesssim a^2 \|\pa_{t,x}^{\leq 1} \chi \chi_1 u\|_{L^2}^2 \lesssim  \|\chi_1 \pa_{t,x} u\|_{L^2}^2 + \|\chi u\|_{L^2}^2
\]
Finally, since $\chi_1=1$ on the support of $\chi^\prime$,
\[
\!\!\int_{\!\M{[\tt_0,\tt_1\!]}} \left| [\mu_j(0)(x,D), \chi] u \cdot \overline{\nu_j^w \!\chi u}\right| dV_K \lesssim a\left(\|\nu_j^w \chi \chi_1 u\|_{L^2[\tt_0, \tt_1]L^2}^2 + \|\chi^{\prime}  u\|_{L^2[\tt_0, \tt_1]L^2}^2 \right) \lesssim \|\chi_1 \pa^{\leq 1} u\|_{L^2[\tt_0, \tt_1]L^2}^2
\]

\end{proof}

In order to absorb the errors created when the derivative falls on the cutoff, we show
 \begin{lemma} \label{chi1}
 We have
 \beq
 \int_{\!\M{[\tt_0,\tt_1\!]}} (1-\chi^2) Q[g_K, X, q, m] + \!{\tsum}_j |M_j u|^2 dV_K \gtrsim \|\chi_1\pa u\|^2_{L^2[\tt_0,\tt_1\!]L^2} - \left(\| u\|_{L^2[\tt_0, \tt_1]L^2_{comp}}^2 + \|D_t u\|_{L^2[\tt_0, \tt_1]H^{-1}_{comp}}^2\right)
 \eq
 \end{lemma}
 \begin{proof}

We can express the symbol of any derivative in $I^{\prime}$ as a linear combination
\[
\xi_{\alpha} = \sum_k \gamma_k^0 \mu_k
\]
where $\gamma_k^0$ are symbols of order $0$.  We can thus write
\[
\chi_1 \chi D_{\alpha} = \sum_k \gamma_k^0 (D, x) \chi_1 M_k u + (R_0+R_{-1} D_t)
\]
where $R_j\in OPS^j$ with compactly supported kernel. This implies
\[\begin{split}
& \|\chi_1 \chi \pa_{\alpha} u\|^2_{L^2[\tt_0,\tt_1\!]L^2} \lesssim  \sum_k \|\gamma_k^0 (D, x) \chi_1 M_k u\|^2_{L^2[\tt_0,\tt_1\!]L^2} +  \left(\| u\|_{L^2[\tt_0, \tt_1]L^2_{comp}}^2 + \|D_t u\|_{L^2[\tt_0, \tt_1]H^{-1}_{comp}}^2\right) \\ & \lesssim \sum_k \|M_k u\|^2_{L^2[\tt_0,\tt_1\!]L^2} + \left(\| u\|_{L^2[\tt_0, \tt_1]L^2_{comp}}^2 + \|D_t u\|_{L^2[\tt_0, \tt_1]H^{-1}_{comp}}^2\right)
\end{split}\]
We also have
\[
 \int_{\!\M{[\tt_0,\tt_1\!]}} (1-\chi^2) \chi_1^2 |\pa_{\alpha} u|^2 dV_K \lesssim \int_{\!\M{[\tt_0,\tt_1\!]}} (1-\chi^2) Q[g_K, X, q, m] dV_K
\]
The conclusion follows by adding the last two inequalities.
 \end{proof}

 We now prove that, on any fixed time slice:
 \begin{lemma} We have
\begin{equation}\label{Mklek}
\begin{split}
\sum_k \|\chi M_k u\|_{L^2}^2 & \gtrsim \|\varrho_1(D, x) \chi(D_t - \tau_2(D,x))\chi u\|_{L^2}^2
+ \|\varrho_2(D, x)\chi (D_t - \tau_1(D,x))\chi u\|_{L^2}^2 \\ & + \| \chi D_r u\|_{L^2}^2 - \| u\|_{L^2_{comp}}^2 - \|D_t  u\|_{H^{-1}_{comp}}^2.
\end{split}
\end{equation}
\end{lemma}
\begin{proof}
To prove \eqref{Mklek}, we notice that, by Lemma~\ref{TTmain}, the symbols $\varrho_1(\tau-\tau_2)$,
$\varrho_2(\tau-\tau_1)$ and $\xi_r$ can be recovered in an elliptic fashion from the principal symbols $\mu_k$ of $M_k$. We can thus write
\beq
\tvarrho_1(x,\xi) \big(\tau-\ttau_2(x,\xi)\big)
=\sum_{k=1,3,4,5}\tgamma_k^1(x,\xi) \big(\mu_k^1(x,\xi)+\mu_k^0(x,\xi) \tau\big),
\eq
where $\tgamma_1^i(x,\xi)$ are homogeneous symbols of order $0$. It follows from the composition formula for pseudo differential operators that
\beq
\varrho_1(D, x) \chi \big(D_t\!-\tau_2(D, x)\big) \chi
=\sum_{k=1,3,4,5}\gamma_1^k(D, x) \chi M_k
+ (R_0 + R_{-1}D_t),
\eq
where $\gamma_1^k\in S^0$ and $R_j \in OPS^j$ with compactly supported kernels. Since $\gamma_1^i(D, x)$ map $L^2$ to $L^2$ and $R_j$ map $L^2$ to $H^{-j}$ we obtain
\[
\| \varrho_1(D, x) \chi \big(D_t\!-\tau_2(D, x)\big) \chi u\|_{L^2}^2 \lesssim \sum_{k=1,3,4,5}\|\chi M_k u\|_{L^2}^2 + \| \chi u\|_{L^2}^2 + \|D_t \chi u\|_{H^{-1}}^2,
\]
which bounds the first term on the RHS of \eqref{Mklek}. The next two terms are bounded in a similar manner.
\end{proof}

Multiplying \eqref{Mklek} by a small (but $a$-independent) constant $c\ll \inf q^{K,0}$, integrating in time and using \eqref{symbMk} yields
\beq\begin{split}
& \int_{\M{[\tt_0,\tt_1]}}\chi^2  Q[g_K, X, q, m] dV_{\mathbf K} + IQ^K_{princ}[g_K, S, E] \\ & \geq C_1 \left(\|\varrho_1(D, x) \chi(D_t - \tau_2(D,x))\chi u\|_{L^2[\tt_0,\tt_1\!]L^2}^2
+ \|\varrho_2(D, x)\chi (D_t - \tau_1(D,x))\chi u\|_{L^2[\tt_0,\tt_1\!]L^2}^2  + \| \chi D_r u\|_{L^2[\tt_0,\tt_1\!]L^2}^2\right) \\ & - c\|D_t  u\|_{H^{-1}_{comp}}^2
\end{split}\label{mekin}
\eq

\eqref{Kerrbd} follows after adding \eqref{mekout} and \eqref{mekin}. One can now use cutoffs to improve the weights at infinity to the optimal ones, see for instance \cite{MMTT}. Finally, using Lemma~\ref{chi1} we can control $\|\chi_1 \pa u\|_{L^2[\tt_0,\tt_1\!]L^2}$, and thus also $\|\tilde{\chi}_0\pa u\|_{L^2[\tt_0,\tt_1\!]L^2}$. This concludes the proof of Theorem~\ref{Kerr}.

\subsection{Vanishing of symbols on the trapped set}

The rest of the section will be devoted to proving that the symbols $s_K$ and $e_K$ from Lemma~\ref{TTmain} have certain vanishing property on the trapped set. This is essential for the results in the next section, as it allows us to prove local energy estimates for metrics that decay slower than $t^{-1}$ near the trapped set to the Kerr metric.

The main result of the section is the following lemma:

\begin{lemma}\label{qK}
There exist symbols $s_i^\ell$, $s_{ij}^{\ell} \in S^{\ell}$ and smoothing symbols $r_i\in S^{-\infty}+\tau S^{-\infty}$ so that
\begin{align}\label{sKdec}
s_K &= s_{1}^0 \varrho_1(\tau-\tau_2) + s_{2}^0 \varrho_2(\tau-\tau_1) + s_{3}^0 \xi_r + r_1,\\
   \pa_{x^k} s_K &= s_{1k}^0 \varrho_1(\tau-\tau_2) + s_{2k}^0 \varrho_2(\tau-\tau_1) + (s_{3k}^0 + s_{4k}^{-1}\tau)\xi_r + r_2, \\
   \pa_{\xi_k} s_K &= s_{1k}^{-1} \varrho_1(\tau-\tau_2) + s_{2k}^{-1}\varrho_2(\tau-\tau_1) + (s_{3k}^{-1} + s_{4k}^{-2}\tau)\xi_r + r_3.
\end{align}

Moreover, there are symbols $e_i^\ell \in S^{\ell}$ and $r_i\in S^{-\infty}+\tau S^{-\infty}$ so that 
\begin{equation}\label{qKdec}
e_K = e_1^{-1} \varrho_1(\tau-\tau_2) + e_2^{-1} \varrho_2(\tau-\tau_1) + (e_3^{-1} + e_4^{-2}\tau) \xi_r + r.
\end{equation}
\end{lemma}

We remark that, since $s_K$ and $e_K$ were defined implicitly, this is not at all obvious.

\begin{proof}

 It is enough to prove the corresponding decomposition for the homogeneous counterparts $\ts_K$ and $\te_K$, since the errors we make are supported at small frequencies and thus are smoothing.

 We start by proving \eqref{sKdec}, which is easier, since $\ts_K$ can be computed explicitly without much effort.
Indeed, let $\tau=\ttau_1$ and $\tau=\ttau_2$ in \eqref{XKdef}. We are left with solving the system
\[
r^{-1}b(r)\tvarrho_i\xi_r = r^{-1} b(r) (r-3M) \xi_r + (\ts_1 + \ts_0\ttau_i), \quad i=1,2,
\]
which has the solutions
\[
\ts_0 = r^{-1}b(r) \xi_r \frac{\tvarrho_1-\tvarrho_2}{\ttau_1-\ttau_2}, \quad \ts_1 = r^{-1}b(r)\xi_r\Bigl(\tvarrho_1 -(r-3M) - \frac{(\tvarrho_1-\tvarrho_2)\ttau_1}{\ttau_1-\ttau_2}\Bigr).
\]
 Thus
\begin{equation}\label{SK}\begin{split}
\ts_K = r^{-1} b(r) (r-3M) \xi_r + \ts_1 + \ts_0\tau = r^{-1} b(r) \xi_r \frac{\tvarrho_1(\tau-\ttau_2)-\tvarrho_2(\tau-\ttau_1)}{\ttau_1-\ttau_2}
\end{split}\end{equation}
which immediately yields \eqref{sKdec}.

Unfortunately it seems difficult to compute $\te_K$ explicitly; instead, we will use the decomposition in Lemma~\ref{TTmain}. Recall that we have
\begin{equation}\label{MT}
\frac1{2i} \{ p_K, \ts_K \} + p_K \te_K = \frac{\nu}4 \bigl(\alpha_1(\tau-\ttau_2) + \alpha_2(\tau-\ttau_1)\bigr)^2 + \frac{1-\nu}4 \bigl(\alpha_1(\tau-\ttau_2) - \alpha_2(\tau-\ttau_1)\bigr)^2 + A\xi_r^2,
\end{equation}
where
\[
0< \nu(r) <1,
\]
is a smooth function,
\[ \quad \alpha_i = \tgamma_i \tvarrho_i, \quad \tgamma_i = \frac{2|\ttau_i|}{\ttau_1-\ttau_2}\alpha(r, \ttau_i, \Phi), \]
are elliptic multiples of $\tvarrho_i$, and $A \in S_{hom}^{0}+ \tau S_{hom}^{-1} + \tau^2 S_{hom}^{-2}$.

The conclusion will follow if we can find $d_1$ and $d_2$ so that
\begin{equation}\label{xir0}
\te_K = d_1 \tvarrho_1(\tau-\ttau_2) + d_2 \tvarrho_2(\tau-\ttau_1)
\end{equation}
when $\xi_r = 0$. Indeed, by Taylor's formula we can now write
\[
\te_K = \te_K \left|_{\xi_r=0} \right. + \tilde d \xi_r,
\]
for some $\tilde d \in S_{hom}^{-1} + a\tau S_{hom}^{-2}$.

 In order to prove \eqref{xir0}, we will  let $\xi_r = 0$ in \eqref{MT}, and assume until the end of the proof that all functions are evaluated at $\xi_r=0$.  We obtain
\begin{equation}\label{MT0}
(\partial_r p_K) \Bigl(\frac{\tvarrho_1(\tau-\ttau_2)}{\ttau_1-\ttau_2} + \frac{\tvarrho_2(\tau-\ttau_1)}{\ttau_2-\ttau_1} \Bigr) + p_K \te_K = \frac14 \bigl(\alpha_1^2(\tau-\ttau_2)^2 + \alpha_2^2(\tau-\ttau_1)^2\bigr) + \frac{2\nu-1}2 \alpha_1\alpha_2 (\tau-\ttau_1)(\tau-\ttau_2).
\end{equation}
 Since $p_K = g_K^{tt} (\tau-\ttau_1)(\tau-\ttau_2)$ we have that
\[
\partial_r p_K = (\partial_r g_K^{tt}) (\tau-\ttau_1)(\tau-\ttau_2) - g_K^{tt} (\partial_r \ttau_1) (\tau-\ttau_2) - g_K^{tt} (\partial_r \ttau_2) (\tau-\ttau_1).
\]
Thus letting $\tau=\ttau_1$ and $\tau=\ttau_2$ respectively in \eqref{MT0} we obtain
\begin{equation}\label{MT1}
- g_K^{tt}\partial_r \ttau_1 = \frac{1}{4}\tvarrho_1 \tgamma_1^2 (\ttau_1-\ttau_2), \qquad -g_K^{tt} \partial_r \ttau_2 =\frac{1}{4} \tvarrho_2 \tgamma_2^2 (\ttau_2-\ttau_1),
\end{equation}
and so
\[
\partial_r p_K = (\partial_r g_K^{tt}) (\tau-\ttau_1)(\tau-\ttau_2)+\frac{1}{4}\tvarrho_1\tgamma_1^2(\tau-\ttau_2)(\ttau_1-\ttau_2)
+\frac{1}{4}\tvarrho_2\tgamma_2^2(\tau-\ttau_1)(\ttau_2-\ttau_1).
\]

Write $\te_K = e_1 (\tau-\ttau_1) + e_2 (\tau-\ttau_2)$. Replacing in \eqref{MT0}, using \eqref{MT1} to cancel the square terms, and dividing by $(\tau-\ttau_1)(\tau-\ttau_2)$ yields
\[
(\partial_r g_K^{tt}) \Bigl(\frac{\tvarrho_1(\tau-\ttau_2)}{\ttau_1-\ttau_2} + \frac{\tvarrho_2(\tau-\ttau_1)}{\ttau_2-\ttau_1} \Bigr)-\frac{1}{4}\tvarrho_1\tvarrho_2 (\tgamma_1^2 + \tgamma_2^2) + g_K^{tt} (e_1 (\tau-\ttau_1) + e_2 (\tau-\ttau_2)) = \frac{2\nu-1}2 \alpha_1\alpha_2.
\]

We can now solve explicitly for $e_1$ and $e_2$. We obtain
\begin{align}
e_1 &= \frac{1}{g_K^{tt}(\ttau_2-\ttau_1)}\bigtwo[ \frac{2\nu-1}2 \tgamma_1\tgamma_2\tvarrho_1 \tvarrho_2
+\frac{1}{4}\tvarrho_1\tvarrho_2 (\tgamma_1^2 + \tgamma_2^2) - (\partial_r g_K^{tt})\tvarrho_2\bigtwo] := d_1 \tvarrho_2,\\
e_2 &= \frac{1}{g_K^{tt}(\ttau_1-\ttau_2)}\bigtwo[ \frac{2\nu-1}2  \tgamma_1\tgamma_2\tvarrho_1 \tvarrho_2
+\frac{1}{4}\tvarrho_1\tvarrho_2 (\tgamma_1^2 + \tgamma_2^2)  - (\partial_r g_K^{tt})\tvarrho_1\bigtwo]  := d_2 \tvarrho_1,
\end{align}
which proves \eqref{xir0}.

\end{proof}

\subsection{The operator quadratic form}

Let $P=D_{\alpha}(g^{\alpha\beta}D_{\beta})$ be symmetric with respect to $dxdt$. If $X$ and $q$ are a vector field and scalar function, we can define in analogy to \eqref{Qdef}
\beq\label{Phvec}
Q[P, X, q] = \left(-X(g^{\alpha\beta})  +g^{\alpha\gamma} \partial_\gamma X^\beta+g^{\beta\gamma}\partial_\gamma X^\alpha - g^{\alpha\beta} \pa_\gamma X^{\gamma} + g^{\alpha\beta} q\right)\pa_\alpha u\pa_\beta u -\frac12 (\pa_{\alpha} g^{\alpha\beta} \pa_\beta q) u^2.
\eq

Assume now that $S$ is a skew-adjoint and $E$ is a self-adjoint operator of the form
 \begin{equation}\label{gpdo}
 S = S_1+ S_0 \partial_t, \qquad E= E_0 + E_{-1} \partial_t.
 \end{equation}
  Here $S_j, E_j \in a OPS^j$ are real-valued, time-independent operators with kernels supported near $r=3M$, and moreover $S_1$ and $E_{-1}$ are self-adjoint, while $S_0$ and $E_0$ are skew-adjoint with respect to $dxdt$. We will also require that
\begin{equation}\label{nodt}
[g^{tt}, S_0] +  g^{tt} E_{-1}  + E_{-1} g^{tt} = 0.
\end{equation}

Then modulo boundary terms
\beq
 \int_{\tt_0}^{\tt_1}\int P u\cdot \overline{(Su+E u)}\, dx dt
=\int_{\tt_0}^{\tt_1}\int \frac12\big([P,S] + P E + E\,P\big)u\cdot\overline{u}\, \ dx dt+\text{BDR}\Big|_{\tt_0}^{\tt_1}.
\eq
where due to \eqref{nodt} there is no term containing two time derivatives on the boundary. In particular, we have
\beq\label{bdryP}
\Bigl| \text{BDR}_{\tt_i}\Bigr| \lesssim E[u](\tt_i).
\eq

The spacetime term can be written as
\[
\frac1{2}\bigtwo([P,S] + P E + E P\bigtwo)= P_2  + 2 P_1 D_t + P_0 D_t^2,
\]
where $Q_j \in OPS^j$ are selfadjoint pseudodifferential operators. Note that, due to \eqref{nodt}, there is no $D_t^3$ term.

When it comes to estimating, it will be more convenient to rewrite the space-time term as a quadratic form. We define
\[
\widehat{Q}[P,S,E]\!= \Re \Bigl(g^{\alpha\beta}[D_{\beta}, S]u \cdot \overline{D_{\alpha} u} + \frac12 [g^{\alpha\beta}, S] D_{\beta} u\cdot \overline{D_{\alpha} u} + \frac12 g^{\alpha\beta} E D_{\beta} u\cdot \overline{D_{\alpha} u} + \frac12 g^{\alpha\beta} D_{\beta} u\cdot \overline{D_{\alpha} Eu} \Bigr).
\]
It is easy to see that
\[
\int_{\tt_0}^{\tt_1}\int \frac12\big([P,S] + P E + E\,P\big)u\cdot\overline{u}\, \ dx dt = -\int_{\tt_0}^{\tt_1}\!\int\! \widehat{Q}[g,S,E]\, dx dt+\text{BDR}\Big|_{\tt_0}^{\tt_1},
\]
where the boundary terms again satisfy \eqref{bdryP}. Moreover, our definition agrees with the definition of $Q[g, X, q]$ from the previous section.

Indeed, let $P= \widehat{\Box}_{g} = \pa_{\alpha}(\sqrt{|g|}g^{\alpha\beta}\pa_{\beta})$, and $X=X^\alpha\pa_\alpha$. We write
\beq
X=S_X+E_X, \quad X=X^\alpha\pa_\alpha, \quad\text{where}\quad  S_X=X+\widehat{q}_X,\quad E_X=-\widehat{q}_X,\quad
\widehat{q}_X=\pa_\alpha{X}^\alpha\!/2,
 \eq
in order that $S_X$ be skew symmetric with respect to $dx dt$. Then
\begin{multline}
 2\widehat{Q}[\widehat{\Box}_{g},S_X,E_X]
=2\sqrt{g}\, \Re\bigtwo( g^{\alpha\beta} [{\pa}_\alpha,X]u\cdot\overline{{\pa}_\beta u}-\big( g^{\alpha\beta} (\widehat{q}_X D_\alpha u)\cdot\overline{{\pa}_\beta u}\big)- X (\sqrt{g}g^{\alpha\beta}){\pa}_\alpha u\cdot \overline{{\pa}_\beta u}\bigtwo)\\
=\sqrt{g}\,{\Re}\bigtwo(\big(-{\mathcal L}_X g^{\alpha\beta}-\div X\, g^{\alpha\beta}\big) {\pa}_\alpha u\cdot \overline{{\pa}_\beta u} \bigtwo)
 =2\sqrt{g}\,Q[g,X].
\end{multline}
Moreover, with the previous notation
\beq
\widehat{Q}[g,S_X, E_X+q]=\sqrt{g}\,Q[g,X,q].
\eq

\subsubsection{How to deal with $3$ time derivatives}
Finally, consider the operator $P$ with the additional assumption that $g^{00}=0$. In this case, \eqref{nodt3} is automatically satisfied. It will be convenient in this case to define the quadratic form so that there are no $D_{t}^2$ terms appearing. We can define
\beq\label{Phquad}
IQ[P, S, E] = \int_{\M[\tt_0, \tt_1]} \tilde{Q}[P, S] + \tilde{Q}[P, E] dxd\tt,
\eq
where
\beq\label{Squad}
 \tilde{Q}[P, S] = \Re \Bigl(g^{\alpha\beta}[D_{\beta}, S]u \cdot \overline{D_{\alpha} u} + \frac12 [g^{j \beta}, S] D_{j} u\cdot \overline{D_{\beta} u}\Bigr),
\eq
\beq\label{Equad}
\tilde{Q}[P, E]\!= \Re \Bigl(g^{\alpha j}  D_{\alpha} u\cdot \overline{D_{j} Eu} + g^{0j} D_{0} u\cdot \overline{D_{j} Eu} + ((D_0 g^{0j}) D_j u- (D_j g^{0j}) D_0 u)\cdot \overline{Eu} \Bigr).
\eq

Again we have
\beq
 \int_{\tt_0}^{\tt_1}\int P u\cdot \overline{(Su+E u)}\, dx dt = IQ[P, S, E] + \text{BDR}\Big|_{\tt_0}^{\tt_1},
\eq
where, under the additional assumption that $|g|\lesssim\epsilon$, we have
\beq\label{bdryPh}
\Bigl| \text{BDR}_{\tt_i}\Bigr| \lesssim \epsilon E[u](\tt_i).
\eq

\newsection{Norm estimates for the pseudodifferential operators and commutators}

In this section we provide the estimates that we will use in our main result. For a thorough treatment of symbols with low regularity, we refer the interested reader to \cite{Taylor2}.

\subsection{Estimates for operators with low regularity symbols}
We start with the following useful bound:
\begin{lemma}\label{L2bd}
Let $M$ be an operator with kernel
\[
K(x, y) = \int m(x, y) R (x, y, \xi) e^{i(x-y)\xi} d\xi .
\]
Here $R$ is of order $0$, smooth, and compactly supported in $I\times I$
\[
m(x, y) = m_1(x) + m_2(x, y),
\]
so that
\[
\|m\|_{\delta} := \|m_1\|_{L^{\infty}} + \sup_{x\neq y} \frac{|m_2(x,y)|}{|x-y|^{\delta}} <\infty,
\]
for some (small) $\delta>0$ on the support of $R$. Then $M: L^2(I)\to L^2(I)$, and
\begin{equation}\label{L2onederiv}
\|M\|_{L^2\to L^2} \leq C_R ||m||_{\delta}.
\end{equation}
Moreover, if also
\[
\|\pa_x m_1\|_{L^{\infty}} + \|\pa_x m_2\|_{\delta} <\infty,
\]
then $M: H^{-1}(I)\to H^{-1}(I)$, and
\begin{equation}\label{L2H-1}
\|M\|_{H^{-1}\to H^{-1}} \leq C_R (||m||_{\delta} + ||\pa_x m||_{\delta}).
\end{equation}
Here the constant $C_R$ only depends on a few derivatives of $R$, but does not depend on $m$.
\end{lemma}

\begin{proof}

We define
\[
K_j(x, y) = \int m_j(x, y) R (x, y, \xi) e^{i(x-y)\xi} d\xi, \qquad M_j u(x) = \int K_j(x, y) u(y) dy, \qquad j=1,2.
\]

We clearly have
\[
M_1 u(x) = m_1(x) \int R (x, y, \xi) e^{i(x-y)\xi} u(y) dyd\xi := m_1(x) Q_0, \qquad Q_0: L^2\to L^2,
\]
and thus
\[
\|M_1\|_{L^2\to L^2} \lesssim \|m_1\|_{L^{\infty}}.
\]

Moreover, by duality
\[
\|M_1 u\|_{H^{-1}(I)} = \sup_{\|v\|_{H^1(I)}=1} \la m_1 Q_0 u, v\ra \lesssim \|u\|_{H^{-1}} \sup_{\|v\|_{H^1(I)}=1} \|Q_0^* (\overline {m_1} v)\|_{H^1(I)} \lesssim \left(\|m_1\|_{L^{\infty}} + \|\pa_x m_1\|_{L^{\infty}}\right) \|u\|_{H^{-1}}.
\]

We now estimate $M_2$. Pick $i$ so $|x_i-y_i|>|x-y|/4$, say. Integrating by parts (in the sense of distributions) we get
 \[
K_2(x, y) = \int m_2(x, y) R (x, y, \xi) (x_i-y_i)^{-N} D_{\xi^i}^N e^{i(x-y)\xi} d\xi=\int m_2(x, y) \big(D_{\xi^i}^N R (x, y, \xi)\big) (x_i-y_i)^{-N} e^{i(x-y)\xi} d\xi,
\]
and since $|D_\xi^\alpha R(x,y,\xi)|\leq C_\alpha \langle\xi\rangle^{-|\alpha|}$ we have for any $N\geq 0$
\beq
|K_2(x,y)|\leq C_R \int \frac{|m(x, y)|\, d\xi}{(1+\langle \xi\rangle |x-y|)^N}.
\eq
If $N=n+1$, where $n$ is the dimension we obtain if we make the change of variables $y\to z=\langle \xi\rangle (x-y)$:
\beq
\int |K_2(x,y)|dy \leq C_R ||m||_{\delta} \int\int  \frac{|x-y|^\delta \, dy d\xi}{(1+\langle \xi\rangle |x-y|)^{n+1}}=C_R ||m||_{\delta} \int\int  \frac{|z|^\delta \, dz d\xi}{(1+ |z|)^{n+1}\langle \xi\rangle^{n+\delta}} \leq C_R ||m||_{\delta}.
\eq

A similar result follows for $\int |K_2(x,y)|dx$. \eqref{L2onederiv} now follows by Schur's lemma.

To prove \eqref{L2H-1} we again use duality. We have
\[
\|M_2 u\|_{H^{-1}} = \sup_{\|v\|_{H^1}=1} \la M_2 u, v\ra \lesssim \sup_{\|v\|_{H^1}=1} \|u\|_{H^{-1}} \|M_2^* v\|_{H^1}.
\]

We now compute
\begin{multline}
D_j M_2^* v(x)=\iint D_{x^j} \overline{\left(m_2(y,x) R(y, x, \xi)e^{i(x-y)\xi}\right)} d\xi v(y) dy\\  = \iint  \overline{\left(D_{x^j} (m_2(y,x) R(y, x, \xi))e^{i(x-y)\xi}\right)} d\xi v(y) dy + \iint  \overline{\left((m_2(y,x) R(y, x, \xi))e^{i(x-y)\xi}\right)} d\xi D_{y^j} v(y) dy,
\end{multline}
and thus \eqref{L2H-1} follows using \eqref{L2onederiv}.
\end{proof}

 In our subsequent proof, $m(x,y)$ will contain all the dependence on $h$, and $R$ will come from the smooth symbol used in \cite{TT} for the Kerr metric.

A quick corollary, which we will use in Section~\ref{sec:quasilinear}, is the following:
\begin{corr}\label{quasell}
 If $Q = \chi q^w \chi$ with $q(x) = q_1(x)+ \tau q_0(x) \in S^1 + \tau S^0$ and $v \in C^{1,\delta}$, then
 \[
 \| [v, Q] u\|_{L^2}\lesssim \|v\|_{C^{1,\delta}} (\|u\|_{L_{cpt}^2} + \|D_t u\|_{H_{cpt}^{-1}}).
 \]
\end{corr}
\begin{proof}
We write
\[
Q = Q_1 + Q_0 D_t, \qquad Q_i = \chi q_i^w \chi.
\]

The kernel of $[v, Q_1]$ is
\[
K_1 (x,y) = \int (v(x)-v(y)) e^{i(x-y)\xi} R_1 (x, y, \xi), \qquad R_1(x, y, \xi) = \chi(x) \chi(y) q_1(\frac{x+y}2, \xi).
\]

By Taylor's theorem, we can write
\[
v(x)-v(y) = \sum_k (x^k-y^k) \int_0^1 (\pa_k v)(x+(1-s)y) ds,
\]
and thus after integrating by parts
\[
K_1 (x,y) = \int \sum_k (\int_0^1 (D_k v)(x+(1-s)y) ds) e^{i(x-y)\xi} D_{\xi_k} R_1(x, y, \xi).
\]
By Lemma~\ref{L2bd} we obtain, with
\[
m(x,y) = (D_k v)(x) + \int_0^1 (D_k v)(x+(1-s)y) - (D_k v)(x) ds, \qquad R = D_{\xi_k} R_1,
\]
that
\[
\| [v, Q_1] u\|_{L^2}\lesssim \|v\|_{C^{1,\delta}} \|u\|_{L_{cpt}^2}.
\]
A similar argument yields that
\[
\| [v, Q_0] u\|_{L^2}\lesssim \|v\|_{C^{1,\delta}} \|u\|_{H_{cpt}^{-1}},
\]
which yields
\[
\| [v, Q_0 D_t] u\|_{L^2} \lesssim \| (D_t v) Q_0 u)\|_{L^2} + \| [v, Q_0] D_t u\|_{L^2} \lesssim \|v\|_{C^{1,\delta}} (\|u\|_{L_{cpt}^2} + \|D_t u\|_{H_{cpt}^{-1}}).
\]
\end{proof}

The next lemma will be used to prove the main commutator estimate, Proposition~\ref{commhest}.
\begin{lemma}\label{lem:theerrorlemma}
Let $s_\ell(x,y,\xi)$ be a symbol of order $\ell$, where $\ell=0,1$, that is compactly supported in $x$ and $y$ in a region $I\times I$.
Suppose that $r_\ell(x,y)$ is a functions that vanish  to order $3$ at the diagonal:
for $|\mu|\leq 1$ and $|\nu|\leq 1$,
\beq
D_x^{\mu}D_y^\nu  r_\ell(x,y)={\sum}_{|\gamma|=2+\ell-|\mu|-|\nu|} (x-y)^\gamma r^{\,\mu\nu}_{\gamma\ell}(x,y),
\eq
where $r^{\,\mu\nu}_{\gamma\ell}$ satisfy:
\beq\label{rbds}
|\pa_x^{\leq 1-\ell} (r^{\,\mu\nu}_{\gamma\ell}(x,y)-r^{\,\mu\nu}_{\gamma\ell}(x,x))|\lesssim C_{\delta}|x-y|^\delta
,\quad\text{and}\quad
|\pa_x^{\leq 1-\ell} r^{\,\mu\nu}_{\gamma\ell}(x,x)|\lesssim
C_0 .
\eq
for all $x, y\in I$.

 Then the bilinear forms
\begin{equation}
B^j[u,v]=\iiint  r_\ell(x,y)s_\ell(x, y, \xi) e^{i(x-y)\xi}
 {D}_{y^j} u(y) \overline{v(x)} \, d\xi dy \, dx,
\end{equation}
satisfy the estimate
\begin{equation}\label{eq:bilinearformestimate}
|B^j[u,v]|\lesssim (C_\delta+C_0) \|u\|_{H^{\ell-1}(I)} \|v\|_{H^{-1}(I)}.
\end{equation}
\end{lemma}
\begin{proof} Integrating by parts in $y$ we get
\begin{multline}
\iiint  r_\ell(x,y)s_\ell(x, y, \xi) e^{i(x-y)\xi}
 {D}_{y^j}u(y) \overline{v(x)} \, d\xi dy \, dx \\
 = \iiint\! D_{y^j}\bigtwo(  r_\ell(x,y)s_\ell(x, y, \xi)e^{i(x-y)\xi}\bigtwo) u(y) \overline{ v(x)} \, d\xi dy \, dx \lesssim \|R_j u\|_{H^1(I)}\|v\|_{H^{-1}(I)},
\end{multline}
where
\[
R_j u(x) = \iint D_{y^j}\bigtwo(  r_\ell(x,y)s_\ell(x, y, \xi)e^{i(x-y)\xi}\bigtwo) u(y) d\xi \, dy.
\]
We will now show that
\beq\label{H1bd}
\|R_j u\|_{H^1(I)} \lesssim (C_\delta+C_0) \|u\|_{H^{\ell-1}(I)},
\eq
which finishes the proof.
Indeed, let $\bold{k}=\vec{e_k}$ and $\bold{j}=\vec{e_j}$. We write
\begin{multline}
D_{x^k} R_j u(x) = \iiint D_{x^k} D_{y^j}\bigtwo(  r_\ell(x,y)s_\ell(x, y, \xi)e^{i(x-y)\xi}\bigtwo) u(y) d\xi dy \\
=\sum_{\gamma_1+\gamma_2+\gamma_3=(\bold{j},\bold{k})}\!\!\!\!\!\! c_{\gamma_1\gamma_2\gamma_3}
\iiint  r_\ell^{\gamma_1\!}(x,y) \, s_\ell^{\gamma_2}(x, y,\xi)\,\xi^{\gamma_3}e^{i(x-y)\xi} u(y) \, d\xi dy,
\end{multline}
where $r_\ell^{\gamma_1}(x,y)=D_{x,y}^{\gamma_1} r_\ell(x,y) $ and $s_\ell^{\gamma_2}(x, y,\xi)= D_{x,y}^{\gamma_2\!} s_\ell(x, y,\xi)$.
 We now write
\[
r_\ell^{\gamma_1}(x,y) = \sum_{|\gamma|=2+\ell-|\gamma_1|}(x-y)^\gamma r^{\gamma_1}_{\gamma\ell}(x,y) . 
\]
Also $(x-y)^\gamma e^{i(x-y)\xi}=\pa_\xi^\gamma$ so integrating by parts in $\xi$ we can write the RHS as
\beq
\sum_{\gamma_1+\gamma_2+\gamma_3=(\bold{j},\bold{k})}\!\!\!\!\!\! c_{\gamma_1\gamma_2\gamma_3}\sum_{|\gamma|=2+\ell-|\gamma_1|}\iiint r^{\gamma_1}_{\gamma\ell}(x,y) s_\ell^{\gamma\gamma_2\gamma_3\!}(x, y, \xi) e^{i(x-y)\xi} u(y) \overline{ v(x)} \, d\xi dy \, dx.
\eq
Here $ s_\ell^{\gamma\gamma_2\gamma_3\!}(x, y, \xi)=\pa_\xi^{\gamma}\bigtwo( s_\ell^{\gamma_2}(x,y,\xi)\xi^{\gamma_3}\bigtwo)$ are symbols of order $\ell+|\gamma_3|-|\gamma| = |\gamma_3| + |\gamma_1|- 2\leq  0$.

Due to \eqref{rbds} we have
\[
\|\pa_x^{\leq 1-\ell} r^{\gamma_1}_{\gamma\ell}(x,x)\|_{L^{\infty}} \lesssim C_0,
\]
\[
\|\pa_x^{\leq 1-\ell} (r^{\gamma_1}_{\gamma\ell} (x,y)-r^{\gamma_1}_{\gamma\ell} (x,x))\|_{\delta} \lesssim C_{\delta}.
\]
The conclusion \eqref{H1bd} now follows directly from Lemma~\ref{L2bd}.
\end{proof}

\subsection{Estimates for the Kerr operators}
 In terms of symbols to operator bounds, the following lemma will quantize the decomposition \eqref{sKdec}:
\begin{lemma}\label{vts}
 Assume that $q_j\in S^j$, $j\in \{-1, 0\}$. We then have
\[\begin{split}
\| \chi (q_0 \varrho_1 (\tau-\tau_2))^w \chi u\|_{L^2[\tt_0, \tt_1]L^2} + \| \chi (q_0 \varrho_2 (\tau-\tau_1))^w \chi u\|_{L^2[\tt_0, \tt_1]L^2} + \| \chi (q_0 \xi_r)^w \chi u\|_{L^2[\tt_0, \tt_1]L^2} \\ \lesssim \|u\|_{LE_K^1[\tt_0, \tt_1]} + \|D_t u\|_{L^2[\tt_0, \tt_1]H_{cpt}^{-1}},
\end{split}\]
\[\begin{split}
\| \chi (q_{-1} \varrho_1 (\tau-\tau_2))^w \chi D_i u\|_{L^2[\tt_0, \tt_1]L^2} + \| \chi (q_{-1} \varrho_2 (\tau-\tau_1))^w \chi D_i u\|_{L^2[\tt_0, \tt_1]L^2} + \| \chi (q_{-1} \xi_r)^w \chi D_i u\|_{L^2[\tt_0, \tt_1]L^2} \\ \lesssim \|u\|_{LE_K^1[\tt_0, \tt_1]} + \|D_t u\|_{L^2[\tt_0, \tt_1]H_{cpt}^{-1}}.
\end{split}\]
\end{lemma}
\begin{proof}
We will only control the first term, as the rest follow in a similar manner.

Let $\chi_0$ be a cutoff that is identically $1$ on the support of $\chi$. Since $\chi\chi_0 = \chi$, we can write
\[
\chi (q_0 \varrho_1 (\tau-\tau_2))^w \chi = q_0 (D, x) \chi_0 \varrho_1 (D, x) \chi (D_t - \tau_2 (D, x)) \chi + Q_{-1} D_t + Q_0
\]
where $Q_j\in OPS^j$ have compactly supported kernels.
We clearly have
\[
\| Q_{-1} D_t u\|_{L^2[\tt_0, \tt_1]L^2} + \| Q_0 u\|_{L^2[\tt_0, \tt_1]L^2} \lesssim \|u\|_{L^2[\tt_0, \tt_1]L_{cpt}^{2}} + \|D_t u\|_{L^2[\tt_0, \tt_1]H_{cpt}^{-1}}.
\]
 Since $q_0 \in S^0$ we obtain
\[
\| q_0 (D, x) \chi_0 \varrho_1 (D, x) \chi (D_t - \tau_2 (D, x)) \chi u\|_{L^2[\tt_0, \tt_1]L^2} \lesssim \|u\|_{LE_K^1[\tt_0, \tt_1]}.
\]
%
which finishes the proof.
\end{proof}

As a quick corollary, we see that Lemma~\ref{qK} and Lemma~\ref{vts} imply that
\beq
\begin{split}
& \| \chi s_K^w \chi u\|_{L^2[\tt_0, \tt_1]L^2}^2 \lesssim \|u\|_{LE_K^1[\tt_0, \tt_1]}^2 ,\\
&\|\chi (D_{x^i} s_K)^w \chi u\|_{L^2[\tt_0, \tt_1]L^2}^2 \lesssim \|u\|_{LE_K^1[\tt_0, \tt_1]}^2 + \|\chi \left(s_{4k}^{-1}\xi_r\right)^w \chi D_t u\|_{L^2[\tt_0, \tt_1]L^2}^2 ,\\
&\|\chi (D_{\xi_i} s_K)^w \chi D_i u\|_{L^2[\tt_0, \tt_1]L^2}^2 \lesssim \|u\|_{LE_K^1[\tt_0, \tt_1]}^2 + \|\chi \left(s_{4k}^{-2}\xi_r\right)^w \chi D_t u\|_{L^2[\tt_0, \tt_1]L^2}^2, \\
&\|\chi e_K^w \chi D_i u\|_{L^2[\tt_0, \tt_1]L^2}^2 + \|D_i \chi e_K^w \chi  u\|_{L^2[\tt_0, \tt_1]L^2}^2 \lesssim \|u\|_{LE_K^1[\tt_0, \tt_1]}^2 + \|\chi \left(e_4^{-2}\xi_r\right)^w \chi D_t u\|_{L^2[\tt_0, \tt_1]L^2}^2.
\end{split}
\eq

Define
\[
\|v\|_{LE^0[\tt_0, \tt_1]} := \|\chi D_r \la D_x\ra^{-1} \chi v\|_{L^2[\tt_0, \tt_1]L^2} + \|\chi \la D_x\ra^{-1} \chi v\|_{L^2[\tt_0, \tt_1]L^2}.
\]

\begin{corr}\label{corr:Kerrprinc0} We have that
\beq\label{Kerrprinc0}
\begin{split}
 \| \chi s_K^w \chi u\|_{L^2[\tt_0, \tt_1]L^2}^2 + \sum_i \|\chi (D_{x^i} s_K)^w \chi u\|_{L^2[\tt_0, \tt_1]L^2}^2 + \|\chi (D_{\xi_i} s_K)^w \chi D_i u\|_{L^2[\tt_0, \tt_1]L^2}^2 \\ + \|\chi e_K^w \chi D_i u\|_{L^2[\tt_0, \tt_1]L^2}^2 + \|D_i \chi e_K^w \chi  u\|_{L^2[\tt_0, \tt_1]L^2}^2 \lesssim \|u\|_{LE_K^1[\tt_0, \tt_1]}^2 + \|D_t u\|_{LE^0[\tt_0, \tt_1]}^2.
\end{split}\eq
\end{corr}

In particular the norms on the left hand side, which show up in our proof, are controlled by the $LE_K^1$ norm and extra terms that involve the time derivative of $u$. We will control the latter by using the equation in Lemma~\ref{time-deriv}.

\subsection{Commutator estimates}

We finish the section with proving the following commutator estimate:
\begin{prop}\label{commhest}
Let $h^{\alpha\beta}$ be symmetric and assume that $h^{00}=0$. Define
\[
h_1(t) = \sup_{\alpha, \beta} \|h^{\alpha\beta}\|_{L_x^{\infty}(I_{ps})} + \|\partial_{t, x} h^{\alpha\beta}\|_{L_x^{\infty}(I_{ps})},
\]
\[
h_3(t) = \sup_{\alpha, \beta} \|h^{\alpha\beta}(t, x)\|_{C_x^{3, \delta}(I_{ps})},
\]
for some small $\delta>0$. Here the H\"older norms are taken with respect to only the space variable
\beq\label{Holdercomm}
\|w(t,x)\|_{C_x^{k,\delta}(I_{ps})} = \|\pa_{x}^{\leq k} w(t,x)\|_{L_x^{\infty}} + \sum_{|\gamma|=3} \sup_{x\neq y} \frac{|\pa_{x}^{\gamma}w(t,x)-\pa_{x}^{\gamma} w(t,y)|}{|x-y|^{\delta}}.
\eq

Then the following holds
\begin{equation}\label{commprt}
\begin{split}
& \int_{\tt_0}^{\tt_1} \!\!\!\int \Bigl|h^{\alpha\beta}[D_{\beta}, \chi s_K^w \chi]u \cdot \overline{D_{\alpha} u}\Bigr| + \frac12 \Bigl|[h^{j \beta}, \chi s_K^w \chi] D_{j} u\cdot \overline{D_{\beta} u}\Bigr| dx dt \\ & \lesssim \Bigl(\|u\|_{LE_K^1[\tt_0, \tt_1]} + \|D_t u\|_{LE^0[\tt_0, \tt_1]}\Bigr) \Bigl(\| h_1 \partial u\|_{L_{ps}^2[\tt_0, \tt_1]} + \| h_3 u\|_{L_{ps}^2[\tt_0, \tt_1]} + \| h_3 D_t u\|_{L^2[\tt_0, \tt_1]H^{-1}(I_{ps})} \Bigr).
\end{split}\end{equation}
\end{prop}
\begin{proof}
The first term is easy to bound, as the coefficient $h^{\alpha\beta}$ does not enter the commutator. We have
\[
[D_{\beta}, \chi s_K^w \chi] = \chi [D_{\beta}, s_K^w] \chi + [D_{\beta}, \chi] s_K^w \chi + \chi s_K^w [D_{\beta}, \chi].
\]
The last two terms clearly satisfy
\[
\|[D_{\beta}, \chi] s_K^w \chi u\|_{L^2[\tt_0, \tt_1]L^2} + \|\chi s_K^w [D_{\beta}, \chi] u\|_{L^2[\tt_0, \tt_1]L^2} \lesssim \|u\|_{LE_K^1[\tt_0, \tt_1]},
\]
due to the support properties of $\chi$.

On the other hand,
\[
\chi [D_{\beta}, s_K^w] \chi = \chi (D_{\beta} s_K)^w \chi + \chi l_0^w \chi, \quad\qquad l_0 \in S^0 + \tau S^{-1},
\]
and thus by \eqref{Kerrprinc0}
\[
\| \chi [D_{\beta}, s_K^w] \chi u\|_{L^2[\tt_0, \tt_1]L^2} \lesssim \|u\|_{LE_K^1} + \|D_t u\|_{LE^0[\tt_0, \tt_1]} .
\]

By Cauchy Schwarz and the last two inequalities we get
\[
\int_{\tt_0}^{\tt_1}\!\!\! \int \Bigl|h^{\alpha\beta}[D_{\beta}, \chi s_K^w \chi]u \cdot \overline{D_{\alpha} u}\Bigr| dx dt\lesssim \Bigl(\|u\|_{LE_K^1[\tt_0, \tt_1]} + \|D_t u\|_{LE^0[\tt_0, \tt_1]}\Bigr) \| h_1 \partial u\|_{L_{ps}^2[\tt_0, \tt_1]}.
\]

We are left with estimating the second term in \eqref{commprt}. Recall that at the symbol level, $s_K = s_1 + \tau s_0$, where $s_j\in S^j$. We will use the notation
\[
S_j(x, y, \xi):= \chi(x) \chi(y) s_j\big(\tfrac{x+y}2, \xi\big).
\]

Now fix $\tt_0\leq\tt\leq\tt_1$. Since we will only work on a fixed time slice, we suppress the $\tt$ in the notation.

We have
\begin{multline}
\int [h^{j\beta}\! ,\chi s_K^w \chi]{D}_j u\cdot \overline{{D}_\beta u} \,dx\\
=\iiint  [h^{j\beta}(x)-h^{j\beta}(y)] [S_1(x,y, \xi) {D}_j u(y) + S_0(x, y, \xi) {D}_t {D}_j u(y)] e^{i(x-y)\xi}
  \overline{{D}_\beta u(x)} \, d\xi dy \, dx.
\end{multline}
We now write
\[
h^{j\beta}(x)-h^{j\beta}(y) = \frac12(x-y)^k \big(\pa_{x^k} h^{j\beta}(x) + \pa_{y^k} h^{j\beta}(y)\big) + r^{j\beta}(x,y).
\]

The first term on the right hand side is the principal part of the commutator. To estimate it, we note that after integrating by parts in $\xi$ we obtain
\[\begin{split}
& \iiint \frac12(x-y)^k \big(\pa_{x^k} h^{\alpha\beta}(x) + \pa_{y^k} h^{j\beta}(y)\big) [S_1(x,y, \xi) {D}_j u(y) + S_0(x, y, \xi) {D}_t {D}_j u(y)]  e^{i(x-y)\xi}\overline{{D}_\beta u(x)} \, d\xi dy \, dx \\ & =
\frac12 \int \Bigl(\pa_{k} h^{j\beta} \chi (\pa_{\xi^k} s_K)^w \chi D_{\alpha} u + \chi (\pa_{\xi^k} s_K)^w \chi \pa_{k} h^{j\beta} D_j u\Bigr) \cdot \overline{{D}_\beta u} \,dx\\ &  = \Re \int \chi (\pa_{\xi^k} s_K)^w \chi D_j u \cdot \overline{\pa_{k} h^{j\beta} {D}_\beta u(x)}\, dx.
\end{split}\]
By Cauchy Schwarz and \eqref{Kerrprinc0} we have after integrating in time
\[
\int_{\tt_0}^{\tt_1} \int \Bigl| \chi (\pa_{\xi_k} s_K)^w \chi D_j u \cdot \overline{\pa_{k} h^{j\beta} {D}_\beta u(x)}\Bigr| dx dt \lesssim \|\chi (\pa_{\xi_k} s_K)^w \chi D_x u\|_{L^2[\tt_0, \tt_1]L^2} \|h_1 \pa u\|_{L_{ps}^2[\tt_0, \tt_1]} \lesssim RHS\eqref{commprt} .
\]
In order to control the error, we check that $r^{j\beta}$ satisfies the hypotheses of Lemma~\ref{lem:theerrorlemma} with $C_0 = C_{\delta} = h_3(\tt)$. Assuming this holds, we obtain by \eqref{eq:bilinearformestimate} with $l=1$
\[
\iiint  r^{j\beta} S_1(x,y, \xi) {D}_j u(y) e^{i(x-y)\xi}
  \overline{{D}_\beta u(x)} \, d\xi dy \, dx \lesssim h_3 \|u\|_{L^2(I_{ps})}^2.
\]

Moreover, \eqref{eq:bilinearformestimate} for $l=0$ implies
\[
\iiint  r^{j\beta} S_0(x,y, \xi)  {D}_j D_t u(y) e^{i(x-y)\xi}
  \overline{{D}_\beta u(x)} \, d\xi dy \, dx \lesssim h_3 \|D_t u\|_{H^{-1}(I_{ps})} \| D_{\beta} u\|_{H^{-1}(I_{ps})}.
\]
The conclusion follows after integrating in time and applying Cauchy Schwarz.
\end{proof}

We are left with proving the following
\begin{lemma}\label{lem:errorlemma} Suppose that $h^{j\beta}\in C^{3,\delta}$ and let
\[ r^{j\beta}(x,y)=h^{j\beta}(x)-h^{j\beta}(y)-\frac12(x-y)^k \big(\pa_{x^k} h^{j\beta}(x) + \pa_{y^k} h^{j\beta}(y)\big).\]
Then for $|\mu|\leq 1$ and $|\nu|\leq 1$, and $\ell=0,1$ we have
\beq\label{eq:thecancellation}
\pa_x^{\mu}\pa_y^\nu  r^{j\beta}(x,y)={\sum}_{|\gamma|=2+\ell-|\mu|-|\nu|} (x-y)^\gamma r^{j\beta,\mu\nu}_{\gamma\ell}(x,y),
\eq
where
\beq\label{eq:thecancellationdelta2}
\big|\,\pa_x^{\leq 1-\ell} \big( r^{j\beta,\mu\nu}_{\gamma\ell}(x,y)- r^{j\beta,\mu\nu}_{\gamma\ell}(x,x)\big)\big|
\lesssim |x-y|^\delta
\|h^{j\beta}\|_{C^{3,\delta}},\quad\text{and}\quad
|\pa_x^{\leq 1-\ell}  r^{j\beta,\mu\nu}_{\gamma\ell}(x,x)|\lesssim
\|h^{j\beta}\|_{C^{3}}.
\eq
\end{lemma}
\begin{proof}Let
$h(t)=(\pa^\mu\pa^\nu h^{j\beta})\big(\tfrac{x+y}2 + t\tfrac{x-y}2\big)$. Then
$\pa_x^{\mu}\pa_y^\nu r^{j\beta}(x,y)=r_{mn}(1)$, where $m=|\mu|$, $n=|\nu|$ and
$$
r_{00}(t)={h}_1-t {h}_1^\prime(t),\qquad
r_{01}(t)={h}_1(t)/2-t h^\prime(t),\qquad r_{11}(t)=-{h}_1(t),\qquad
{h}_1(t)=h(t)-h(-t).
$$

Since $r_{00}(0)=r_{00}^\prime(0)=0$ we have
$r_{00}(1)=\int_0^1 (1-t) r_{00}^{\prime\prime}(t)\, dt
=-\int_0^1 (1-t)\big(h_1^{\prime\prime}(t)+t h_1^{\prime\prime\prime}(t)\big)\, dt$, and integrating by parts $r_{00}(1)
=-\int_0^1\big( \frac{1}{2}(1-t)^2 h_1^{\prime\prime\prime}(t)+(1-t)t h_1^{\prime\prime\prime}(t)\big)\, dt$ and since $h_1^{\prime\prime\prime}(t)=h^{\prime\prime\prime}(t)+h^{\prime\prime\prime}(-t)$
we have
$$
 r_{00}(1)
 =\frac{1}{2}\int_{-1}^1 \!\!\! (t^2\!-1) h^{\prime\prime\prime}(t)\, dt
  =\frac{1}{2}\int_{-1}^1 \!\!\! (t^2\!-1)\sum_{|\gamma|=3}\big(\tfrac{x-y}2\big)^\gamma\big( \pa^\gamma h^{j\beta})\big(\tfrac{x+y}2 + t\tfrac{x-y}2\big)\, dt
  =\sum_{|\gamma|=3}(x-y)^\gamma r^{j\beta}_{\gamma 1}(x,y),
$$
where
$$
r^{j\beta}_{\gamma 1}(x,y)=\frac{1}{2^4}\int_{-1}^1 (t^2-1)\big( \pa^\gamma h^{j\beta})\big(\tfrac{x+y}2 + t\tfrac{x-y}2\big)\, dt,
$$
which proves \eqref{eq:thecancellation} and \eqref{eq:thecancellationdelta2} for $|\mu|=|\nu|=0$ and $\ell=1$.

Since $r_{00}(0) = 0$ we compute
\[
r_{00}(1)=\int_0^1 r_{00}^{\prime}(t)\, dt = \int_0^1 -t {h}_1^{\prime\prime}(t)\, dt = \int_{-1}^1 -t {h}^{\prime\prime}(t)\, dt
= \sum_{|\gamma|=2} (x-y)^\gamma  r^{j\beta}_{\gamma 0}(x,y),
\]
for
\[
 r^{j\beta}_{\gamma 0}(x,y) = \frac{1}{2^2} \int_{-1}^1 -t \pa^{\gamma}h^{j\beta} \big(\tfrac{x+y}2 + t\tfrac{x-y}2\big)\, dt,
\]
which proves \eqref{eq:thecancellation} and \eqref{eq:thecancellationdelta2} for $|\mu|=|\nu|=0$ and $\ell=0$.

Similarly $r_{10}(0)=r_{10}^\prime(0)=0$ so
$r_{10}(1)=\int_0^1 (1-t) r_{10}^{\prime\prime}(t)\, dt
=\int_0^1 (1-t)\big(h_1^{\prime\prime}(t)/2- 2h^{\prime\prime}(t)-t h^{\prime\prime\prime}(t)\big)\, dt$, and if we integrate the last term by parts we get $r_{10}(1)=
\int_0^1 (1-t)\big(-3h^{\prime\prime}(t)/2- h^{\prime\prime}(-t)/2\big)+(1-2t) h^{\prime\prime}(t)\, dt $, which with $H(t)=1$, when $t>0$ and $H(t)=0$, when $t<0$ we as above can write as
$$
 r_{10}(1)
 =\frac{1}{2}\int_{-1}^1 \!\!\! \,\big((1-|t|)+2tH(t)\big)h^{\prime\prime}(t)\, dt
  =\sum_{|\gamma|=2}(x-y)^\gamma r^{j\beta}_{\gamma  1}(x,y),
$$
where
$$
r^{j\beta}_{\gamma 1}(x,y)=\frac{1}{2^3}\int_{-1}^1 \big((1-|t|)+2tH(t)\big)(\pa^\gamma \pa^\mu h^{j\beta})\big(\tfrac{x+y}2 + t\tfrac{x-y}2\big)\, dt,
$$
which proves \eqref{eq:thecancellation} and \eqref{eq:thecancellationdelta2}for the case $|\mu|=1 $ and $|\nu|=0$ and $\ell=1$.

On the other hand, we can write $r_{10}(1)=\int_0^1 r_{10}'(t) dt=-\int_0^1 (h^\prime(t)-h^{\prime}(-t))/2+ t  h^{\prime\prime}(t) \, dt= -h^\prime(1)+\int_0^1 (h^\prime(t)+h^\prime(-t)) dt/2$.
$$
 r_{10}(1)
 =-h^\prime(1)+\frac{1}{2}\int_{-1}^1 h^\prime(t)\, dt
  =\sum_{|\gamma|=1}(x-y)^\gamma r^{j\beta}_{\gamma  0}(x,y),
$$
where
\[
 r^{j\beta}_{\gamma 0}(x,y) = -\frac12 \pa^{\gamma} \pa^{\mu} h^{j\beta}(x)  +\frac{1}{2^2} \int_{-1}^1(\pa^{\gamma} \pa^{\mu} h^{j\beta}) \big(\tfrac{x+y}2 + t\tfrac{x-y}2\big)\, dt .
\]
which proves \eqref{eq:thecancellation} and \eqref{eq:thecancellationdelta2} for the case $|\mu|=1 $ and $|\nu|=0$ and $\ell=0$.

Since the proof of the case
$|\mu|=0 $ and $|\nu|=1$ is the same it only remains to prove the case
$|\mu|=1 $ and $|\nu|=1$, in which case
$$
 r_{11}(1)
 =-\int_{-1}^1 \!\!\!  h^{\prime}(t)\, dt
  =\sum_{|\gamma|=1}(x-y)^\gamma r^{j\beta}_{\gamma 1}(x,y),\qquad \text{where}\quad r^{j\beta}_{\gamma 1}(x,y)=\frac{1}{2}\int_{-1}^1 \!\!(\pa^\gamma \pa^\mu \pa^\nu h^{j\beta})\big(\tfrac{x+y}2 + t\tfrac{x-y}2\big)\, dt,
$$
and
$$
 r_{11}(1)
 =-h(t)+h(-t)
  = r^{j\beta}_{\gamma 0}(x,y),\qquad \text{where}\quad
   r^{j\beta}_{\gamma 0}(x,y) = -\pa^{\mu}\pa^{\nu} h^{j\beta}(x) + \pa^{\mu}\pa^{\nu} h^{j\beta}(y).\qedhere
$$
\end{proof}

\bigskip

\newsection{Local energy estimates close to the trapped set}

\subsubsection{Conditions on the metric}\label{sec:metricconditions}
 Let $g_K$ be the Kerr metric, and $R$ be a large constant, and $\delta>0$ be an arbitrarily small number. Let $g$ be a metric that is a small perturbation of $g_K$ in the sense that the difference
$h^{\alpha\beta}:= g^{\alpha\beta}-g_K^{\alpha\beta}$ satisfies
\begin{equation}\label{diff}
 |\partial h^{\alpha\beta}|+ |h^{\alpha\beta}|\lesssim \epsilon,
\end{equation}
 everywhere. Moreover, near the trapped set and the light cone we need additional decay estimates as follows:

i) When $ |r-3M| < \frac{M}4$ (which is a region close to the trapped set) we have
\begin{align}
  |\partial h^{\alpha\beta}|+| h^{\alpha\beta} | &\leq \kappa_1(t)\lesssim \epsilon  \la t\ra^{-1/2} ,\label{cpt1}\\
 |\partial_{t} h^{\alpha\beta}| &\leq \kappa_0(t)\lesssim \epsilon \la t\ra^{-1}.\label{cpt2}
\end{align}

In contrast to our previous result on Schwarzschild, see Section 4 of \cite{LT}, we also need smallness (but not decay) on higher order derivatives of $h$. This is due to the use of pseudodifferential operators, which generate errors involving such derivatives. We will assume that
\begin{equation}\label{hgsmall}
\| h^{\alpha\beta} \|_{C^{3, \delta}} \lesssim \epsilon,
\end{equation}
for some small $\delta>0$. Here the H\"older norms are taken with respect to only the space variable, but include all the derivatives
\beq\label{Holder}
\|u\|_{C^{k,\delta}} = \|\pa_{t,x}^{\leq k} u\|_{L^{\infty}L^{\infty}} + \sup_t \sup_{x\neq y} \frac{|\pa_{t,x}^{\leq k} \left(u(t,x)- u(t,y)\right)|}{|x-y|^{\delta}}.
\eq

ii) In the intermediate region $R_1^* \leq \rs\leq \frac{t}2$ we will assume that
\begin{equation}\label{intrm}
|\partial h^{\alpha\beta}| + | h^{\alpha\beta}|\, r^{-1}\lesssim \epsilon r^{-1-\delta},
\end{equation}

iii) In the region close to the cone $\rs>\frac{t}2$ we will distinguish two cases.

The first case, which is the most natural for proving pointwise decay estimates for the linear problem, is to assume that \eqref{intrm} also holds in this region.

For applications to the quasilinear wave equations, we need to be able to also handle the case where the derivative of the metric decays like $\frac1{r}$ near the cone. In this second case, we need to assume different decay rates for different components.

The component of the metric that multiply the derivatives with worst decay $\partial_{\lL}^2 u$ will be required to satisfy the better decay estimates
\begin{equation}\label{badpart}
|\partial h^{\uL\uL} | + |h^{\uL\uL}|\la t-\rs\ra^{-1}
\lesssim \epsilon\la t \ra^{-1-\delta}.
\end{equation}
This is needed for the estimates and is
consistent with what holds for Einstein's equations.

 The other components of $h$ only need to satisfy the weaker estimates:
\begin{equation}\label{goodpart}
 |\partial h | +|h|\la t-\rs\ra^{-1}
 \lesssim \epsilon\la t \ra^{-\frac12-\delta}\la t-\rs\ra^{-\frac12-\delta}.
\end{equation}

 We will  denote by
 \[
 \M_{ps}[\tt_0,\tt_1] = \M[\tt_0,\tt_1]\cap \big\{ \, r\,;\,  |r-3M|\leq M/4\, \big\},
 \]
 a neighborhood of the photonsphere, and by $\chi$ a smooth cutoff supported on $I_{ps}$ so that $\chi = 1$ on $\big\{ \, r\,;\,  |r-3M|\leq M/8\, \big\}$. We also let $\widetilde{\chi}$ be a smooth cutoff supported on
$ \widetilde{I}_{ps}=\big\{ \, r\,;\,  |r-3M|\leq 3M/8\, \big\}$ such that $\widetilde{\chi}= 1$ on $I_{ps}$.

We define
\beq
\|u\|_{L^2_{ps}[\tt_0, \tt_1]}^2=\int_{\M_{ps}[\tt_0,\tt_1]} | u|^2 dV_g,
\eq
\beq
\kappa(t)^2 = \kappa_1(t)^2 + \kappa_0(t).
\eq

 The main goal of this section is to prove the following local energy estimate, which is the counterpart of Theorem 4.1 in \cite{LT}.

\begin{theorem}\label{LE} Let $u$ solve the inhomogeneous linear wave equation
\begin{equation}\label{inhomwave}
 \Box_g u = F,
\end{equation}
where $g$ is a Lorentzian metric satisfying the conditions above, with either \eqref{intrm} or \eqref{badpart} and \eqref{goodpart} being satisfied close to the cone. Then for any $\tt_0 < \tt_1$
\begin{equation}\label{LEK}
\|\partial  u\|_{L^{\infty}[\tt_0, \tt_1] L^2}^2 + \|u\|_{LE_K^1[\tt_0, \tt_1]}^2
+\|D_t u\|_{LE^0[\tt_0, \tt_1]}^2 \lesssim  \|\kappa \pa u\, \|_{L^2_{ps}[\tt_0, \tt_1]}^2 + \|\partial u(\tt_0)\|_{L^2}^2 +B(F,u)_{[\tt_0, \tt_1]},
\end{equation}
where the implicit constant is independent of $\tt_0$, $\tt_1$ and $\epsilon$.
Here $B(F,u)$ is a bilinear norm
\begin{multline}\label{bilform}
B(F,u)_{[\tt_0, \tt_1]}= \!\!\!\!\! \inf_{f_1 + f_2  = F}\Big(
 \|\tchi f_1\|_{L^1[\tt_0, \tt_1] L^2}\|\partial  u\|_{L^{\infty}[\tt_0, \tt_1] L^2} + \|\pa^{\leq 1\!} (\tchi f_2)\|_{L^2[\tt_0, \tt_1]L^2} \|u\|_{LE_K^1[\tt_0, \tt_1]}\Big)\\
 +\int_{\M[\tt_0,\tt_1]}\!\!\!\!\!\! \!\!\!\! (1\!-\!\chi)|F| (|\pa u|+|u|/r) dV_g .
\end{multline}
\end{theorem} Note that here
\beq
 \|\kappa \pa u\, \|_{L^2_{ps}[\tt_0, \tt_1]}^2
 \leq \int_{\tt_0}^{\tt_1} \kappa(\tt)^2 \|\partial  u(\tt)\|_{ L^2}^2\, d\tt ,
\eq
and using Gr\"onwall's lemma we obtain
\begin{equation}\label{LEK0}
\|\partial  u\|_{L^{\infty}[\tt_0, \tt_1] L^2}^2 + \|u\|_{LE_K^1[\tt_0, \tt_1]}^2 \lesssim  e^{\,{\bigone\int}_{\!\!\tt_0}^{\tt_1}\!\! \kappa(\tt)^2d\tt}\Big( \|\partial u(\tt_0)\|_{L^2}^2 +B(F,u)_{[\tt_0, \tt_1]}\Big).
\end{equation}

If ${\bigone\int}_{\!\!\tt_0}^{\tt_1}\! \kappa(\tt)^2 d\tt\leq C$ this estimate implies the
estimate in Theorem \ref{Kerr} for perturbations of Kerr.  This is the case if we assume, for instance, that
\[
\kappa(t) = \la t\ra^{-1/2-\delta},\qquad  \delta>0,
\]
in which case we obtain
 \beq\label{LEK2}
\|\partial  u\|_{L^{\infty}[\tt_0, \tt_1] L^2}^2 + \|u\|_{LE_K^1[\tt_0, \tt_1]}^2  \lesssim \|\partial u(\tt_0)\|_{L^2}^2 + \|F\|_{L^1[\tt_0, \tt_1] L^2+LE^*_K[\tt_0, \tt_1]}^2.
\eq

On the other hand, if we assume that
\[
\kappa(t) = \la t\ra^{-1/2},
\]
we obtain
\beq\label{LEK1}
\|\partial  u\|_{L^{\infty}[\tt_0, \tt_1] L^2}^2 + \|u\|_{LE_K^1[\tt_0, \tt_1]}^2  \lesssim  \int_{\M_{ps}{[\tt_0,\tt_1]}}\frac{\epsilon}{t} |\partial u|^2 dV_g + \|\partial u(\tt_0)\|_{L^2}^2 +B(F,u)_{[\tt_0, \tt_1]},
\eq
which is the estimate we need for the quasilinear problem. Note that this estimate in particular implies by Gronwall's inequality that for the homogeneous problem the energy increases at most as $t^{C\epsilon}$.

\subsubsection{Additional control of $\langle D_x\rangle^{-1} D_t$}
The extra norm that we control in \eqref{LEK},
\[
\|D_t u\|_{LE^0[\tt_0, \tt_1]} = \|\chi D_r \la D_x\ra^{-1} \chi D_t u\|_{L^2[\tt_0, \tt_1]L^2} + \|\chi \la D_x\ra^{-1} \chi D_t u\|_{L^2[\tt_0, \tt_1]L^2},
\]
is needed to control commutator terms in the proof, see \eqref{Kerrprinc0}.
 It is in fact bounded by the other terms in \eqref{LEK}:
 \begin{prop}
\begin{equation}\label{LEtest}
\|D_t u\|_{LE^0[\tt_0, \tt_1]}^2 \lesssim \|u\|_{LE_K^1[\tt_0, \tt_1]}^2 +  E[u](\tt_0)+E[u](\tt_1) + \|\kappa \pa u\, \|_{L^2_{ps}[\tt_0, \tt_1]}^2+B(F,u)_{[\tt_0, \tt_1]}.
\end{equation}
\end{prop}
The proof follows from the following lemma:
\begin{lemma}\label{time-deriv}
For any compactly supported operator $Q_0 = \chi q_0^w \chi$, where $q_0\in S^0$, and $\Box_g u = F$ we have
\beq\label{t-princ}
\| D_t Q_0 u(t)\|_{L^2[\tt_0, \tt_1]L_x^2}^2 \!\lesssim E[u](\tt_0)+E[u](\tt_1)+ \| D_x Q_0 u\|_{L^2[\tt_0, \tt_1]L_x^2}^2 +  \|\kappa \pa u \|_{L^2_{ps}[\tt_0, \tt_1]}^2+ \|u \|_{L^2_{ps}[\tt_0, \tt_1]}^2
+B(F,u)_{[\tt_0, \tt_1]}.
\eq
Similarly, we have that
\beq\label{t-lot}
\|D_t u\|_{L^2[\tt_0, \tt_1]H^{-1}_{comp}}^2 \lesssim E[u](\tt_0)+E[u](\tt_1) + \|u\|_{L^2[\tt_0, \tt_1]L^2_{comp}}^2 + B(F,u)_{[\tt_0, \tt_1]}.
\eq
\end{lemma}

\begin{proof}

Let us prove \eqref{t-princ}. Since $\pa_t$ and $Q_0$ commute, we obtain after integrating by parts in time
\[
\| D_t Q_0 u(t)\|_{L^2[\tt_0, \tt_1]L_x^2}^2\! \lesssim  E[u](\tt_0)+E[u](\tt_1) +\!\! \int_{\M_{ps}{[\tt_0,\tt_1]}} \!\!\! \!\!\!\!\pa_t (g^{00}\!\!\!\sqrt{\!|g|}\pa_t u) \frac1{g^{00}\!\!\!\sqrt{\!|g|}} Q_0^2 u +
g^{00}\!\!\!\sqrt{\!|g|} \pa_t u \,\pa_t (\frac1{\!g^{00}\!\!\!\sqrt{\!|g|}}) Q_0^2 u \,dx dt.
\]

Clearly due to \eqref{cpt2} we have
\[
\int_{\M_{ps}{[\tt_0,\tt_1]}} g^{00}\sqrt{|g|} \pa_t u\,  \pa_t (\frac1{g^{00}\sqrt{|g|}}) Q_0^2 u\, dx dt \lesssim \int_{\M_{ps}{[\tt_0,\tt_1]}} \kappa^2|\partial u|^2 + u^2 dt dx  \lesssim RHS\eqref{t-princ}.
\]

On the other hand, we write
\[
\sqrt{|g|} \Box_g u = \pa_t (g^{00}\sqrt{|g|}\pa_t u) + \pa_i (g^{ij} \sqrt{|g|} \pa_{j} u) + \pa_t (g^{0j} \sqrt{|g|} \pa_{j} u) + \pa_j (g^{j0} \sqrt{|g|} \pa_{t} u).
\]
We multiply $\Box_g u = F$ by $\frac1{g^{00}} Q_0^2 u$ and integrate by parts. We clearly have
\begin{multline}
\int_{\M_{ps}{[\tt_0,\tt_1]}} F (\frac1{g^{00}\sqrt{|g|}} Q_0^2 u) dxdt  \lesssim \inf_{F_1 + F_2 =  F }\Big(
 \|F_1\|_{L^1[\tt_1, \tt_2] L^2}\|\tchi u\|_{L^{\infty}[\tt_1, \tt_2] L^2}
 + \|F_2\|_{L^2_{ps}[\tt_0, \tt_1]}\!\|u\, \|_{L^2_{ps}[\tt_1, \tt_2]}\Big) \\ \lesssim B(F,u)_{[\tt_1, \tt_2]},
\end{multline}
where we used a version of Hardy's inequality to bound the first term.

We will show that
\beq\label{ij}
\int_{\M_{ps}{[\tt_0,\tt_1]}} g^{ij} \sqrt{|g|} \pa_{j} u \cdot \pa_{i} (\frac1{g^{00}\sqrt{|g|}} Q_0^2 u)\, dx dt \lesssim RHS\eqref{t-princ}, \eq
\beq\label{0j}\begin{split}
\int_{\M_{ps}{[\tt_0,\tt_1]}} \Bigl(g^{0j} \sqrt{|g|} \pa_{j} u\Bigr) \cdot \pa_{t}\Bigl(\frac1{g^{00}\sqrt{|g|}} Q_0^2 u\Bigr) + \Bigl(g^{j0} \sqrt{|g|} \pa_{t} u\Bigr) \cdot \pa_{j}\Bigl(\frac1{g^{00}\sqrt{|g|}} Q_0^2 u\Bigr)  dx dt \\ \lesssim RHS\eqref{t-princ} + a\| D_t Q_0 u(t)\|_{L^2[\tt_0, \tt_1]L_x^2}^2.
\end{split}\eq

Since the terms involving $D_t Q_0 u$ on the right hand side can be absorbed on the left hand side for small enough $a$ and $\epsilon$, this completes the proof of \eqref{t-princ}.

Recall that by \eqref{cpt1}-\eqref{cpt2} we have
\[
\Bigl|\pa_{\leq 1}\Bigl(\frac1{g^{00}\sqrt{|g|}} - \frac1{g_K^{00}\sqrt{|g_K|}}\Bigr)\Bigr|^2 \lesssim \kappa_1,
\]
and thus by Cauchy Schwarz
\[\begin{split}
\int_{\M_{ps}{[\tt_0,\tt_1]}} \bigl|g^{ij} \sqrt{|g|} \pa_{j} u \cdot \pa_{i} \bigl((\frac1{g^{00}\sqrt{|g|}}- \frac1{g_K^{00}\sqrt{|g_K|}})Q_0^2 u\bigr)\bigr| dx dt & \lesssim \int_{\M_{ps}{[\tt_0,\tt_1]}} \kappa^2 |\partial u|^2 + |D_x Q_0^2 u|^2 + |Q_0^2 u|^2 dx dt \\ &\lesssim RHS\eqref{t-princ}.
\end{split}\]

Similarly, since
\[
\Bigl|\pa_{\leq 1} \Bigl(g^{ij} \sqrt{|g|} - g_K^{ij} \sqrt{|g_K|}\Bigr)\Bigr| \lesssim \kappa_1,
\]
we obtain
\[\begin{split}
\int_{\M_{ps}{[\tt_0,\tt_1]}} \!\bigl|(g^{ij}\! \sqrt{|g|} \pa_{j} u - g_K^{ij}\! \sqrt{|g_K|})\pa_j u\cdot \pa_{i} \bigl(\frac1{g_K^{00}\!\sqrt{|g_K|}}Q_0^2 u\bigr)\bigr|\, dx dt  & \lesssim \int_{\M_{ps}{[\tt_0,\tt_1]}}\!\!\! \!\! \kappa^2|\partial u|^2\! + |D_x Q_0^2 u|^2\! + |Q_0^2 u|^2 dx dt  \\ & \lesssim RHS\eqref{t-princ}.
\end{split}\]

Finally, after using the symmetry of $Q_0$, we obtain
\beq\begin{split}
& \int_{\M_{ps}{[\tt_0,\tt_1]}} g_K^{ij} \sqrt{|g_K|} \pa_{i} u \cdot \pa_{j} (\frac1{g_K^{00}\sqrt{|g_K|}} Q_0^2 u) \,dx dt \\ &
  = \int_{\M_{ps}{[\tt_0,\tt_1]}} g_K^{ij} \sqrt{|g_K|}\pa_i  u \cdot [\pa_j, \frac1{g_K^{00}\sqrt{|g_K|}}Q_0] Q_0 u+ [Q_0, \frac{g_K^{ij}}{g_K^{00}} \pa_i] u\cdot \pa_j Q_0 u + \frac{g_K^{ij}}{g_K^{00}} \pa_i Q_0 u \cdot \pa_j Q_0 u \,dx dt.
\end{split}\eq
Here one sees directly that the last two terms are bounded by the right hand side of
\eqref{t-princ}. To prove that this is true also for the first one integrates $\pa_i$ by parts.
This finishes the proof of \eqref{ij}.

The proof of \eqref{0j} is similar. Using that $g_K^{0j} = O(a)$ we obtain
\[\begin{split}
& \int_{\M_{ps}{[\tt_0,\tt_1]}} g_K^{0j} \sqrt{|g_K|} \pa_{j} u \cdot \pa_{t}\Bigl(\frac1{g_K^{00}\!\sqrt{|g_K|}} Q_0^2 u\Bigr)\, dx dt = \! \int_{\M_{ps}{[\tt_0,\tt_1]}}
\frac{g_K^{0j}}{g_K^{00}} \pa_j Q_0 u \cdot \pa_t Q_0 u + [Q_0, \frac{g_K^{0j}}{g_K^{00}} \pa_j] u\cdot \pa_t Q_0 u\, dxdt \\ & \lesssim RHS\eqref{t-princ} + a\| D_t Q_0 u(t)\|_{L^2[\tt_0, \tt_1]L_x^2}^2.
\end{split}\]

The other terms are similar to the ones above.

The proof of \eqref{t-lot} is similar to the analogous statement \eqref{Dtell} in Kerr. One multiplies $\Box_g u = F$ by $Q^2$, where $Q\in OPS^{-1}$ is compactly supported, and integrates by parts.
\end{proof}


\subsection{Proof of the main theorem}
We will now prove Theorem~\ref{LE}. Away from $r=3M$ the argument is identical to the one in \cite{LT}. The main difficulty comes from the complicated nature of the trapped set, which is best described in phase space. We will use the same pseudodifferential operator as the one used in the proof of Theorem~\ref{Kerr}. As we shall see, at the principal symbol level,  under the assumptions \eqref{diff}, \eqref{cpt1}, \eqref{cpt2}, the error terms will satisfy
\[
Err \lesssim \epsilon\Bigl(\varrho_1^2(\tau-\tau_2)^2+\varrho_2^2(\tau-\tau_1)^2+\xi_r^2\Bigr) + \kappa(\tt) |\xi|^2.
\]
The first three terms are controlled by the left hand side, while the last one gives rise to the first term on the right hand side of \eqref{LEK}.

We will prove the slightly weaker version of \eqref{LEK}:
\beq\label{LEKw}
\|\partial  u\|_{L^{\infty}[\tt_0, \tt_1] L^2}^2 + \|u\|_{LE_{K,\delta}^1[\tt_0, \tt_1]}^2  \lesssim  \int_{\M_{ps}{[\tt_0,\tt_1]}} \kappa(\tt)|\partial u|^2 dV_g + \|\partial u(\tt_0)\|_{L^2}^2 + B(F,u)_{[\tt_0, \tt_1]} ,
\eq
where
\[
\| u\|_{LE_{K,\delta}^1[\tt_0, \tt_1]} = \|(1- \chi_{R_1}) u\|_{LE_K^1[\tt_0, \tt_1]} + \| \chi_{R_1} r^{-1/2-\delta} \partial u\|_{L^2[\tt_0, \tt_1]L^2} + \| \chi_{R_1} r^{-3/2-\delta} u\|_{L^2[\tt_0, \tt_1]L^2}.
\]
for some $R_1 \gg M$.
This estimate is identical to \eqref{LEK}, except for a loss of $r^{\delta}$. The transition to \eqref{LEK} once \eqref{LEKw} holds is straightforward, see the end of Section 4 in \cite{LT}.

Let $f_0$ be a function so that
\beq\label{f0def}
 f_0(\tt, x) = \left\{ \begin{array}{cc} \frac{\sqrt{|g_K|} g_K^{00}}{\sqrt{|g|} g^{00}} , & |r-3M| \leq M/4 ,\cr 1 & |r-3M|>3M/8.
 \end{array}
 \right.
\eq
and so that
\[
|f_0 - 1| + |\partial f_0| \lesssim \kappa_1(t)
\]
This is possible due to \eqref{cpt1}. The role of $f_0$ is to simplify the problem by making $h^{00}=0$ near the trapped set.

We compute
\begin{equation}\label{oBoxcomp}
f_0\Box_g  = P_0 - g^{\alpha\beta} \big(\pa_\alpha f_0) \pa_\beta, \qquad P_0 := \frac1{\sqrt{|g|}}\pa_{\alpha} (g^{\alpha\beta} f_0 \sqrt{|g|} \pa_\beta).
\end{equation}

In order to apply the results of \cite{LT}, it will be convenient to work with a metric that is identically Kerr near the trapped set and identically $g$ away from it. We thus define
\[
\tg = \tchi g_K + (1-\tchi) g.
\]

We can now write
\beq\label{f0mult}
f_0 \widehat{\Box}_g = \widehat{\Box}_\tg + P_h - g^{\alpha\beta} \big(\pa_\alpha f_0) \pa_\beta,
\eq
where
\[
P_h = \pa_{\alpha}(\th^{\alpha\beta}\pa_{\beta}), \qquad \th^{\alpha\beta}:= f_0 \sqrt{|g|} g^{\alpha\beta} - \sqrt{|\tg|} \tg^{\alpha\beta}.
\]

Clearly $\th^{\alpha\beta}$ are supported near the trapped set, and $\th^{00}=0$ in $\M_{ps}{[\tt_0,\tt_1]}$. It is easy to check (see also Lemma 4.2 in \cite{LT}) that $\th$ also satisfy \eqref{diff}-\eqref{hgsmall}. We will thus slightly abuse notation and redefine $\th=h$ for the rest of the section.

The multiplier we will use will be of the form $C\partial_t + S_g$, with
\[
S_g = f_0 S_K, \qquad S_K = X+S+q+E.
\]

For the $C\pa_t$ term we have
\[
\int_{\M{[\tt_0,\tt_1]}} (\Box_g u) \overline {C\partial_t u}\, dV_g = \int_{\M{[\tt_0,\tt_1]}} Q[g, C\partial_t]\, dV_g -C \left. BDR_1^{g}[u]\right|_{t = \tt_0}^{t = \tt_1} - C\left. BDR_1^{g}[u] \right|_{r=r_e}.
\]

On the other hand, we also compute
\[
\int_{\M{[\tt_0,\tt_1]}} (\widehat{\Box}_\tg u) \overline{S_K u}\,dxdt = \int_{\M{[\tt_0,\tt_1]}} Q[\tg, X, q, m]\, dtdx + IQ[\tg, S, E] - \left. BDR_2^{g}[u]\right|_{t = \tt_0}^{t = \tt_1} - \left. BDR_1^{g}[u] \right|_{r=r_e},
\]
We can now write
\[
S_K = \chi (s_K^w +  e_K^w) \chi + S_{out}, \quad S_{out} = X_{out} + q_{out}
\]

Clearly $X_{out}$ and $q_{out}$ are a first order vector field and scalar function respectively.

We thus have (recall the definitions \eqref{Phvec} and \eqref{Phquad})
\[
\int_{\M{[\tt_0,\tt_1]}} (P_h u) \overline{S_K u}\,dxdt =  \int_{\M{[\tt_0,\tt_1]}} Q[P_h, X_{out}, S_{out}] dxdt + IQ[P_h, \chi s_K^w \chi, \chi e_K^w \chi] - \left. BDR_3^{g}[u]\right|_{t = \tt_0}^{t = \tt_1}.
\]

Since
\[
\int_{\M{[\tt_0,\tt_1]}} (f_0 \widehat{\Box}_g u) \overline{S_K u}\,dxdt = \int_{\M{[\tt_0,\tt_1]}} \Box_g u \cdot  \overline{S_g u} \,  dV_{g},
\]
we obtain, using \eqref{f0mult}
\begin{multline}\label{intdivg}
\int_{\!\M{[\tt_0,\tt_1]}}\!\!\!\!\!\!\! \sqrt{|g|} Q[g, C\pa_t] + \sqrt{|\tg|} Q[\tg, X, q, m] + Q[P_h, X_{out}, S_{out}] - \sqrt{|\tg|} g^{\alpha\beta} \pa_\alpha f_0\, \pa_\beta u\,  S_K u\, dtdx \\ + IQ[\tg, S, E]   + IQ[P_{h}, \chi s_K^w \chi, \chi e_K^w \chi] =
    - \int_{\M{[\tt_0,\tt_1]}} \Box_g u \cdot  \overline{(C\partial_t + S_g) u} \  dV_{g}
    - \left. BDR^{g}[u]\right|_{t = \tt_0}^{t = \tt_1} - \left. BDR^{g}[u] \right|_{r=r_e}.
  \end{multline}

The boundary terms are small perturbations of the boundary terms for the Kerr metric, and thus satisfy
\beq\label{bdryterms}
  \left. BDR^{g}[u]\right|_{t = \tt_i} \approx  \  \|\partial  u(\tt_i)\|_{L^2}^2,
  \quad i=1,2,\qquad\quad
  \left. BDR^{g}[u]\right|_{r=r_e} \approx \ \| u\|_{H^1(\Sigma^+{[\tt_0,\tt_1]})}^2.
\eq

For the inhomogeneous term we prove the following:
\begin{lemma}\label{inhest}
We have
\[
- \int_{\M{[\tt_0,\tt_1]}} \Box_g u \cdot  \overline{(C\partial_t + S_g) u} \  dV_{g} \lesssim B(F,u)_{[\tt_0, \tt_1]} .
\]
\end{lemma}
\begin{proof}
Away from the trapped set, we have that $S_g = X+q$. Thus the inhomogeneous term satisfies
\beq\label{inhomfar}
-\int_{\M{[\tt_0,\tt_1]}\setminus \M_{ps}{[\tt_0,\tt_1]}} \Box_g u \cdot  \overline{(C\partial_t + S_g) u} \  dV_{g} \lesssim \int_{\M{[\tt_0,\tt_1]}} (1-\chi) |F|(|\pa u| + r^{-1} |u|) dV_g,
\eq
by the choice of $X$ and $q$ in the Schwarzschild case.

On the other hand, we also have
\beq\label{inhomest}
\Bigl| \int_{\M_{ps}{[\tt_0,\tt_1]}} \Box_g u \cdot  \overline{(C\partial_t + S_g) u} \  dV_{g}\Bigr|  \lesssim B(F,u)_{[\tt_0, \tt_1]}.
\eq

Indeed, by Cauchy Schwarz we have
\begin{multline}
\Bigl| \int_{\M_{ps}{[\tt_0,\tt_1]}} F_1 \cdot  \overline{(C\partial_t + S_g) u} \  dV_{g}\Bigr| \lesssim \|\tchi F_1\|_{L^1[\tt_0,\tt_1]L^2} \|\tchi (C\partial_t + S_g) u\|_{L^{\infty}[\tt_0, \tt_1] L^2} \\ \lesssim \|\tchi F_1\|_{L^1[\tt_0,\tt_1]L^2} \|\partial u\|_{L^{\infty}[\tt_0, \tt_1] L^2},
\end{multline}
\[
\Bigl| \int_{\M_{ps}{[\tt_0,\tt_1]}} F_2 \cdot  \overline{S_g u} \  dV_{g}\Bigr| \lesssim \|F_2 u\|_{L^2_{ps}[\tt_0,\tt_1]}\| S_g u\|_{L^2_{ps}[\tt_0, \tt_1]} \lesssim \|\tchi F_2 u\|_{L^2[\tt_0,\tt_1]L^2}\| u\|_{LE_K^1[\tt_0, \tt_1]},
\]

Finally, integration by parts in time, the trace theorem and Cauchy Schwarz gives
\begin{multline}
\Bigl| \int_{\M_{ps}{[\tt_0,\tt_1]}} F_2 \cdot  \overline{C\partial_t  u} \  dV_{g}\Bigr| \lesssim  \|\partial_{t} F_2\|_{L^2_{ps}[\tt_0,\tt_1]} \| u\|_{L^2_{ps}[\tt_0, \tt_1]} + \|F_2 u\|_{L^{\infty}[\tt_0, \tt_1] L^2} \\ \lesssim \|\partial_{\leq 1} F_2\|_{L^2_{ps}[\tt_0,\tt_1]} \Bigl(\| u\|_{L^2_{ps}[\tt_0, \tt_1]} + \|\partial u\|_{L^{\infty}[\tt_0, \tt_1]L^2} \Bigr).
\end{multline}

This finishes the proof of \eqref{inhomest}.
\end{proof}

We will also prove that
\beq\label{lhs}
LHS\eqref{intdivg} + \int_{\M_{ps}{[\tt_0,\tt_1]}} \kappa(\tt)|\partial u|^2 dV_g + (a+\epsilon) \left(E[u](\tt_0) + E[u](\tt_1) + B(F,u)_{[\tt_0, \tt_1]}\right) \gtrsim \|u\|_{LE_{K,\delta}^1[\tt_0, \tt_1]}^2.
\eq

The estimates \eqref{bdryterms}, \eqref{inhomest} and \eqref{lhs} imply the desired conclusion \eqref{LEKw} for small enough $a$ and $\epsilon$.

We now prove \eqref{lhs}.

We start with the following lemma that gives control of the local energy estimate. The other terms on the left hand side of \eqref{intdivg} are perturbative.
\begin{lemma} We have
\begin{multline}
\int_{\!\M{[\tt_0,\tt_1]}}\!\!\!\!\!\!\!  \sqrt{|g|} Q[g, C\pa_t] + \sqrt{|\tg|} Q[\tg, X, q, m] \, dtdx + IQ[\tg, S, E] \\ + \int_{\M_{ps}{[\tt_0,\tt_1]}} \kappa(\tt)|\partial u|^2 dV_g + (a+\epsilon) \left(E[u](\tt_0) + E[u](\tt_1) + B(F,u)_{[\tt_0, \tt_1]}\right) \gtrsim \|u\|_{LE_{K,\delta}^1[\tt_0, \tt_1]}^2.
\end{multline}
\end{lemma}
\begin{proof}

Due to \eqref{Kerrbd} we have that
\[
\int_{\!\M{[\tt_0,\tt_1]}}\!\!\!\!\!\!\!  \sqrt{|g_K|} Q[g_K, X, q, m] \, dtdx + IQ[g_K, S, E] \gtrsim \|u\|_{LE_{K,\delta}^1[\tt_0, \tt_1]}^2 - a \|D_t u\|_{L_t^2 H^{-1}_{comp}}^2.
\]
The last term on the right can be controlled by \eqref{t-lot}.

Moreover, since $\tg=g_K$ on $\M_{ps}{[\tt_0,\tt_1]}$, we obtain
\beq\label{diffps}
IQ[\tg, S, E] = IQ[g_K, S, E], \qquad \int_{\M_{ps}{[\tt_0,\tt_1]} } Q[\tg,X,q,m] dV_{\tg} = \int_{\M_{ps}{[\tt_0,\tt_1]} } Q[g_K,X,q,m] dV_K.
\eq
On the other hand, away from the trapped set both $g$ and $\tg$ are small perturbations of the Schwarzschild metric $g_S$ with the same mass $M$ (and are actually equal away from a compact set). In the first case when \eqref{intrm} holds everywhere, an easy computation yields
\beq\label{diffQint0}
\begin{split}
& \int_{\M{[\tt_0,\tt_1]}\setminus \M_{ps}{[\tt_0,\tt_1]}} \sqrt{|g|} Q[g,C\partial_t] + \sqrt{|\tg|} Q[\tg, X, q, m] - \sqrt{|g_S|} Q[g_S,C\partial_t +X,q,m] dt dx \\ & \lesssim \epsilon \int_{\M{[\tt_0,\tt_1]}\setminus \M_{ps}{[\tt_0,\tt_1]} } r^{-1-\delta}|\pa u|^2 + r^{-3-\delta}|u|^2 dt dx \lesssim \epsilon \|u\|_{LE_{K,\delta}^1[\tt_0, \tt_1]}^2
\end{split}\eq

In the second case, when \eqref{badpart} and \eqref{goodpart} hold, the proof of (4.25) in \cite{LT} yields
\begin{equation}\label{diffQint}
\begin{split}
&  \int_{\M{[\tt_0,\tt_1]}\setminus \M_{ps}{[\tt_0,\tt_1]}} \sqrt{|g|} Q[g,C\partial_t] + \sqrt{|\tg|} Q[\tg, X, q, m] - \sqrt{|g_S|} Q[g_S,C\partial_t +X,q,m] dt dx \\ & \lesssim \epsilon \Bigl(\int_{\M{[\tt_0,\tt_1]}\setminus \M_{ps}{[\tt_0,\tt_1]} } Q[g_S,C\partial_t +X,q,m]\sqrt{|g_S|} + |F|(|\pa u| + r^{-1} |u|) dt dx + \|\partial u\|^2_{L^{\infty}[\tt_0, \tt_1]L^2} \Bigr).
\end{split}
\end{equation}

 Moreover, since $|\pa^{\leq 1} (g_K-g_S)| \lesssim \frac{a}{r^2}$ we trivially have that
\begin{equation}\label{diffKS}\begin{split}
\int_{\M{[\tt_0,\tt_1]}\setminus \M_{ps}{[\tt_0,\tt_1]}} & Q[g_K, C\partial_t + X,q,m]\sqrt{|g|} - Q[g_S,C\partial_t + X,q,m]\sqrt{|g_S|} dt dx \\ &\lesssim a\int_{\M{[\tt_0,\tt_1]}\setminus \M_{ps}{[\tt_0,\tt_1]} } Q[g_S,X,q,m]\sqrt{|g_S|}dt dx.
\end{split}\end{equation}

Finally, \eqref{cpt2} shows that
\beq\label{dtps}
\int_{\M_{ps}{[\tt_0,\tt_1]}} Q[g, C\pa_t] dV_g \lesssim \|\kappa_0 \partial u\|^2_{L_{ps}^2}.
\eq

This finishes the proof of the lemma.
\end{proof}

Let us now control the error caused by $f_0$.
\begin{lemma} We have
\[
\int_{\M{[\tt_0,\tt_1]}} \left|\bigl(g^{\alpha\beta} (\pa_\alpha f_0) \pa_\beta u)\bigr) (S_K u)\right|  \, dV_{\tg}  \lesssim \epsilon \left(\|u\|_{LE_K^1[\tt_0, \tt_1]}^2 + E[u](\tt_0) + E[u](\tt_1) + B(F,u)_{[\tt_0, \tt_1]}\right).
\]
\end{lemma}
\begin{proof}
We first note that $\pa f_0 = 0$ outside $[\tt_0,\tt_1]\times \tilde I_{ps}$. Secondly, we write as above
\[
S_K = \chi (s_K^w +  e_K^w) \chi + S_{out}.
\]
Clearly $S_{out}$ is a first order differential operator with support outside $I_{ps}^{\prime}$ and thus, since $\partial f_0 = O(\epsilon)$:
\[
\int_{\M{[\tt_0,\tt_1]}} \left|\bigl(g^{\alpha\beta} (\pa_\alpha f_0) \pa_\beta u)\bigr) (S_{out} u)\right|  \, dV_{\tg} \lesssim \epsilon \int_{\tt_0}^{\tt_1} \int_{\tilde I_{ps}\setminus I_{ps}^{\prime}}|\pa u|^2 + u^2  \, dV_{\tg} \lesssim \epsilon \|u\|_{LE_K^1[\tt_0, \tt_1]}^2 .
\]
Finally by Cauchy Schwarz, \eqref{cpt1}, \eqref{Kerrprinc0} and \eqref{LEtest} we can now estimate the error near the trapped set:
\[\begin{split}
& \int_{\M{[\tt_0,\tt_1]}} \left|\bigl(g^{\alpha\beta} (\pa_\alpha f_0) \pa_\beta u)\bigr) (\chi (s_K^w +  e_K^w) \chi u)\right|  \, dV_{\tg}  \lesssim \int_{\M_{ps}{[\tt_0,\tt_1]}}  \kappa(\tt) |\partial u|^2  + \epsilon |\chi (s_K^w +  e_K^w) \chi u|^2 dV_{\tg} \\ &\lesssim \int_{\M_{ps}{[\tt_0,\tt_1]}} \kappa(\tt) |\partial u|^2 \, dV_g + \epsilon \Bigl(\|u\|_{LE_K^1[\tt_0, \tt_1]}^2 +  E[u](\tt_0)+E[u](\tt_1) + B(F,u)_{[\tt_0, \tt_1]}\Bigr).\qedhere
\end{split}\]
\end{proof}

We are thus left with proving that:
\begin{lemma} We have
\begin{equation}\label{overest}
\begin{split}
\Bigl| \int_{\M{[\tt_0,\tt_1]}} Q[P_h, X_{out}, q_{out}] dxdt + IQ[P_h, \chi s_K^w \chi, \chi  e_K^w \chi] dV_{\tg} \Bigr| \lesssim \int_{\M_{ps}{[\tt_0,\tt_1]}} \kappa(\tt)|\partial u|^2 dV_g \\ + \epsilon \Bigl(\|u\|_{LE_K^1[\tt_0, \tt_1]}^2 +  E[u](\tt_0)+E[u](\tt_1) + B(F,u)_{[\tt_0, \tt_1]}\Bigr) .
\end{split}\end{equation}
\end{lemma}
\begin{proof}
Since $X_{out}$ and $q_{out}$ are supported away from the trapped set, and $h$ is supported in a compact region, it is immediate that
\[
\Bigl| \int_{\M{[\tt_0,\tt_1]}} Q[P_h, X_{out}, q_{out}] dxdt \Bigr| \lesssim \epsilon \|u\|_{LE_K^1[\tt_0, \tt_1]}^2.
\]

It is thus enough to show that
\begin{equation}\label{overest2}
\Bigl| IQ[P_h, \chi s_K^w \chi, \chi  e_K^w \chi] \Bigr| \lesssim \int_{\M_{ps}{[\tt_0,\tt_1]}} \kappa(\tt)|\partial u|^2 dV_g + \epsilon \Bigl(\|u\|_{LE_K^1[\tt_0, \tt_1]}^2 +  E[u](\tt_0)+E[u](\tt_1) + B(F,u)_{[\tt_0, \tt_1]}\Bigr).
\end{equation}

Let us first estimate the error coming from the zero order term $\chi e_K^w\chi$. Using the definition \eqref{Equad}, we need to show that
\beq\label{0err}
\int_{\tt_0}^{\tt_1}\!\!\!\int \Bigl|h^{\alpha j}  D_{\alpha} u\cdot \overline{D_{j} \chi e_K^w \chi u} + h^{0j} D_{0} u\cdot \overline{D_{\!j} \chi e_K^w \chi u} + \big(D_0 h^{0j} D_{\!j} u- D_{\!j} h^{0j} D_0 u\big)\cdot \overline{\chi e_K^w \chi u} \Bigr| dx dt \lesssim RHS\eqref{overest2} .
\eq

Since by \eqref{Kerrprinc0} and \eqref{LEtest} we have
\[
\int_{\tt_0}^{\tt_1}\!\!\!\int |\chi e_K^w\chi D_{j} u|^2 + | D_{j} \chi e_K^w\chi u|^2 + |\chi e_K^w\chi u|^2 dx dt \lesssim \|u\|_{LE_K^1[\tt_0, \tt_1]}^2 + E[u](\tt_0)+E[u](\tt_1)+ B(F,u)_{[\tt_0, \tt_1]} \, ,
\]
\eqref{0err} follows by Cauchy Schwarz and \eqref{cpt1}.

We are left with dealing with the contribution coming from the first order term $\chi s_K^w \chi$. Indeed, using the definition \eqref{Squad} combined with Proposition \ref{commhest}, \eqref{cpt1}, \eqref{hgsmall} and \eqref{LEtest} we obtain

\begin{equation}\label{commprt1}
\begin{split}
& \int_{\tt_0}^{\tt_1} \!\!\!\int \Bigl|h^{\alpha\beta}[D_{\beta}, \chi s_K^w \chi]u \cdot \overline{D_{\alpha} u}\Bigr| + \frac12 \Bigl|[h^{j\beta}, \chi s_K^w \chi] D_j u\cdot \overline{D_{\alpha} u}\Bigr| dx dt \\ & \lesssim \Bigl(\|u\|_{LE_K^1[\tt_0, \tt_1]} + \|D_t u\|_{LE^0[\tt_0, \tt_1]}\Bigr) \Bigl(\| \kappa_1 \partial u\|_{L_{ps}^2[\tt_0, \tt_1]} + \|\epsilon u\|_{L_{ps}^2[\tt_0, \tt_1]}\Bigr) \\ & \lesssim \epsilon \Bigl(\|u\|_{LE_K^1[\tt_0, \tt_1]}^2 + E[u](\tt_0)+E[u](\tt_1) + B(F,u)_{[\tt_0, \tt_1]}\Bigr)  + \int_{\M_{ps}{[\tt_0,\tt_1]}} \kappa(\tt)|\partial u|^2 dV_g .
\end{split}\end{equation}

This finishes the proof of the lemma.
\end{proof}

\newsection{Quasilinear wave equations close to Kerr}\label{sec:quasilinear}

The main application of Theorem~\ref{LE} will be to establish global existence of solutions to quasilinear wave equations close to Kerr metrics with small angular momentum, which extends our previous result from \cite{LT}.

Let us recall the setup from \cite{LT}. We study the equation
\begin{equation}\label{maineeqn}
\Box_{g(u, t, x)} u = 0, \qquad u |_{\tt=0} = u_0, \qquad \tilde T u |_{\tt=0} = u_1
\end{equation}
Here $\Box_g$ denotes the d'Alembertian with respect to a Lorentzian metric $g$, which is equal the Schwarzschild metric when $u\equiv 0$, and
$\tilde T$ is a smooth, everywhere timelike vector field that equals
$\partial_t$ away from the black hole.  The coordinate $\tt$
is chosen so that the slice $\tt=0$ is space-like and so $\tt=t$
away from the black hole.

The metric $g$ is given by
\begin{equation}\label{quasmetric}
g^{\alpha\beta} = g_K^{\alpha\beta} + H^{\alpha\beta}(t,x)u + O(u^2)
\end{equation}
with $g_K$ denoting the Kerr metric and $H^{\alpha\beta}$ being smooth functions.

Let $\tilde r$ denote a smooth strictly increasing function (of $r$) that equals $r$ for $r\leq R_1$ and $\rs$ for $r\geq 2R_1$, where $R_1 >>  6M$. Our favorite sets of vector fields will be
\[
\partial = \{ \partial_{\tt}, \partial_i\}, \qquad \Omega = \{x^i \partial_j -
x^j \partial_i\}, \qquad S = \tt \partial_{\tt} + \tilde r \partial_{\tilde r},
\]
namely the generators of translations, rotations and scaling. We set
$Z = \{ \partial,\Omega,S\}$.

For a triplet $\Lambda = (i,j,k)$ of multi-indices $i$, $j$ and $k$ we
denote $|\Lambda| = |i|+3|j|+3k$ and
\[
u_{\Lambda} = \partial^i \Omega^j S^k u, \qquad u_{\leq m} =
(u_{\Lambda})_{|\Lambda|\le m}.
\]

We will also use the notation
\[
\partial_{\leq m} u = (\partial_{\alpha} u)_{|\alpha|\leq m}.
\]

Let $h^{\alpha\beta}:= g^{\alpha\beta}-g_K^{\alpha\beta}$ be the difference in the metric coefficients. We will allow the metric $g$ to depend on the solution $u$, so that the difference in the metric coefficients $h^{\alpha\beta}(t, x, u)$ are smooth functions satisfying
\begin{equation}\label{metricform}
h^{\alpha\beta}(t, x, u) = H^{\alpha\beta}(t,x)u + O(u^2).
\end{equation}
Since we want $h\approx u$ we want the functions $H^{\alpha\beta}$ to satisfy
\begin{equation}\label{metriccoeff}
H^{\alpha\beta}\in S^Z(1).
\end{equation}

 For the derivatives of $H$ we need to impose extra conditions to make the metric satisfy the conditions in Theorem~\ref{LE}, namely
\begin{equation}\label{derivcoeff}
\partial H^{\alpha\beta}\in S^Z\Bigl(\frac{\tt^{1+\delta}}{r^{1+\delta}\la \tt-\rs\ra^{1/2+\delta}}\Bigr), \qquad  \rs\geq R_1^*,
\end{equation}
where $R_1^*=r^*(R_1)$.  We remark that a natural condition to impose on $\partial H^{\alpha\beta}$ is
\[
 \partial H^{\alpha\beta}\in S^Z\Bigl(\frac{\tt}{r\la \tt-\rs\ra}\Bigr)
\]
as this yields that $\partial h\approx \partial u$. However, we chose to instead work with the weakest possible assumption under which we can prove our result, which is \eqref{derivcoeff}.

In addition we will need two technical conditions near the trapped set and the light cone. Let us define
\begin{equation}\label{W-def}
W^{\beta} =  \frac{1}{\sqrt{|g_S|}}\partial_{\alpha} (g_S^{\alpha\beta}\sqrt{|g_S|}) - \frac{1}{\sqrt{|g|}}\partial_{\alpha} (g^{\alpha\beta}\sqrt{g}).
\end{equation}

A priori, near the trapped set $W$ satisfies the bound
\begin{equation}\label{W-estfar}
|W_{\Lambda}| \lesssim |\partial h_{\Lambda}| + |h_{\leq \Lambda}| \lesssim |H||\partial u_{\Lambda}| + |u_{\leq \Lambda}|, \qquad |\Lambda|\leq N
\end{equation}
which for the highest order term will not suffice to close the estimates under the assumption \eqref{metriccoeff}. We will thus impose the extra assumption that
\begin{equation}\label{W-est}
|W_{\Lambda}| \lesssim \la t\ra^{-1/2} |\partial u_{\Lambda}| + |u_{\leq \Lambda}|, \qquad |\Lambda|\leq N,\qquad \text{when}\quad \tfrac{5M}2 \leq r \leq \tfrac{7M}2.
\end{equation}

One way to make sure this condition is satisfied is to assume, for example, that
\[
|H_{\leq N}| \lesssim r^{1/2}\la t\ra^{-1/2}
\]
near the trapped set. Indeed, the condition above clearly implies \eqref{W-est}. We remark that in the context of Einstein's Equations written down in generalized wave  coordinates $W = 0$, so good behavior of $W$ can be expected.

On the other hand, in the region close to the cone $\rs\sim t$, $\rs>\frac{t}2$ we need to assume
additional decay for the components of $h$ multiplying the worst decaying derivatives.
To formulate this we need to express $h$ in a nullframe with respect to the Schwarzschild metric $g_S$:
\begin{equation}
\uL=\partial_t - \partial_{\rs}, \quad L= \partial_t + \partial_{\rs},
\quad A=A^i(\omega)\partial_i,\quad B=B^i(\omega)\partial_i,\quad
\partial_{\rs} =\big(1-\frac{2M}{r}\big)\omega^i\partial_i.
\label{nullframe}
\end{equation}
Here $\uL$ and $L$ are null vectors
\[
g_S(L,L)=g_S(\uL,\uL)=0,\qquad g_S(L,\uL)=-2\big(1-\frac{2M}{r}\big),
\]
and $A$ and $B$
 are  two orthonormal vectors
 \[
g_S(A,A)=g_S(B,B)=1,\qquad g_S(A,B)=0,
 \]
 tangential to the spheres where $r$ is constant:
 \[
 g_S(L,B)=g_S(L,A)=g_S(\uL,B)=g_S(\uL,A)=0.
 \]

  Expanding $h$ in the nullframe
\begin{multline}
h^{\alpha\beta}=h^{\uL\uL}{\uL}^\alpha{\uL}^\beta
+ \sum_{T\in \mathcal{T}} h^{\uL T} ({\uL}^\alpha T^\beta +T^\alpha {\uL}^\beta)
+\sum_{U,T\in \mathcal{T}} h^{UT} U^\alpha T^\beta\\
=h^{\uL\uL}{\uL}^\alpha{\uL}^\beta
+\sum_{T\in \mathcal{T}} h^{\uL T} T^\alpha {\uL}^\beta
+\sum_{T\in \mathcal{T}} h^{\alpha T} T^\beta,
\qquad\text{where}\quad \mathcal{T}=\{L,A,B\},\label{nullframeexpansion}
\end{multline}
we need to assume additional decay on $h^{\uL\uL}$.
We note that the coefficients in the nullframe expansion can be determined from
$h$ applied to the dual vectors with respect to the Schwarzschild metric
\begin{equation}
U_\alpha=g_{S\alpha\beta} U^\beta,\qquad U^\alpha V_\alpha=g_S(U,V).\label{loowering}
\end{equation}
A calculation shows that $A_i=A^i$, $B_i=B^i$, and
 \begin{equation}
 L_0=-\Big(1-\frac{2M}{r}\Big),\quad  L_i=\omega_i,\qquad
 \uL_0=-\Big(1-\frac{2M}{r}\Big),\quad  \uL_i=-\omega_i.\label{lowernull}.
 \end{equation}
With the notation
\begin{equation}
h_{\alpha\beta} = g_{S\alpha\gamma}g_{S\alpha\delta} h^{\gamma\delta},
\qquad h(U,V)=h_{\alpha\beta}U^\alpha V^\beta=h^{\alpha\beta}U_\alpha V_\beta,
\end{equation}
we have in particular
\begin{equation}
h^{\uL\uL}=h(L,L)/g_S(L,\uL)^2.\label{huLuLcomponent}
\end{equation}
Geometrically the $h^{\uL\uL}$ component controls the bending of the outgoing light cones.
For the reduced Einstein's equations the metric satisfy the wave coordinate condition
which in particular gives control of this component through the wave coordinate condition
that heuristically says that $\partial h^{\uL\uL}=O(\overline{\partial} h)+O(h^2)$.

We need to assume that $h^{\lL\lL}$ decays at a faster rate like in \eqref{badpart}
because it is the coefficient multiplying the  second derivative transversal to the
outgoing light cones that has the least decay.
More explicitly, we assume that it satisfies the decay estimates
\begin{equation}\label{badpart2}
|h^{\lL\lL}_{\Lambda}|\lesssim \la t-\rs\ra^{\delta} \la t \ra^{-\delta}|u_{\leq \Lambda}|, \qquad |\partial h^{\lL\lL}_{\Lambda} | \lesssim \la t-\rs\ra^{\delta}\la t \ra^{-\delta}\Bigl(|\partial u_{\leq \Lambda}| + \la t-\rs\ra^{-1/2} |u_{\leq \Lambda}| \Bigr)
\end{equation}
for some small $\delta>0$ and all $|\Lambda|\leq N$.

 Again, in the context of Einstein's equations in wave  coordinates, we expect \eqref{badpart2} to hold (see \cite{LR10,L17}). Here it follows from the following assumption on
 $H$:
 \begin{equation}\label{nullcond}
|H^{\lL\lL}_{\leq N}|\lesssim \la t-\rs\ra^{\delta}\la t \ra^{-\delta},\qquad |\partial  H^{\lL\lL}_{\leq N} |\lesssim \la t-\rs\ra^{\delta-1/2}\la t\ra^{-\delta} 
 \end{equation}

The metric coefficient $h^{\lL \lL}$ is in front of the derivative with the least decay
$\partial_{\underline{L}}^2 u$. In \cite{LR10,L17}) it was proven that for Einstein's equations in wave coordinates
\begin{equation}
|Z^I h^{\lL \lL}|\leq \frac{C\varepsilon}{1+t+r} \Big(\frac{1+|t-r^*|}{1+t}\Big)^\gamma,\label{einsteinwave}
\end{equation}
if initial data are asymptotically flat, i.e. $h\big|_{t=0}=M/r+O\big(r^{-1-\gamma}\big)$,
$0\!<\!\gamma\!<\!1$. Our method here works for the case corresponding to any small $\gamma>0$; if one assumes more decay of the coefficient
$|H^{\lL\lL}|\leq C\varepsilon \langle t-r^*\rangle^\gamma \langle t\rangle^{-\gamma}$ corresponding to larger $\gamma$ it may be possible to prove some additional decay for the solution in the interior.

 We are now ready to state the our main result. We pick a large enough integer $N$, and define
\[
\mathcal{E}_N(t) = \|\partial  u_{\leq N}\|_{L^{\infty}[0, t] L^2}^2 + \|u_{\leq N}\|_{LE_K^1[0, t]}^2
\]

\begin{theorem}\label{mainthm}

Assume that the metric $g$ is like in \eqref{quasmetric}, and satisfies the extra conditions \eqref{metriccoeff}, \eqref{derivcoeff}, \eqref{W-est} and \eqref{nullcond}. Then there exists a global classical solution to \eqref{maineeqn}, provided that the initial data is smooth, compactly supported and satisfies, for a certain $\epsilon_0 \ll 1$,
\[
\mathcal{E}_N(0) \leq \epsilon_0^2
\]
Moreover, the solution satisfies the estimates \eqref{enapbds} and \eqref{ptwseapbds} below.

\end{theorem}

We will now outline the bootstrap argument. There are three main theorems one needs to prove.

\begin{theorem}\label{higherLE} Let $u$ solve \eqref{maineeqn}, where $g$ is a Lorentzian metric satisfying the conditions from Section 3.
 Then for some constant $C_N$ independent of $\tt_0$, $\tt_1$ we have:
\begin{equation}\label{higherLES}
 \|\partial  u_{\leq N}\|_{L^{\infty}[\tt_0, \tt_1] L^2} + \|u_{\leq N}\|_{LE_K^1[\tt_0, \tt_1]} \leq C_N \tt_1^{C_N\epsilon} \|\partial u_{\leq N}(\tt_0)\|_{L^2}
\end{equation}
\end{theorem}

\begin{theorem}\label{ptwsedcy} Let $T$ be a fixed time and $u$ solve \eqref{maineeqn} in the time interval $T\leq t\leq 2T$. Assume that $g(u,t,x)$ satisfies the conditions from Section 3.
Then for any multi index $|\Lambda|\leq N-13$ we have for $T\leq t\leq 2T$:
\begin{equation}\label{ptdecayu}
|u_{\Lambda}| \leq C'_{|\Lambda|} \la t\ra^{-1} \la t-\rs\ra^{1/2} \|u_{\leq |\Lambda|+13}\|_{LE_K^1[T, 2T]}
\end{equation}
\begin{equation}\label{ptdecaydu}
|\partial u_{\Lambda}| \leq C'_{|\Lambda|} \la r\ra^{-1} \la t-\rs\ra^{-1/2} \|u_{\leq |\Lambda|+13}\|_{LE_K^1[T, 2T]}
\end{equation}
\end{theorem}
\begin{theorem}\label{lowerLE} Let $u$ solve \eqref{maineeqn}, where $g$ is as in Theorem~\ref{higherLE}.
Then for $\tN \leq N-3$, we have:
\begin{equation}\label{lowLES}
 \|\partial  u_{\leq \tN}\|_{L^{\infty}[\tt_0, \tt_1] L^2}^2 + \|u_{\leq \tN}\|_{LE_K^1[\tt_0, \tt_1]}^2 \lesssim \|\partial u_{\leq N}(\tt_0)\|_{L^2}^2
\end{equation}
\end{theorem}

Let us now assume that the initial data is small enough,
\begin{equation}\label{indata}
\mathcal{E}_N(0) \leq \mu_N \epsilon^2
\end{equation}
where $N\geq 36$ and $\mu_N>0$ is a fixed, small $N$-dependent constant to be determined below.

 Let $N_1=\frac{N}2 +2$ and $\tN = N-3$. We will assume that the following a-priori bounds hold for some large constant $\tilde C$ independent of $\epsilon$ and $t$, and a fixed small $\delta>0$
\begin{equation}\label{enapbds}
\mathcal{E}_N(t) \leq \tilde C \mu_N\epsilon^2 \la t\ra^{\delta}, \qquad \mathcal{E}_{\tN}(t) \leq \tilde C \mu_N\epsilon^2,
\end{equation}
\begin{equation}\label{ptwseapbds}
|u_{\leq N_1}| \leq \frac{\epsilon\la t-\rs\ra^{1/2}}{\la t\ra}, \qquad |\partial u_{\leq N_1}| \leq \frac{\epsilon}{\la r\ra\la t-\rs\ra^{1/2}}
\end{equation}

Clearly \eqref{enapbds} and \eqref{ptwseapbds} are true for small enough times by standard local theory existence combined with \eqref{indata} and Sobolev embeddings. We will now assume that \eqref{enapbds} and \eqref{ptwseapbds} hold for all $0\leq t\leq T$, and we will improve the constants on the right hand side by a factor of $1/2$. By a standard continuity argument this will imply the desired result.

Due to the assumptions \eqref{ptwseapbds} we can apply Theorem~\ref{higherLE}. We obtain
\[
\mathcal{E}_N(t) \leq C_N \la t\ra^{C_N \epsilon} \mathcal{E}(0)
\]
If we take $\tilde C = 2C_N$ and $\epsilon < \frac{\delta}{C_N}$ we improve the a-priori bound for $\mathcal{E}_N(t)$ to
\[
\mathcal{E}_N(t) \leq \frac12 \tilde C  \mu_N \epsilon^2 \la t\ra^{\delta}
\]

Similarly due to the assumptions \eqref{ptwseapbds} we can apply Theorem~\ref{lowerLE} to obtain
\[
\mathcal{E}_{\tN}(t) \leq \frac12 \tilde C  \mu_N \epsilon^2
\]

 Finally, since $N_1\leq \tN - 13$, we can apply Theorem~\ref{ptwsedcy} and obtain
\[
|u_{\leq N_1}| \leq C'_{N_1}\frac{\la t-\rs\ra^{1/2}}{\la t\ra} \sqrt{\mathcal{E}_{\tN}(T)} \leq \frac12 \frac{\epsilon\la t-\rs\ra^{1/2}}{\la t\ra}
\]
\[
|\partial u_{\leq N_1}| \leq C'_{N_1}\frac{1}{\la r\ra\la t-\rs\ra^{1/2}} \sqrt{\mathcal{E}_{\tN}(T)} \leq \frac12 \frac{\epsilon}{\la r\ra\la t-\rs\ra^{1/2}}
\]
since we can choose $\mu_N$ so that $C'_{N_1}\tilde{C}^{1/2}\mu_N^{\frac12} \leq 1/2$. This finishes the proof of Theorem~\ref{mainthm}.

The proofs of Theorems~\ref{higherLE} - \ref{lowerLE} in the Schwarzschild case are given in detail in Sections 5, 6, and 7 of \cite{LT}. They are very similar in the Kerr case, replacing $g_S$ by $g_K$ in the argument. There are two minor differences: the use of the microlocal norm $LE_K^1$ instead of $LE_S^1$, and the fact that there is one extra term in the estimates coming from the difference $\Box_{g_K} - \Box_{g_S}$. The starting point is Theorem~\ref{LE}, which replaces Theorem 4.1 from \cite{LT}.

It is easy to check that
\[
\Bigl|(g_K^{\alpha\beta} - g_S^{\alpha\beta})_{\Lambda}\Bigr| \leq C_{|\Lambda|} \frac{a}{r^2}
\]
and thus
\beq\label{K-Sc}
\Bigl| (\Box_{g_K} u- \Box_{g_S} u)_{\Lambda}\Bigr| \lesssim \frac{a}{r^2} (|\pa^2 u_{\Lambda}| + |\pa u_{\leq |\Lambda|}|)
\eq
\beq\label{K-Sc-comm}
\Bigl| [\Box_{g_K} - \Box_{g_S}, Z^{\Lambda}]\Bigr| \lesssim \frac{a}{r^2} (|\pa^2 u_{<|\Lambda|}| + |\pa u_{<|\Lambda|}|)
\eq

We use \eqref{K-Sc} to bound the extra term in Lemma 5.4 and Lemma 6.3 from \cite{LT}, and \eqref{K-Sc-comm} to bound the extra term when commuting with vector fields in Sections 5 and 7.

We now outline the proof of Theorem~\ref{higherLE}. Theorem~\ref{LE} applied to $u_{\Lambda}$ gives for any multiindex $\Lambda$
\[\begin{split}
 \|\partial  u_{\Lambda}\|_{L^{\infty}[\tt_0, \tt_1] L^2} + \|u_{\Lambda}\|_{LE_K^1[\tt_0, \tt_1]} & \lesssim \int_{\M_{ps}{[\tt_0,\tt_1]}} \frac{\epsilon}{t}|\partial u_{\Lambda}|^2 dV_g + \|\partial u_{\Lambda}(\tt_0)\|_{L^2}^2 + B(F_{\Lambda},u_{\Lambda})_{[\tt_0, \tt_1]}
\end{split}\]
where
\[
F_{\Lambda} = [\Box_g, Z^{\Lambda}] u
\]

The first step is to only consider time derivatives. Here the argument is identical to the one in \cite{LT}, since $[\Box_{g_K} - \Box_{g_S}, \pa_t] = 0$. In particular, near the trapped set we control $\|F_{\Lambda}\|_{L^1[\tt_0,\tt_1]L^2}$, using \eqref{W-est} in the process. We obtain \eqref{higherLES} for $\pa_t^N u$:
\[
\|\partial  \partial_t^N u\|_{L^{\infty}[\tt_0, \tt_1] L^2} + \|\partial_t^N u\|_{LE_K^1[\tt_0, \tt_1]} \leq C_N \tt_1^{C_N\epsilon} \|\partial \partial_t^N u(\tt_0)\|_{L^2}
\]

We can now prove \eqref{higherLES} for spatial derivatives also by using elliptic estimates away from the event horizon, and the red-shift effect near the event horizon. The main estimate is the following, corresponding to Lemma 5.3 in \cite{LT}:
\begin{equation}\label{ell-est} \begin{split}
\|\partial^2 u_{\Lambda}(\tt_1)\|_{L^2} + \|\partial u_{\Lambda}\|_{LE_K^1[\tt_0, \tt_1]} & \lesssim \|\partial \partial_{\leq 1}u_{\Lambda}(\tt_0)\|_{L^2} + \|\partial_t \partial_{\leq 1} u_{\leq |\Lambda|}(\tt_1)\|_{L^2} + \|\partial u_{\leq |\Lambda|}(\tt_1)\|_{L^2} \\ & + \|\partial_t u_{\leq |\Lambda|}\|_{LE_K^1[\tt_0, \tt_1]} +  \|u_{\leq |\Lambda|}\|_{LE_K^1[\tt_0, \tt_1]}
\end{split} \end{equation}

The main difference in the proof of \eqref{ell-est} comes in proving the analogue of (5.11) from \cite{LT}, which in this case means proving the following:
\beq\label{spatderivout2}
\|\chi_{out}\partial u\|^2_{LE_K^1[\tt_0, \tt_1]} \lesssim \sum_{i=1}^2 \|\partial \partial_{\leq 1} u(\tt_i)\|^2_{L^2} + \|\partial_t u\|^2_{LE_K^1[\tt_0, \tt_1]} +  \|u\|^2_{LE_K^1[\tt_0, \tt_1]}  + \bigl\|\Box_g u\|^2_{{LE[\tt_0, \tt_1]}}
\eq

Here $\chi_{out}=1$ when $2M+2\varepsilon\leq r$, $\chi_{out}=0$ when $r<2M+\varepsilon$ for some $\varepsilon \ll M$.

We postpone the proof of \eqref{spatderivout2} to the end of the section.

Once \eqref{ell-est} is proved, one starts commuting with vector fields. The proof is identical to the one in section 5 in \cite{LT}, with the exception of the fact that we need to estimate $[\Box_{g_K} - \Box_{g_S}, Z^{\Lambda}] u$. One immediately sees that
\[
\Bigl| [\Box_{g_K} - \Box_{g_S}, Z]\Bigr| \lesssim \frac{a}{r^2} |\pa \pa^{\leq 1} u|, \quad Z\in\{\Omega, S\}
\]
and thus, if either $j>0$ or $k>0$, we have that the commutator is of lower order:
\[
\Bigl| [\Box_{g_K} - \Box_{g_S}, Z^{\Lambda}]\Bigr| \lesssim \frac{a}{r^2} |\pa u_{<|\Lambda|}|
\]

In particular, using the second term in the definition of $B(F,u)_{[\tt_0, \tt_1]}$, we have near the trapped set
\[
\|\partial^{\leq 1}\left(\tchi [\Box_{g_K} - \Box_{g_S}, Z^{\Lambda}]\right)\|_{L^2[\tt_0, \tt_1]L^2}^2 \lesssim a \int_{\M_{ps}{[\tt_0, \tt_1]}} |\pa u_{<|\Lambda|}|^2 dt dx \lesssim a \|u_{\leq N}\|^2_{LE_K^1[\tt_0, \tt_1]}
\]
and away from the trapped set
\[\begin{split}
& \int_{\M{[\tt_0, \tt_1]}} (1-\chi)\Bigl| [\Box_{g_K} - \Box_{g_S}, Z^{\Lambda}]\Bigr| (|\partial u_{\Lambda}| + r^{-1} |u_{\Lambda}|) dt dx \\ & \lesssim \int_{\M{[\tt_0, \tt_1]}} (1-\chi)\frac{a}{r^2} |\pa u_{<|\Lambda|}| (|\partial u_{\Lambda}| + r^{-1} |u_{\Lambda}|) dt dx \lesssim a \|u_{\leq N}\|^2_{LE_K^1[\tt_0, \tt_1]}
\end{split}\]
which suffices.

Let us now prove \eqref{spatderivout2}. The only difficulty is in dealing with the norms near the trapped set, as the rest is identical to the proof of (5.11) in \cite{LT}. Recall that near the trapped set
\[
\| u\|_{LE_K^1} \lesssim \|\varrho_1(D,x) \chi(D_t - \tau_2(D,x))\chi u\|_{L^2L^2}
+ \|\varrho_2(D,x)\chi (D_t - \tau_1(D,x))\chi u\|_{L^2L^2} + \|D_r u\|_{L^2} + \|u\|_{L^2L^2}
\]
We will bound the first term, as the rest follow similarly. We will show that
\[\begin{split}
\|\varrho_1(D,x) \chi(D_t - \tau_2(D,x))\chi D_x u\|_{L^2[\tt_0, \tt_1]L^2}^2 & \lesssim \sum_{i=0}^1 \|\partial \partial_{\leq 1} u(\tt_i)\|^2_{L^2} + \|\partial_t u\|^2_{LE_K^1[\tt_0, \tt_1]} \\ & +  \|\pa^{\leq 1}u\|^2_{L^2[\tt_0, \tt_1]L_{cpt}^2}  + \bigl\|\Box_g u\|^2_{L^2[\tt_0, \tt_1]L_{cpt}^2}
\end{split}\]

We write
\begin{equation}\label{ell}
\mathscr{L}_{0} u = \mathscr{L}_t u + \Box_g u
\end{equation}
where
\begin{equation}\label{ellt}
\mathscr{L}_t u = -\frac{1}{\sqrt{|g|}}\partial_t (\sqrt{|g|} g^{tj}\partial_{j} u) -\frac{1}{\sqrt{|g|}}\partial_{j} (\sqrt{|g|} g^{tj}\partial_t u) -\frac{1}{\sqrt{|g|}}\partial_t (\sqrt{|g|} g^{tt}\partial_{t} u)
\end{equation}
\begin{equation}\label{ell0}
\mathscr{L}_0 u = \frac{1}{\sqrt{|g|}}\partial_i (\sqrt{|g|} g^{ij}\partial_{j} u)
\end{equation}

Let $Q_1 = \chi (\varrho_1 (\tau-\tau_2))^w \chi$. We multiply \eqref{ell0} by $\overline{Q_1^2 u}$ and integrate with respect to $dV_g$. We compute
\[\begin{split}
& \int_{\M_{ps}{[\tt_0,\tt_1]}}\!\!\!\!\mathscr{L}_0 u\, \,Q_1^2 u dV_g = \int_{\M_{ps}{[\tt_0,\tt_1]}}\!\!\!\! \sqrt{|g|} g^{ij}\!D_{j} u\overline{D_i Q_1^2 u} dtdx
=\int_{\M_{ps}{[\tt_0,\tt_1]}} \!\!\!\!\sqrt{|g|} g^{ij}\!D_{j} u \big(\overline{Q_1^2 D_i u \!+\! [D_i, Q_1^2]u}\big) dtdx \\ & =
  \int_{\M_{ps}{[\tt_0,\tt_1]}} Q_1(\sqrt{|g|} g^{ij}D_{j} u)\cdot \overline{Q_1 D_i u} + \sqrt{|g|} g^{ij}D_{j} u\cdot \overline{[D_i, Q_1^2]u} dtdx + \left. BDR\right|_{t = \tt_0}^{t=\tt_1} \\ & =  \! \int_{\M_{ps}{[\tt_0,\tt_1]}}\!\!\!\! g^{ij}Q_1 D_{ \!j} u\, \overline{Q_1 D_i u} dV_{\!g} + \! \!\int_{\M_{ps}{[\tt_0,\tt_1]}}\!\!\!\! [Q_1,  \!\sqrt{ \!|g|} g^{ij}]D_{ \!j} u \cdot \overline{Q_1 D_i u} + \! \sqrt{ \!|g|} g^{ij} \!D_{ \!j} u \cdot \overline{[D_i, Q_1^2]u} dtdx +\! \left. \! BDR\right|_{t = \tt_0}^{t=\tt_1}
\end{split}\]
where the boundary terms are
\[
BDR(\tt_i) = \int_{t=\tt_i} g^{ij} \sqrt{|g|} D_j u \cdot \overline{\chi\varrho_1^w\chi Q_1 D_i u} dx
\]

By Cauchy Schwarz, we can bound
\[
\Bigl| BDR(\tt_i) \Bigr| \lesssim \|\partial \partial_{\leq 1} u(\tt_i)\|_{L^2}^2
\]

By the ellipticity of $g^{ij}$ we have that
\[
\int_{\M_{ps}{[\tt_0,\tt_1]}} g^{ij}(Q_1 D_{j} u) \overline{Q_1 D_i u}  dV_g \gtrsim \|Q_1 D_x u\|_{L^2L^2}^2
\]

On the other hand, since
\[
[D_i, Q_1^2] = 2[D_i, Q_1] Q_1 + [Q_1, [D_i, Q_1]]
\]
and
\[ [D_i, Q_1], \, [Q_1, [D_i, Q_1]]\in OPS^1 + D_t OPS^0 \]
we have that
\[
\| [D_i, Q_1^2] u\|_{L^2L^2} \lesssim \| D_{t,x} Q_1 u\|_{L^2L^2} + \|\pa^{\leq 1} u\|_{L^2L_{cpt}^2} \lesssim  \| Q_1 D_{t,x} u\|_{L^2L^2} + \|\pa^{\leq 1} u\|_{L^2L_{cpt}^2}
\]
and by Cauchy Schwarz
\[
\Bigl|\int_{\M_{ps}{[\tt_0,\tt_1]}} (\sqrt{|g|} g^{ij}D_{j} u) \cdot \overline{[D_i, Q_1^2]u} dtdx \Bigr| \leq \delta_0 \| Q_1 D_{t,x} u\|_{L^2L^2}^2 + C(\delta_0) \|\pa^{\leq 1}u\|_{L^2[\tt_0, t]L_{cpt}^2}^2
\]
Moreover, Corollary~\ref{quasell} implies that
\[
\| [Q_1, \sqrt{|g|} g^{ij}]D_j u\|_{L^2L^2}\lesssim \|\sqrt{|g|} g^{ij}\|_{C^{1,\delta}} (\|D_j u\|_{L^2L_{cpt}^2} + \|D_t u\|_{L^2L_{cpt}^2}) \lesssim \|\pa^{\leq 1} u\|_{L^2[\tt_0, t]L^2}
\]
and by Cauchy Schwarz
\[
\Bigl|\int_{\M_{ps}{[\tt_0,\tt_1]}} [Q_1, \sqrt{|g|} g^{ij}]D_j u \cdot \overline{Q_1 D_i u} dtdx \Bigr| \leq \delta_0 \| Q_1 D_{t,x} u\|_{L^2L^2}^2 + C(\delta_0) \|\pa^{\leq 1}u\|_{L^2[\tt_0, t]L_{cpt}^2}
\]

For small enough $\delta_0$ we thus obtain
\[
\|Q_1 D_x u\|_{L^2L^2}^2 \lesssim \int_{\M_{ps}{[\tt_0,\tt_1]}} (\mathscr{L}_0 u) (Q_1^2 u) dV_g + RHS(\eqref{spatderivout2})
\]

On the other hand, a similar computation yields
\[
\Bigl|\int_{\M_{ps}{[\tt_0,\tt_1]}} (\mathscr{L}_t u) (Q_1^2 u) dV_g\Bigr| \lesssim RHS(\eqref{spatderivout2})
\]

Since
\[
\|Q_1 D_x u\|_{L^2L^2}^2 + \int_{\M_{ps}{[\tt_0,\tt_1]}} |\pa^{\leq 1} u|^2 dV_g \gtrsim \|\varrho_1(D,x) \chi(D_t - \tau_2(D,x))\chi D_x u\|_{L^2L^2}^2
\]
the conclusion \eqref{spatderivout2} follows.

\newsection{Price's law for perturbations of Kerr}\label{sec:ptwse}

As a second application of our estimate, we will prove pointwise decay for solutions to the linear wave equation on perturbations of the Kerr metric.

The main result of the section is the following:

\begin{theorem}\label{Price}
   Let $g$ be a Lorentzian metric close to $g_{\mathbf K}$ ,  in the sense that $h=g-g_{\mathbf K}$ satisfies, for large enough $N$ and small $\epsilon$
 \beq
 |h^{\mu\nu}_{\leq N}| \lesssim \epsilon r^{-2}
 \eq

In addition we assume that the metric decays to
  $g_{\mathbf K}$ near the photon sphere,
  \begin{equation}\label{aproxmetcpt}
     |\partial^{\leq 1} h^{\mu\nu}|\lesssim \kappa_1(\tt),
\qquad  |\partial_t h^{\mu\nu}|\lesssim \kappa_0(\tt), \quad \text{when} \, |r-3M|<\frac{M}4
 \end{equation}
where $\int_{0}^{\infty} \kappa_1^2(\tt) + \kappa_0(\tt) d\tt < \infty$. In particular, we can assume that
\[
\kappa_1(\tt) = \tt^{-1/2-\delta}, \kappa_0(\tt) = \tt^{-1-\delta}
\]

  Let $u$ solve
  \beq\label{box}
  \Box_g u = 0, \quad u(0) = u_0, \quad \pa_t u(0) = u_1
  \eq
  with smooth, compactly supported initial data. Then we have:
\begin{equation}
| u(t,x)| \lesssim  \frac{1}{\la t\ra \la t-\rs\ra^{2}}, \qquad | \nabla u(t,x)|
\lesssim \frac{1}{\la r \ra \la t-\rs\ra^{3}}
\label{pointwiseest1}
\end{equation}
\end{theorem}

The improvement over Theorem 4.6 in \cite{MTT} comes from the fact that we don't require integrability of $h$ and $\pa_{r,\omega} h$ near the trapped set.

The proof of Theorem~\ref{Price} is identical to the proof of Theorem 4.6 in \cite{MTT}. Clearly Theorem~\ref{LE} implies that
 \begin{equation}\label{lee}
    \|u\|_{H^1(\Sigma_R^+)}^2 +\sup_{\tilde t \geq 0}
    \|\nabla u(\tilde t)\|_{L^2}^2 +  \|u\|_{LE^1_{\mathbf K}}^2 \lesssim \|\nabla u(0)\|_{L^2}^2
 +    \|f\|_{LE^*_{\mathbf K}}^2.
  \end{equation}
  which is the same as (4.23) from \cite{MTT}. Commuting with derivatives as in the proof of Theorem 4.5 in \cite{MTT} yields
  \begin{equation}\label{leek}
   \|u\|_{H^k(\Sigma_R^+)}^2 +\sup_{\tilde t \geq 0}
    \|\nabla u(\tilde t)\|_{H^k}^2 +  \|u\|_{LE^{1,k}_{\mathbf K}}^2 \lesssim \|\nabla u(0)\|_{H^k}^2
 +    \|f\|_{LE^{*,k}_{\mathbf K}}^2,
  \end{equation}
and so Theorem 4.5 in \cite{MTT} holds in this context.

\section*{Acknowledgments}
H.L. was supported in part by NSF grant DMS-1500925 and Simons Collaboration Grant 638955. M.T. was supported in part by the NSF grant DMS--1636435 and Simons collaboration Grant 586051. We would also like to thank the Mittag Leffler Institute for their hospitality during the Fall 2019 program in Geometry and Relativity.

\bigskip

\end{document}